\documentclass[11pt,letterpaper]{amsart}
\usepackage[leqno]{amsmath}
\usepackage{amsmath,amsfonts,amsthm,amssymb}
\usepackage{epsfig}
\usepackage{graphicx}
\usepackage{subfigure}
\usepackage{cite,url}
\usepackage{hyperref}
\usepackage[left=0.8in,right=0.8in,top=0.8in,bottom=0.8in,foot=0.5in]{geometry}

\usepackage{epsfig}
\usepackage{color}
\usepackage{amsmath}
\usepackage{amssymb}
\newtheorem{lemma}{Lemma}
\newtheorem{proposition}{Proposition}
\newtheorem{theorem}{Theorem}
\newtheorem{corollary}{Corollary}
\newtheorem{claim}{Claim}

\newtheorem{mtheorem}{Theorem}

\theoremstyle{definition}

\newtheorem{problem}{Problem}

\newtheorem{conjecture}{Conjecture}

\newcommand{\conv}{\ensuremath{{\rm conv}}}
\newcommand{\St}{\ensuremath{{\rm St}}}

\begin{document}

\thispagestyle{empty}

\centerline{\Large\bf Hypercellular graphs: partial cubes without $Q_3^-$ as partial cube minor}

\vspace{8mm}
\centerline{Victor Chepoi$^{\small 1}$, Kolja Knauer$^{\small 1,2}$, and Tilen Marc$^{\small 3}$}

\medskip
\begin{small}
\medskip
\centerline{$^{1}$Laboratoire d'Informatique et Syst\`emes, Aix-Marseille Universit\'e and CNRS,}
\centerline{Facult\'e des Sciences de Luminy, F-13288 Marseille Cedex 9, France}

\centerline{\texttt{\{victor.chepoi, kolja.knauer\}@lis-lab.fr}}

\medskip
\centerline{$^{2}$
Departament de Matem\`atiques i Inform\`atica, Universitat de Barcelona (UB),}
\centerline{Barcelona, Spain}

\medskip
\centerline{$^{3}$ Faculty of Mathematics and Physics, University of Ljubljana}
\centerline{and}
\centerline{Institute of Mathematics, Physics, and Mechanics, Ljubljana, Slovenia}
\centerline{\texttt{tilen.marc@imfm.si}}

\end{small}

 \bigskip\bigskip\noindent
{\footnotesize {\bf Abstract.} We investigate the structure of isometric subgraphs of hypercubes
(i.e., partial cubes) which do not contain finite convex subgraphs
contractible to the 3-cube minus one vertex $Q^-_3$ (here contraction
means contracting the edges corresponding to the same coordinate
of the hypercube). Extending similar results for median and cellular
graphs, we show that the convex hull of an isometric cycle of such a
graph is gated and isomorphic to the Cartesian product of edges and
even cycles. Furthermore, we show that our graphs are exactly the class
of
partial cubes in which any finite convex subgraph can be obtained from
the
Cartesian products of edges and even cycles via
successive gated amalgams. This decomposition result enables us to
establish a variety
of results. In particular, it yields that our class of graphs
generalizes median and cellular
graphs, which motivates naming our graphs hypercellular. Furthermore, we
show that hypercellular graphs are tope graphs of zonotopal complexes of
oriented matroids.
Finally, we characterize hypercellular graphs as being median-cell -- a
property naturally generalizing the notion of median graphs.
}

\section{Introduction}\label{sec:intro}
{\it Partial cubes} are the graphs which admit an isometric embedding into a hypercube.  They comprise many important and
complex graph classes occurring in metric graph theory and initially arising in completely different areas of research.
Among them there are the graphs of regions of hyperplane arrangements in ${\mathbb R}^d$~\cite{bjedzi-90},
and, more generally, tope graphs of oriented matroids (OMs)~\cite{BjLVStWhZi}, median graphs (alias 1-skeleta of CAT(0) cube
complexes)~\cite{BaCh_survey,ImKl}, netlike graphs~\cite{Po1,Po2,Po3,Po4}, bipartite cellular graphs~\cite{BaCh_cellular},
bipartite graphs with $S_4$ convexity~\cite{Ch_separation},
graphs of lopsided sets~\cite{BaChDrKo,La},  1-skeleta of CAT(0) Coxeter zonotopal complexes~\cite{HaPa}, and tope graphs of
complexes of oriented matroids (COMs)~\cite{BaChKn}. COMs represent a general unifying structure for many of the above  classes of partial cubes: from tope graphs of OMs to median graphs, lopsided sets,  cellular graphs,
and graphs of CAT(0) Coxeter zonotopal complexes. 
Median graphs
are obtained by gluing in a specific way cubes of different dimensions. In particular, they give rise not only to contractible but
also to CAT(0) cube complexes. Similarly, lopsided sets yield contractible cube complexes, while cellular graphs give contractible polygonal complexes whose cells are regular even polygons. Analogously to median graphs, graphs of CAT(0) Coxeter zonotopal
complexes can  be viewed as partial cubes obtained by gluing zonotopes. COMs can be viewed as a common generalization of all these
notions: their tope graphs are the partial cubes obtained by gluing tope graphs of OMs 
in a lopsided (and thus contractible) fashion.

In this paper, we investigate the structure of a subclass of zonotopal COMs, in which all cells are gated subgraphs isomorphic to Cartesian products of edges
and even cycles, see Figure~\ref{fig:C6xC6} for such a cell.  More precisely, we study the partial cubes in which all finite convex subgraphs
can be obtained from Cartesian products of edges and even cycles by successive gated amalgamations. We show that our graphs share
and extend many properties of bipartite cellular graphs of~\cite{BaCh_cellular}; they can be viewed as high-dimensional analogs of
cellular graphs. This is why we call them {\it hypercellular graphs}, see Figure~\ref{fig:exmpl} for an example. There is another way of
describing hypercellular graphs, requiring a few definitions.

\begin{figure}[ht]
\centering
\subfigure[\label{fig:C6xC6}]{\includegraphics[width=.33\textwidth]{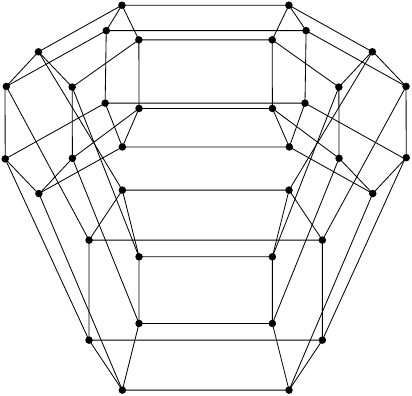}}
\subfigure[\label{fig:exmpl}]{\includegraphics[width=.55\textwidth]{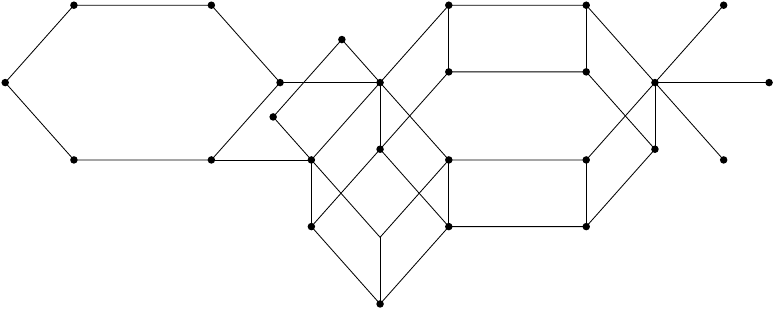}}
\caption{\subref{fig:C6xC6} a four-dimensional cell isomorphic to $C_6\square C_6$. \subref{fig:exmpl}~a hypercellular graph with eight maximal cells: $C_6$, $C_4$, $C_4$, $K_2\square K_2\square K_2$, $C_6\square K_2$, and three $K_2$.}
\end{figure}

Djokovi\'{c}~\cite{Djokovic}
characterized partial cubes in the following simple but pretty way:
{\it a graph $G=(V,E)$ can be isometrically embedded in a hypercube if and only if  $G$ is bipartite and for any edge
$uv$, the sets $W(u,v)=\{ x\in V: d(x,u)<d(x,v)\}$ and $W(v,u)=\{ x\in V: d(x,v)<d(x,u)\}$
are convex.} In this case, $W(u,v)\cup W(v,u)=V$, whence $W(u,v)$ and $W(v,u)$ are complementary convex subsets
of $G$, called {\it halfspaces}. The edges between $W(u,v)$ and $W(v,u)$ correspond to a coordinate in a hypercube embedding of $G$.

Moreover, partial cubes have the separation property $S_3$: any convex subgraph $G'$ of a partial cube $G$ can be represented as an intersection of halfspaces of $G$~\cite{AlKn,Ba_S3,Ch_thesis}. We will call such a representation (or simply the convex subgraph $G'$) a {\it restriction} of $G$.
A \emph{contraction} of $G$ is the partial cube $G'$ obtained from $G$ by contracting all edges
corresponding to a given coordinate in a hypercube embedding. Now, a partial cube $H$ is called a {\it partial cube-minor} (abbreviated, {\it pc-minor}) of $G$ if
$H$ can be obtained by a sequence of contractions from a convex subgraph of $G$.
If $T_1,\ldots,T_m$ are finite
partial cubes, then ${\mathcal F}(T_1,\ldots,T_m)$ is the set of all partial cubes $G$ such that {\it no} $T_i, i=1,\ldots,m,$ can be obtained as a pc-minor
of $G$. We will say that a class of partial cubes ${\mathcal C}$ is
{\it pc-minor-closed} if we have that $G\in {\mathcal C}$ and $G'$ is a pc-minor of $G$ imply that $G'\in {\mathcal C}$. As we will see in Section~\ref{minors}, for any set of partial cubes $T_1,\ldots,T_m$, the class ${\mathcal F}(T_1,\ldots,T_m)$ is pc-minor-closed.

\begin{figure}[ht]
\centering
\subfigure[\label{fig:Q3minus}]{\includegraphics[width=.15\textwidth]{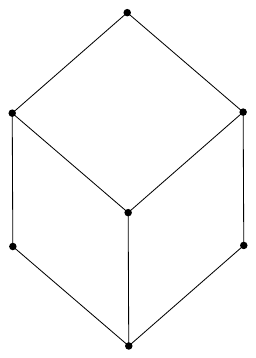}}
\hspace{5em}%
\subfigure[\label{fig:3cubecond}]{\includegraphics[width=.4\textwidth]{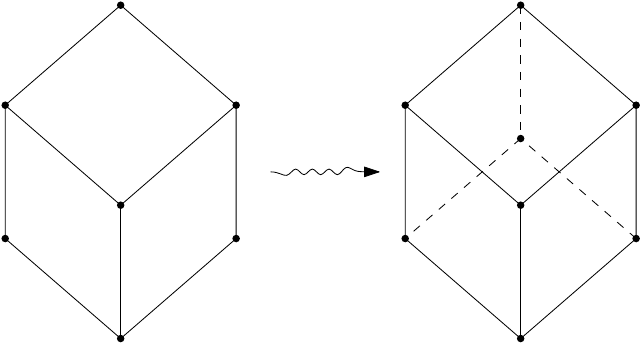}}
\caption{\subref{fig:Q3minus} $Q^-_3$ -- the 3-cube minus one vertex.\subref{fig:3cubecond}~the 3-cube condition.}
\end{figure}

It turns out that the class of hypercellular graphs coincides
with the minor-closed class ${\mathcal F}(Q^-_3)$, where $Q^-_3$ denotes the 3-cube minus one vertex, see Figure~\ref{fig:Q3minus}. In a sense, this is the first nontrivial class ${\mathcal F}(T)$. Indeed, the class
${\mathcal F}(C)$, where $C$ is a $4$-cycle, is just the class of all trees. Also, the classes ${\mathcal F}(T)$ where $T$ is the union of two 4-cycles sharing one vertex or one edge are quite special. Median graphs, graphs of lopsided sets, and tope graphs of COMs are pc-minor closed, whereas tope graphs of OMs are only closed under contractions but not under restrictions.
Another class of pc-minor closed partial cubes is the class ${\mathcal S}_4$ also known as \emph{Pasch graphs}. It consists of bipartite graphs in which the geodesic convexity satisfies
the separation property $S_4$~\cite{Ch_thesis,Ch_separation}, i.e., any two disjoint convex sets can be separated by disjoint half-spaces. It is shown in~\cite{Ch_thesis,Ch_separation} that ${\mathcal S}_4={\mathcal F}(T_1,\ldots,T_m)$,
where all $T_i$ are isometric subgraphs of $Q_4$; see Figure~\ref{fig:forbidden3} for the complete list, from which $T_5$ and $T_7$ were missing in~\cite{Ch_thesis,Ch_separation}.  In particular, $Q^-_3$ is a pc-minor of all of the $T_i$. Thus, ${\mathcal F}(Q^-_3)\subseteq{\mathcal S}_4$.

\begin{figure}[ht]
\begin{center}
\includegraphics[width = .9\textwidth]{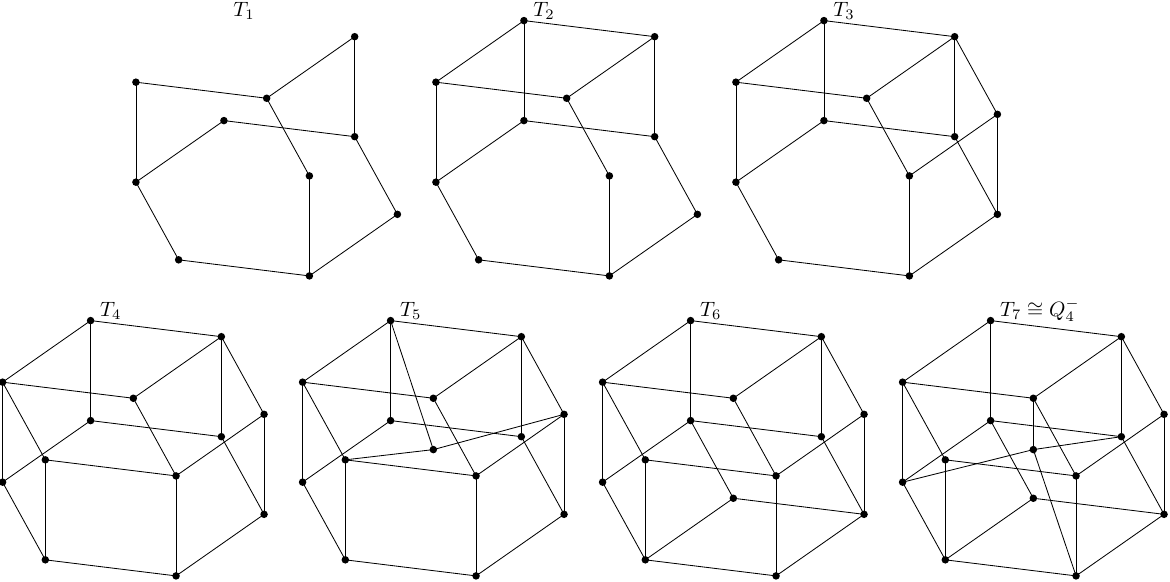}
\caption{The set of minimal forbidden pc-minors of ${\mathcal S}_4$.}
\label{fig:forbidden3}
\end{center}
\end{figure}

Our results mainly concern the cell-structure  of graphs from ${\mathcal F}(Q^-_3)$. It is well-known~\cite{Ba_median} that median graphs
are exactly the graphs in which the convex hulls of isometric cycles are hypercubes; these hypercubes are gated subgraphs. Moreover, any finite median graph can be obtained by
gated amalgams from cubes~\cite{Is,VdV1}. Analogously, it was shown in
\cite{BaCh_cellular} that any isometric cycle of a bipartite cellular graph is a convex and gated subgraph; moreover, the bipartite cellular graphs are exactly the
bipartite graphs  which can be obtained by gated amalgams from even cycles. We extend these results in the following way:

\begin{mtheorem}\label{mthm:cells}
 The convex closure of any isometric cycle of a graph $G\in\mathcal{F}(Q_3^-)$ is a gated subgraph isomorphic to a Cartesian product of edges and even cycles. Moreover, the convex closure of any isometric cycle of a graph $G\in\mathcal{S}_4$ is a gated subgraph, which is isomorphic to a Cartesian product of edges and even cycles if it is antipodal.
\end{mtheorem}

\medskip
In view of Theorem~\ref{mthm:cells} we will call a subgraph $X$ of a partial cube $G$ a {\it cell} if $X$ is a convex subgraph of $G$ which is a Cartesian product of edges and
even cycles. Note that since a Cartesian product of edges and even cycles is the convex hull of an isometric cycle, by Theorem~\ref{mthm:cells} the cells of  $\mathcal{F}(Q_3^-)$
can be equivalently defined as convex hulls of isometric cycles. Notice also that  if we replace each cell $X$ of $G$ by
a convex polyhedron $[X]$ which is the Cartesian product of segments and regular polygons (a segment for each edge-factor and a regular polygon
for each cyclic factor), then we associate with $G$ a cell complex ${\bf X}(G)$.

We will say that a partial cube $G$ satisfies the {\it 3-convex cycles condition} (abbreviated, {\it 3CC-condition}) if for any three convex cycles $C_1,C_2,C_3$
that intersect in a vertex and pairwise intersect in three different edges the convex hull of $C_1\cup C_2\cup C_3$ is a cell; see Figure~\ref{fig:3CC} for an example.
Notice that the absence of cycles satisfying the preconditions of the 3CC-condition together with the gatedness of isometric cycles characterizes bipartite cellular graphs~\cite{BaCh_cellular}.

\begin{figure}[ht]
\begin{center}
\includegraphics[width = .45\textwidth]{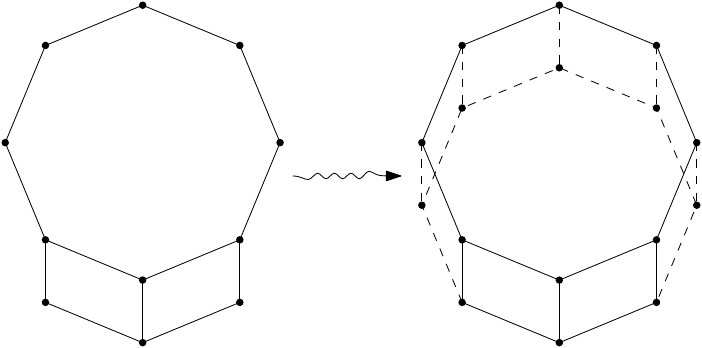}
\caption{The 3-convex cycles condition.}
\label{fig:3CC}
\end{center}
\end{figure}

Defining the dimension of a cell $X$ as the number of edge-factors plus two times the number cyclic factors (which corresponds to the topological dimension of $[X]$) one can give a natural generalization of the 3CC-condition. We say that a partial cube $G$ (or its cell complex ${\bf X}(G)$) satisfies the {\it 3-cell condition} (abbreviated, {\it 3C-condition}) if for any three cells $X_1,X_2,X_3$ of dimension $k+2$
that intersect in a cell of dimension $k$ and pairwise intersect in three different cells of dimension $k+1$ the convex hull of $X_1\cup X_2\cup X_3$ is a cell. In case of cubical complexes
${\bf X}$, the 3-cell condition coincides with Gromov's flag condition~\cite{Gr} (which can be also called cube condition, see Figure~\ref{fig:3cubecond}), which together with simply connectivity of ${\bf X}$ characterize CAT(0) cube complexes. By ~\cite[Theorem 6.1]{Ch_CAT}, median graphs are exactly the 1-skeleta of CAT(0) cube complexes (for other generalizations of these two results, see~\cite{BrChChGoOs,ChChHiOs}).

The following main characterization of graphs from $\mathcal{F}(Q_3^-)$ establishes those analogies with median and cellular graphs, that lead to the name hypercellular graphs.

\begin{mtheorem}\label{mthm:amalgam}
 For a partial cube $G=(V,E)$, the following conditions are equivalent:
\begin{itemize}
\item[(i)] $G\in\mathcal{F}(Q_3^-)$, i.e., $G$ is hypercellular;
\item[(ii)] any cell of $G$ is gated and $G$ satisfies the 3CC-condition;
\item[(iii)] any cell of $G$ is gated and $G$ satisfies the 3C-condition;
\item[(iv)] each finite convex subgraph of $G$ can be obtained by gated amalgams
from cells.
\end{itemize}
\end{mtheorem}

\begin{figure}[ht]
\centering
\subfigure[\label{fig:medianvertex}]{\includegraphics[width=.26\textwidth]{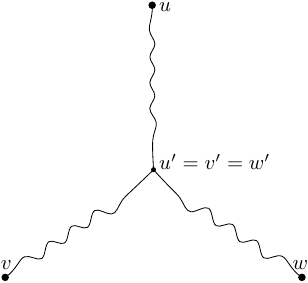}}
\hspace{.6em}
\subfigure[\label{fig:mediancycle}]{\includegraphics[width=.26\textwidth]{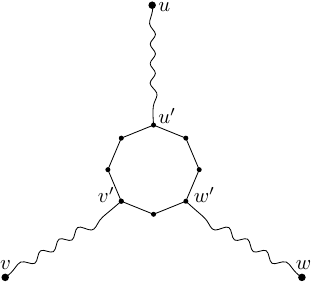}}
\hspace{.6em}
\subfigure[\label{fig:mediancell}]{\includegraphics[width=.26\textwidth]{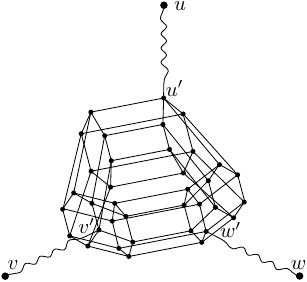}}
\caption{\subref{fig:medianvertex} a median-vertex. \subref{fig:mediancycle} a median-cycle. \subref{fig:mediancell} a median-cell.}
\end{figure}

A further characterization of hypercellular graphs is analogous to median and cellular graphs, see the corresponding properties in Figure~\ref{fig:medianvertex} and~\ref{fig:mediancycle}, respectively. We show that hypercellular graphs satisfy the so-called
{\it median-cell property}, which is essentially defined as follows: for any three vertices $u,v,w$ of $G$ there exists a unique gated cell $X$ of $G$ such that if $u',v',w'$ are the gates of $u,v,w$ in $X$, respectively, then
$u',v'$ lie on a common $(u,v)$-geodesic, $v',w'$ lie on a common $(v,w)$-geodesic, and $w',u'$ lie on a common $(w,u)$-geodesic, see Figure~\ref{fig:mediancell} for an illustration. Namely, we prove:

\begin{mtheorem}\label{mthm:median}
 A partial cube $G$ satisfies the median-cell property if and only if $G$ is hypercellular.
\end{mtheorem}

\medskip
Theorem~\ref{mthm:amalgam} has several immediate consequences, which
we formulate next.

\begin{mtheorem}\label{mthm:COM}
Let $G$ be a locally finite hypercellular graph. Then ${\bf X}(G)$ is a contractible zonotopal complex. Additionally, if $G$ is finite,
then $G$ is a tope  graph of a zonotopal COM.
\end{mtheorem}

Theorem~\ref{mthm:amalgam} also immediately implies that median graphs and bipartite cellular graphs are hypercellular. Furthermore, a subclass of netlike partial cubes, namely partial cubes which are gated amalgams of even cycles and cubes~\cite{Po3}, are hypercellular.
In particular, we obtain that these three classes coincide with $\mathcal{F}(Q_3^-, C_6)$, $\mathcal{F}(Q_3^-, Q_3)$, and $\mathcal{F}(Q_3^-, C_6\square K_2)$, respectively.
Other direct consequences of Theorem~\ref{mthm:amalgam} concern
convexity invariants (Helly, Caratheodory, Radon, and partition numbers) of hypercellular graphs which are shown to be either a constant or bounded by the
topological dimension of ${\bf X}(G)$.

Let $G$ be a hypercellular graph. For an equivalence class $E_f$ of edges of $G$ (i.e., all edges
corresponding to a given coordinate $f$ in a hypercube embedding of $G$), we denote by $N(E_f)$ the \emph{carrier} of $f$, i.e., subgraph of $G$
which is the union of all cells of $G$ crossed by $E_f$. It was shown in~\cite[Proposition 5]{BaChKn} that carriers of COMs are also COMs.
A \emph{star} $\St(v)$ of a vertex $v$ (or a star $\St(X)$ of a cell $X$) is
the union of all cells of $G$ containing $v$ (respectively, $X$). The {\it thickening} $G^{\Delta}$ of $G$ is a graph having the same set of vertices as $G$ and two vertices
$u,v$ are adjacent in $G^{\Delta}$ if and only if $u$ and $v$ belong to a common cell of $G$. Finally, a graph $H$ is called a {\it Helly graph} if any collection
of pairwise intersecting balls has a nonempty intersection. Helly graphs play an important role in metric graph theory as discrete analogs of injective spaces: any graph embeds
isometrically into a smallest Helly graph (for this and other results, see the survey~\cite{BaCh_survey} and the recent paper~\cite{ChChHiOs}). It was shown in~\cite{BavdV} that
the thickening of median graphs are finitely Helly graphs (for a generalization of this result, see~\cite[Theorem 6.13]{ChChHiOs}).

\begin{mtheorem}\label{mthm:stars}
 Let $G$ be a hypercellular graph. Then all carriers $N(E_f)$  and stars $\St(X)$ of $G$ are gated. If additionally  $G$ is locally-finite,
then the thickening $G^{\Delta}$ of $G$ is a  Helly graph.
\end{mtheorem}

Finally, we generalize fixed box theorems for median graphs to hypercellular graphs and prove that in this case the fixed box is a cell. More precisely, we conclude the paper with the following:

\begin{mtheorem}\label{mthm:fixed}
Let $G$ be a hypercellular graph.
\begin{itemize}
\item[(i)] if $G$ does not contain infinite isometric rays, then $G$ contains a cell $X$ fixed by every automorphism of $G$;
\item[(ii)] any non-expansive map $f$ from $G$ to itself fixing a finite set of vertices (i.e., $f(S)=S$ for a finite set $S$)
also fixes a finite cell $X$ of $G$. In particular, if $G$ is finite, then any non-expansive map $f$ from $G$ to itself fixes a cell of $G$;
\item[(iii)] if $G$ is finite and regular, then $G$ is a single cell, i.e., $G$ is isomorphic to a Cartesian product of edges
and cycles.
\end{itemize}
\end{mtheorem}

\subsection*{Structure of the paper:} In Section~\ref{sec:prel} we introduce preliminary definitions and results needed for this paper. In particular,
we discuss convex and gated subgraphs in partial cubes, the notion of partial cube minors and their relation with convexity and gatedness. We also
briefly discuss the properties of Cartesian products central to our work. Section~\ref{sec:cells} is devoted to the structure of cells in hypercellular
graphs and graphs from $\mathcal{S}_4$; in particular, we prove Theorem~\ref{mthm:cells}. Section~\ref{sec:amalgam} is devoted to amalgamation and
decomposition of hypercellular graphs; we prove Theorem~\ref{mthm:amalgam}. In Section~\ref{sec:median} we discuss the median cell property
of hypercellular graphs and prove Theorem~\ref{mthm:median}. Section~\ref{sec:properties} provides a rich set of properties of hypercellular graphs.
In Subsection~\ref{subsec:otherclasses} we expose relations to other classes of partial cubes and in particular prove Theorem~\ref{mthm:COM}.
Subsection~\ref{subsec:convex} gives several properties with respect to convexity parameters. Subsection~\ref{subsec:stars} is devoted to the
proof of Theorem~\ref{mthm:stars}. In Subsection~\ref{subsec:fixedcells} we prove several fixed cell results for hypercellular graphs,
in particular, we prove Theorem~\ref{mthm:fixed}.  We conclude the paper with several problems and conjectures in Section~\ref{sec:concl}.

\section{Preliminaries}\label{sec:prel}
\subsection{Metric subgraphs and partial cubes}
All graphs $G=(V,E)$ occurring in this paper  are simple, connected,
without loops or multiple edges, but not necessarily finite.
The {\it distance}
$d(u,v):=d_G(u,v)$ between two vertices $u$ and $v$ is the length
of a shortest $(u,v)$-path, and the {\it interval} $I(u,v)$
between $u$ and $v$ consists of all vertices on shortest
$(u,v)$--paths, that is, of all vertices (metrically) {\it
between} $u$ and $v$:
$$I(u,v):=\{ x\in V: d(u,x)+d(x,v)=d(u,v)\}.$$
An induced subgraph of $G$ (or the corresponding vertex set $A$)
is called {\it convex} if it includes the interval of $G$ between
any two of its vertices.  Since the intersection of convex subgraphs
is convex, for every subset $S\subseteq V$ there exists the
smallest convex set $\conv(S)$ containing
$S$, referred to as the {\it convex hull} of $S$.
An induced subgraph $H$ of $G$ is {\it
isometric} if the distance between any pair of vertices in $H$ is
the same as that in $G.$  In particular, convex subgraphs are
isometric.

A subset $W$ of $V$ or the
subgraph $H$ of $G$ induced by $W$ is called {\it gated} (in $G$)~\cite{DrSch}
if for every vertex $x$ outside $H$ there exists a vertex $x'$
(the {\it gate} of $x$) in $H$ such that each vertex $y$ of $H$ is
connected with $x$ by a shortest path passing through the gate
$x'$. It is easy to see that if $x$ has a gate in $H$, then it is unique
and that gated sets are convex. Gated sets enjoy the finite {\it Helly property}
\cite[Proposition 5.12 (2)]{VdV},  that
is, every finite family of gated sets that pairwise intersect has
a nonempty intersection. Since the intersection of gated subgraphs
is gated, for every subset $S\subseteq V$ there exists the
smallest gated set $\langle\langle S\rangle\rangle$ containing
$S,$ referred to as the {\it gated hull} of $S$. A graph $G$ is a
{\it gated amalgam} of two graphs $G_1$ and $G_2$ if $G_1$ and
$G_2$ constitute  two intersecting gated subgraphs of $G$ whose
union is all of $G.$

A graph $G=(V,E)$ is
{\it isometrically embeddable} into a graph $H=(W,F)$ if there
exists a mapping $\varphi : V\rightarrow W$ such that $d_H(\varphi
(u),\varphi (v))=d_G(u,v)$ for all vertices $u,v\in V$,
 i.e., $\varphi(G)$ is  an isometric subgraph of $H$. A graph $G$ is called a {\it partial cube} if it admits an isometric embedding into some hypercube
$Q(\Lambda)=\{-1,+1\}^{\Lambda}$.
From now on, we will always suppose that a partial cube $G=(V,E)$ is an isometric subgraph of the hypercube
$Q(\Lambda)=\{-1,+1\}^{\Lambda}$ (i.e., we will identify $G$ with its image under the isometric embedding). If this causes no confusion, we will
denote the distance function of $G$ by $d$ and not $d_G$.

For an edge $e=uv$ of $G$, define the sets $W(u,v)=\{ x\in V: d(x,u)<d(x,v)\}$ and $W(v,u)=\{ x\in V: d(x,v)<d(x,u)\}$. By Djokovi\'c's theorem~\cite{Djokovic}, a graph $G$ is a partial cube if and only if $G$ is bipartite and for any edge $e=uv$
the sets $W(u,v)$ and $W(v,u)$ are convex. The sets of the form $W(u,v)$ and $W(v,u)$ are called {\it complementary halfspaces} of $G$.
To establish an isometric embedding of $G$ into a hypercube,  Djokovi\'{c}~\cite{Djokovic}
introduces the following binary relation $\Theta$ -- called \emph{Djokovi\'{c}-Winkler relation} -- on the edges of $G$:  for two edges $e=uv$ and $e'=u'v'$ we set $e\Theta e'$ if and only if
$u'\in W(u,v)$ and $v'\in W(v,u)$. Under the conditions of the theorem, it can be shown that $e\Theta e'$ if and only if
$W(u,v)=W(u',v')$ and $W(v,u)=W(v',u')$, whence $\Theta$ is an equivalence relation. Let ${\mathcal E}=\{ E_i: i\in \Lambda\}$ be the equivalence classes
of $\Theta$ and let $b$ be an arbitrary fixed vertex taken as the base point of $G$. For an equivalence class $E_i\in {\mathcal E}$,
let $\{ H^-_i,H^+_i\}$ be the pair of complementary
convex halfspaces  of $G$ defined by setting $H^-_i:=W(u,v)$ and $H^+_i:=W(v,u)$ for an arbitrary edge $uv\in E_i$ with $b\in W(u,v)$.

Two subfamilies of partial cubes are particularly important for our work. \emph{Bipartite cellular graphs} are by one of their  characterizations provided in~\cite{BaCh_cellular} the bipartite graphs in which all isometric cycles are gated. Equivalently, they are precisely the graphs obtained from even cycles by gated amalgamation. \emph{Median graphs} are graphs in which for every three vertices $u,v,w$ there exists a unique \emph{median-vertex} $x$ that simultaneously lies on $(u,v)$-, $(u,w)$-, and $(v,w)$-geodesic. They are precisely the graphs obtained by gated amalgamation of hypercubes~\cite{Ch_CAT}, or equivalently graphs in which the convex closure of every isometric cycle is isomorphic to a hypercube~\cite{Ba_median}.

\subsection{Partial cube minors}\label{minors}

Let $G=(V,E)$ be an isometric subgraph of the hypercube $Q(\Lambda)=\{ -1,+1\}^{\Lambda}$.
Given $f\in \Lambda$,  an {\it elementary restriction} consists in taking one of the subgraphs $G(H^-_f)$  or $G(H^+_f)$ induced by the
complementary halfspaces $H^-_f$ and $H^+_f$, which we will denote by $\rho_{f^-}(G)$ and $\rho_{f^+}(G)$, respectively. These graphs are isometric subgraphs of the hypercube $Q(\Lambda\setminus \{ f\}).$
Now applying twice the elementary restriction to two different coordinates $f,g$, independently of the order of $f$ and $g$,
we will obtain one of the four (possibly empty) subgraphs induced by the $H^-_f\cap H^-_g,H^-_f\cap H^+_g,H^+_f\cap H^-_g,$ and $H^+_f\cap H^+_g$. Since the intersection of convex subsets is convex, each of these
four sets is convex in $G$ and consequently induces an isometric subgraph of the hypercube $Q(\Lambda\setminus\{ f,g\})$. More generally, a {\it restriction} is a subgraph of $G$ induced by the
intersection of a set of (non-complementary) halfspaces of $G$. We denote a restriction by $\rho_{A}(G)$, where $A\in\Lambda^{\{+,-\}}$ is a signed set of halfspaces of $G$. For subset $S$ of the vertices of $G$, we denote $\rho_{A}(S):=\rho_{A}(G)\cap S$. The following is well-known:

\begin{lemma}[\hspace*{-0.15cm}\cite{AlKn,Ba_S3,Ch_thesis}] \label{restriction}
 The set of restrictions of a partial cube $G$ coincides with its set of convex subgraphs. In particular, the class of partial cubes is closed under taking restrictions.
\end{lemma}
%
%

For $f\in \Lambda$, we say that the graph  $G/E_f$ obtained from
$G$ by contracting the edges of the equivalence class $E_f$ is an ($f$-){\it contraction} of $G$. For a vertex $v$ of $G$, we will denote by $\pi_f(v)$ the image of $v$ under the $f$-contraction in $G/E_f$, i.e., if $uv$ is an
edge of $E_f$, then $\pi_f(u)=\pi_f(v)$, otherwise $\pi_f(u)\ne \pi_f(v)$. We will apply $\pi_f$ to subsets $S\subset V$, by setting $\pi_f(S):=\{\pi_f(v): v\in S\}$. In particular we denote the $f$-{\it contraction} of $G$ by $\pi_f(G)$.

It is well-known and easy to prove and in particular follows from the proof of the first part of ~\cite[Theorem 3]{Ch_hamming} that $\pi_f(G)$ is an isometric subgraph of $Q(\Lambda\setminus \{ f\})$. Since  edge contractions in graphs commute, i.e., the resulting graph does not depend on the order in which a set of edges is contracted, we have:

\begin{lemma}\label{commut_contraction}
Contractions commute in partial cubes, i.e., if $f,g\in \Lambda$ and $f\ne g$, then $\pi_g(\pi_f(G))=\pi_f(\pi_g(G))$. Moreover, the class of partial cubes is closed under contractions.
\end{lemma}
%

Consequently, for a set $A\subset \Lambda$, we can denote by $\pi_A(G)$ the isometric subgraph of $Q(\Lambda\setminus A)$ obtained from $G$ by contracting the classes $A\subset \Lambda$ in $G$.

A partial cube $G$ is an \emph{expansion} of a partial cube $G'$ if $G'=\pi_f(G)$ for some equivalence class $f$ of $\Theta(G)$. More generally, let $G'$ be a graph containing two isometric subgraphs $G'_1$ and $G'_2$ such that $G'=G'_1\cup G'_2$, there are no edges from $G'_1\setminus G'_2$ to $G'_2\setminus G'_1$, and $G'_0:=G'_1\cap G'_2$ is nonempty.  A graph $G$ is an {\it isometric expansion}
of  $G'$ with respect to $G_0$ (notation $G=\psi(G')$) if $G$ is obtained from $G'$ by replacing each vertex $v$ of $G'_1$ by a vertex $v_1$ and each vertex $v$ of  $G'_2$ by a vertex $v_2$ such that $u_i$ and $v_i$, $i=1,2$ are adjacent in $G$ if and only if $u$ and $v$ are adjacent vertices of $G'_i$ and $v_1v_2$ is an edge of $G$ if and only if $v$ is a vertex of $G'_0$. The following is well-known:

\begin{lemma} [\hspace*{-0.13cm}\cite{Ch_thesis,Ch_hamming}]
A graph $G$ is a partial cube if and only if $G$ can be obtained by a sequence of isometric expansions from a single vertex.
\end{lemma}

\begin{lemma}\label{commut_rest_contaction}
Contractions and restrictions commute in partial cubes, i.e., if $f,g\in \Lambda$ and $f\ne g$, then $\rho_{g^+}(\pi_f(G))=\pi_f(\rho_{g^+}(G))$.
\end{lemma}

\begin{proof}
Let $f,g\in \Lambda$ and $f\ne g$. The crucial property is that $E_g$ is an edge-cut of $G$ and $E_g\cap E_f=\emptyset$. If we see vertices as sign vectors in the hypercube, the vertex set of $\pi_f(\rho_{g^+}(G))$ can be described as $\{x\in V(G): x_g=+\}/E_f=\{x\in V(G)\setminus V(E_f): x_g=+\}\cup \{xy\in E_f:  x_g=y_g=+\}$. The vertex set of $\rho_{g^+}(\pi_f(G))$ is $\{x\in V(G)\setminus V(E_f)\}\cup \{xy\in E_f\}\setminus H^-_g$ which again equals $\{x\in V(G)\setminus V(E_f): x_g=+\}\cup \{xy\in E_f:  x_g=y_g=+\}$. Furthermore, identifying a vertex of the form $\{x,y\}\in E_f$ with the vector $z$ arising from $x$ or $y$ by omitting the $f$-coordinate, adjacency is defined the same way in both graphs, namely by taking the induced subgraph of the hypercube. This concludes the proof.
\end{proof}

The previous lemmas show that any set of restrictions and any set of contractions of a partial cube $G$ provide the same result, independently of the order in which we perform the restrictions and contractions. The resulting graph $G'$ is also a partial cube, and $G'$ is called a {\it partial cube-minor} (or {\it pc-minor}) of $G$. In this paper we will study classes of partial cube excluding a given set of minors.


\subsection{Partial cube minors versus metric subgraphs}

In this section we present conditions under which contractions and restrictions preserve metric properties of subgraphs.

Let $G=(V,E)$ be an isometric subgraph of the hypercube
$Q(\Lambda)$ and let $S$ be a subgraph of $G$. Let $f$ be any coordinate of $\Lambda$.
We will say that $E_f$ {\it crosses} $S$ if and only if $S\cap H^-_f\ne \emptyset$ and $S\cap H^+_f\ne\emptyset$. We will say that $E_f$ {\it osculates} $S$ if and only if $E_f$ does not cross $S$ and there exists an edge
$e=uv\in E_f$ such that $\{ u,v\}\cap S\ne\emptyset$. Otherwise, we will say that $E_f$ is {\it disjoint} from $S$.
%

\begin{lemma}\label{contraction_convex}
If $S$ is a convex subgraph of $G$ and $f\in\Lambda$, then $\rho_{f^+}(S)$ is a convex subgraph of $\rho_{f^+}(G)$. If $E_f$ crosses $S$ or is disjoint from $S$, then also $\pi_f(S)$ is a convex subgraph of $\pi_f(G)$.
\end{lemma}
\begin{proof}
 Let $S$ be convex. Then by Lemma~\ref{restriction}, $S$ can be written as $\rho_{A}(G)$, where $A$ is a signed set of those $\Theta$-classes that osculate with $S$. Again by Lemma~\ref{restriction}, $\rho_{f^+}(S)=\rho_{f^+}(\rho_{A}(G))$ is a convex subgraph of $\rho_{f^+}(G)$, proving the first assertion. Now, if $f\in\Lambda\setminus A$, then $\pi_f(S)=\pi_f(\rho_{A}(G))$, which by Lemma~\ref{commut_rest_contaction} equals $\rho_{A}(\pi_f(G))$, i.e., $\pi_f(S)$ is a convex subgraph of $\pi_f(G)$.
\end{proof}

\begin{lemma}\label{convex_expansion} If $S'$ is a convex subgraph of $G'$ and $G$ is obtained from $G'$ by an isometric expansion $\psi$, then $S:=\psi(S')$ is a convex subgraph of $G$.
\end{lemma}

\begin{proof}
Let $f\in\Lambda$ be such that $G'=\pi_f(G)$ and $S'$ a convex subgraph of $G'$. By Lemma~\ref{restriction}, there exists a signed set of $\Theta$-classes $A\subset \Lambda \setminus\{f\}$, such that $S'=\rho_{A}(G')=\rho_{A}(\pi_f(G))$. By Lemma~\ref{commut_rest_contaction}, $S'=\pi_f(\rho_{A}(G))$, thus $ \rho_{A}(G)\subset S$. For every $g^+\in A$, we have $\pi_f(\rho_{g^-}(G))= \rho_{g^-}(G')$, thus it is disjoint with $S$. Then  $S=\rho_{A}(G)$, which is convex by Lemma~\ref{restriction}.
\end{proof}

\begin{lemma}\label{convex_hull} If $S$ is a subset of vertices of $G$ and $f\in \Lambda$, then $\pi_f(\conv(S))\subseteq \conv(\pi_f(S))$. If $E_f$ crosses $S$, then $\pi_f(\conv(S))= \conv(\pi_f(S))$.
\end{lemma}
\begin{proof}
Let $y'\in  \pi_f(\conv(S))$, i.e., there is a $y\in\pi_f^{-1}(y')$ on a shortest path $P$ in $G$ between vertices $x,z\in S$. Contracting $f$ yields a shortest path $P_f$ in $G_f$ between two vertices on $\pi_f(S)$ containing $y'$. This proves $\pi_f(\conv(S))\subseteq \conv(\pi_f(S))$.

For the second claim note that since $\conv(S)\supseteq S$, we have $\pi_f(\conv(S))\supseteq \pi_f(S)$ and $\conv(\pi_f(\conv(S)))\supseteq \conv(\pi_f(S))$. Finally, since $E_f$ crosses $S$ it also crosses $\conv(S)$ and by Lemma~\ref{contraction_convex} we have that $\conv(\pi_f(\conv(S)))=\pi_f(\conv(S))$, yielding the claim.
\end{proof}

We call a subgraph  $S$ of a graph $G=(V,E)$ {\it antipodal} if for every vertex $x$ of $S$ there is a vertex $x^-$ of $S$ such that $S=\conv(x,x^-)$ in $G$. Note that antipodal graphs are sometimes defined in a different but equivalent way (graphs satisfying our definition are also called symmetric-even, see~\cite{BKo}). By definition, antipodal subgraphs are convex.

\begin{lemma}\label{antipodal_minor}
 Let $S$ be an antipodal subgraph of $G$ and $f\in\Lambda$. If $E_f$ is disjoint from $S$, then $\rho_{f^+}(S)$ is an antipodal subgraph of $\rho_{f^+}(G)$. If $E_f$ crosses $S$ or is disjoint from $S$, then $\pi_f(S)$ is an antipodal subgraph of $\pi_f(G)$.
\end{lemma}
\begin{proof}
 If $E_f$ is disjoint from $S$, then $S\rho_{f^+}(S)=S$ and by Lemma~\ref{contraction_convex} is convex. This yields the first assertion.
 For the second assertion, again by Lemma~\ref{contraction_convex}, $\pi_f(S)$ is convex. Moreover, by Lemma~\ref{convex_hull} if $\conv(x,x^-)=S$, then $\pi_f(S)=\pi_f(\conv(x,x^-))=\conv(\pi_f(\{x,x^-\}))$. Since every vertex in $\pi_f(S)$ is an image under the contraction, $\pi_f(S)$ is antipodal.
\end{proof}

\begin{lemma}\label{antipodal_cycle}
If $S$ is an antipodal subgraph of $G$, then $S$ contains an isometric cycle $C$ such that $\conv(C)=S$.
\end{lemma}
\begin{proof}
 Let $x\in S$ and let $P=(x=x_0,x_1,\ldots,x_k=x^-)$ be a shortest path in $S$ to the antipodal vertex $x^-$ of $x$. It is well-known that the mapping $x\mapsto x^{-}$ is a graph automorphism of $S$, thus $C=(x=x_0,x_1,\ldots,x_k=x^-_0,x^-_1,\ldots,x^-_k=x)$ is a cycle. Furthermore, by the properties of the map $x\mapsto x^{-}$ every subpath of $C$ of length at most $k$ is a shortest path. Thus, $C$ is an isometric cycle of $S$. Since $C$ contains antipodal vertices of $S$, we have $\conv(C)=S$.
\end{proof}

%
%
%

\begin{lemma}\label{contraction_gated}
If $S$ is a gated subgraph of $G$, then $\rho_{f^+}(S)$ and $\pi_f(S)$ are  gated subgraphs of $\rho_{f^+}(G)$ and $\pi_f(G)$, respectively.
\end{lemma}
\begin{proof}
 Let $x\in G$ with gate $y\in S$, $z\in S$, and let $P$ be a shortest path from $x$ to $z$ passing via $y$. To prove  that $\rho_{f^+}(S)$ is gated, suppose that $x,z\in \rho_{f^+}(G)$. This
 implies $y\in \rho_{f^+}(G)$, thus $y$ is also the gate of $x$ in $\rho_{f^+}(S)$ in the graph $\rho_{f^+}(G)$.

To prove the second assertion, notice that the distance in $\pi_f(G)$ between $x$ and $z$ decreases by one if and only if $P$ crosses $E_f$ and remains unchanged otherwise, thus $\pi_f(P)$ is a shortest path in $\pi_f(G)$. This shows that  $\pi_f(y)$ is the gate of $\pi_f(x)$ in $\pi_f(S)$ in the graph $\pi_f(G)$.
\end{proof}

%
%

\subsection{Cartesian products}
The \emph{Cartesian product} $F_1\square F_2$ of two graphs $F_1=(V_1,E_1)$ and $F_2=(V_2,E_2)$ is the graph defined on $V_1\times V_2$ with an edge $(u,u')(v,v')$ if and only if $u=v$ and $u'v'\in E_2$ or $u'=v'$ and $uv\in E_1$. This definition generalizes in a straightforward way to products of sets of graphs. If $G=F_1\square\cdots\square F_k$, then each $F_i$ is called a \emph{factor} of $G$. A \emph{subproduct} of such $G$ is a product $F'_1\square\cdots\square F'_k$, where $F'_i$ is a subgraph of $F_i$ for all $1\leq i\leq k$. A \emph{layer} is a subproduct, where all but one of the $F'_i$ consist of a single vertex and the remaining $F'_i$ coincides with $F_i$.

It is well-known that products of partial cubes are partial cubes, and thus products of even cycles and edges are partial cubes, which we will be particularly interested in. It is easy to see that any contraction of a product of even cycles and edges is a product of even cycles and edges. Furthermore, any Cartesian product of even cycles and edges is antipodal, since taking the antipode with respect to all factors gives the antipode with respect to the product. By Lemma~\ref{antipodal_cycle} any such product is the convex hull of an isometric cycle.
We will use the following properties of these graphs frequently (and sometime without an explicit reference):

\begin{lemma} \label{product_cycles_gated} Let $G\cong F_1\square\cdots\square F_k$ be a Cartesian product of edges and even cycles and let $H$ be an induced subgraph of $G$. Then $H$ is a convex subgraph if and only if $H$ is a Cartesian product $F'_1\square\cdots\square F'_m$, where each $F'_i$ either coincides with $F_i$ or is a  convex subpath of $F_i$. Furthermore, $H$ is a gated subgraph of $G$ if and only if $H$ is a Cartesian product $F'_1\square\cdots\square F'_m$, where each $F'_i$ either coincides with $F_i$ or is a  vertex or an edge  of $F_i$.
\end{lemma}

\begin{proof} It is well known (see for example~\cite{VdV} that convex subsets (respectively, gated subsets) of Cartesian products of metric spaces are exactly the Cartesian products of convex (respectively, gated) subsets of factors. Now, the proper convex subsets of an even cycle $C$ are exactly the convex paths, while the proper gated subsets of $C$ are the vertices and the edges of $C$.
\end{proof}

\begin{lemma}\label{lem:subcells}
 Let $G\cong F_1\square\cdots\square F_k$ be a Cartesian product of edges and even cycles  and let $G'$ be a connected induced subgraph of $G$. If for every 2-path $P$ of $G'$ its gated hull $\langle\langle P\rangle\rangle$ is included in $G'$, then $G'$ is a gated subgraph of $G$.
\end{lemma}

\begin{proof}
Let $H$ be a maximal gated subgraph of $G'$. By Lemma~\ref{product_cycles_gated} $H$ is a subproduct $F'_1\square\cdots\square F'_k$ of $G$, such that for all $1\leq i\leq k$ we either have $F'_i= F_i$ or $F'_i$ is a vertex or an edge of $F_i$. Suppose by way of contradiction that  $H\neq G'$. Since $G'$ is connected, there exists an edge $ab$ in $G'$ such that $a\in H$ and $b\in G'\backslash H$. Without loss of generality, assume that $ab$ is an edge arising from the factor $F_1$.  Thus $ab$ can be represented as $a_1b_1 \square v_2 \square \cdots \square v_k$, where $a_1b_1$ is an edge of $F_1$, $a_1\in F'_1$ and $b_1\notin F_1'$. Consider the subgraph $H'=(F'_1 \cup a_1b_1) \square F'_2 \square \cdots  \square F'_k$. We assert that $H'$ is a subgraph of $G'$. For any $i>1$, consider the layer $L'_i$ of $H'$ passing via the vertex $a$. Let $L''_i$ be the subgraph of $G$ obtained by shifting $L'_i$ along the edge $ab$ (thus both $L'_i$ and $L''_i$ are isomorphic to $F'_i$). We assert that $L''_i$ is also included in $G'$. This is trivial if $L''_i$ is a vertex, because then $L''_i=b$. Otherwise, using that the gated hull of any 2-path of $G'$ is included in $G'$, $L''_i$ is connected and $L'_i$ is  in $G'$, one can easily conclude that
$L''_i$  is also included in $G'$. Propagating this argument through the graph, we obtain that $H'$ is a subgraph of $G'$. However, either its factor $(F'_1 \cup a_1b_1)$ is an edge and $H'$ is gated by Lemma~\ref{product_cycles_gated} or it is a 2-path and thus the gated subgraph $\langle\langle F'_1 \cup a_1b_1\rangle\rangle \square F'_2 \square \cdots  \square F'_k$ is contained in $G$. This contradicts the maximality of $H$ and shows that $H=G'$. The proof is complete. 
\end{proof}

\section{Cells in hypercellular graphs and graphs from $\mathcal{S}_4$}\label{sec:cells}


Let $\mathcal{F}(Q_3^-)$  be the class of all partial cubes not containing the 3-cube minus one vertex $Q^-_3$ as a $pc$-minor.
Our subsequent goal will be to establish a cell-structure of such graphs in the following sense.
We show that for $G\in \mathcal{F}(Q_3^-)$, the convex hull of any isometric cycle $C$ of $G$ is gated in $G$ and furthermore isomorphic to a Cartesian product of edges and even cycles. Using these results we establish that a finite partial cube $G$ belongs to $\mathcal{F}(Q_3^-)$ if and only if $G$ can be obtained by gated amalgams from Cartesian
products of edges and even cycles.  Throughout this paper, we will call a subgraph $H$ of a graph $G$ a {\it cell}, 
if $H$ is convex and isomorphic to a Cartesian product of edges and even cycles. 

Some of the results of this section extend to bipartite graphs satisfying the {\it separation property} $S_4$. This is, any two disjoint convex sets $A,B$  can be separated by complementary convex sets $H',H''$, i.e., $A\subseteq H', B\subseteq H''$. By~\cite{Ch_thesis} and \cite[Theorem 7]{Ch_separation}, the $S_4$ separation property is equivalent to the Pasch axiom: for any triplet of vertices $u,v,w\in V$ and $x\in I(u,v), y\in I(u,w)$, we have $I(v,y)\cap I(w,x)\ne \emptyset$.
The bipartite graphs with $S_4$ convexity have been characterized in~\cite{Ch_thesis} and \cite[Theorem 10]{Ch_separation}: these are the partial cubes without any pc-minor among six isometric subgraphs of $Q_4$ five of which were listed in~\cite{Ch_separation} plus $Q^-_4$ -- the cube $Q_4$ minus one vertex. Note that we correct here the result in~\cite{Ch_thesis,Ch_separation}, where $Q^-_4$ was missing from the list. All these six forbidden graphs can be obtained from $Q^-_3$ by an isometric expansion and thus if we denote by $\mathcal{S}_4$ the class of bipartite graphs with $S_4$, then hypercellular graphs are in $\mathcal{S}_4$.


The {\it full subdivision} $H'$ of a graph $H$ is the graph obtained by subdividing every edge of $H$ once. The vertices of $H$ in $H'$
are called the {\it original} vertices of the full subdivision.

\begin{proposition}\label{thm:closure}
Let $G=(V,E)$ be a partial cube and $S\subseteq V$. If $\conv(S)$ is not gated, then either there exists
$f\in \Lambda$ such that $\conv(\pi_f(S))$ is not gated in $\pi_f(G)$ or there is an $m\geq 2$ such that:

\begin{itemize}
 \item[(i)] $G$ contains a full subdivision $H$ of $K_{m+1}$ as an isometric subgraph, and                                                                                                                                                                                                                                                       \item[(ii)] $H$ contains a full subdivision $H'$ of $K_m$, such that no vertex of $G$ is adjacent to all                                                                                                                                                                                                                                                    original vertices of $H'$.                                                                                                                                                                                                                                           \end{itemize}
Furthermore, if $S$ is an isometric cycle of $G$, then $m\ge 3$.
 \end{proposition}

\begin{proof}
Suppose that $G$ contains a subset $S$ such that $X:=\conv(S)$ is not gated. We can assume that $G$ is selected in a such a way that for any element $f\in \Lambda$, the convex hull of $\pi_f(S)$ is gated in $\pi_f(G)$. Since any $f$-contraction of
an isometric cycle $C$ of size at least $6$ is an isometric cycle $\pi_f(C)$ of $\pi_f(G)$ , this assumption is also valid for proving the claim in the case that $S=C$, because $4$-cycles are always gated.  

Let $v$ be a vertex of $G$ that has no gate in $X$ and is as close as possible to $X$, where $d_G(v,X)=\min \{ d_G(v,z): z\in X\}$ is the distance from $v$ to $X$. Let $P_v:=\{ x\in X: d_G(v,x)=d_G(v,X)\}$ be the \emph{metric projection} of $v$ to $X$. Let also $Q_v:=\{ x\in X: I(v,x)\cap X=\{ x\}\}.$ Obviously, $P_v\subseteq Q_v$. Notice that $u\in V$ has a gate in $X$  if and only if $Q_u=P_u$ and $P_u$ consists of a single vertex. We will denote the vertices of $P_v$ by $x_1,\ldots,x_m$. For any vertex $x_i\in P_v$, let $\gamma_i$ be a shortest path from $v$ to $x_i$. Let $v_i$ be the neighbor of $v$ in $\gamma_i$. From the choice of $v$ we conclude that each vertex $v_i$ has a gate in $X$. From the definition of $P_v$ it follows that $x_i$ is the gate of $v_i$ in $X$. Notice that this implies that the vertices $v_1,\ldots,v_m$ are pairwise distinct. Since $x_i\in\gamma_i$ we have $x_i\in W(v_i,v)$. Furthermore, for any $y\in Q_v\setminus \{x_i\}$ we have $y\in W(v,v_i)$ since otherwise $x_i, y\in Q_{v_i}$, which contradicts that $v_i$ has a gate in $X$. Denote the equivalence classes of $\Theta$ containing the edges $vv_1,\ldots,vv_m$ by $E_{g_1},\ldots, E_{g_m}$, respectively. Then each $E_{g_1},\ldots,E_{g_m}$ crosses $X$.
For any edge $zz'$ of $\gamma_i$ comprised between $v_i$ and $x_i$ with $z$ closer to $v_i$ than $z'$, we have $X\subseteq W(z',z)$ and $v\in W(z,z')$. Thus any such edge $zz'$ belongs to an equivalence class $E_f$ which separates $X$ from $v$. Therefore such $E_f$ crosses any shortest path between $v$ and a vertex $y\in Q_v$. Denote the set of all such $f\in \Lambda$ by $A$. Notice that $d_G(v,X)=d_G(v,x_i)=|A|+1$ for any $x_i\in P_v$.

We continue with a claim:

\begin{claim} Each equivalence class $E_e$ crossing $X$ coincides with one of the equivalence classes $E_{g_i}, i=1,\ldots, m$.
\end{claim}

\begin{proof} Assume that $vv_1 \notin E_e$. Denote $X':=\conv(\pi_e(S))$ in $\pi_e(G)$. Let also $v':=\pi_e(v)$ and $x_1':=\pi_e(x_1)$. By Lemma~\ref{convex_hull}, $X'$ coincides with $\pi_e(X)$. Let $x' \in X'$ be the gate of $v'$ in $X'$. Since for each $f \in A$, $E_f$  separates $v$ and $X$, $E_f$ also separates $v'$ and $x'$. Thus $d_{G'}(v',X')\geq |A|$. On the other hand, since $vv_1 \notin E_e$, $d_{G'}(v',x_1')=|A|+1$. But $x_1'$ cannot be the gate in $X'$, thus $d_{G'}(v',x')=d_{G'}(v',X')=|A|$. Let $y$ be such that $x'=\pi_e(y)$. Since $E_e$ crosses $X$ and $X'=\pi_e(X)$, we have $y\in X$. Applying the expansion, the distance between two vertices can only increase by one, thus $d_G(v,y)\leq |A|+1$. On the other hand, it holds that $d_G(v,y)\geq |A|+1$ since $d_G(v,X)= |A|+1$. Thus $y\in P_v,$ say $y=x_i$, and every shortest  $(v,y)$-path traverses  an edge in $E_e$. Since $E_e$ crosses $X$, we have $e\notin A$, whence $E_e=E_{g_i}$.
\end{proof}

First we prove that $P_v=Q_v$. Suppose by way of contradiction that $z\in Q_v\setminus P_v$. Then $d_G(v,z)>d_G(v,X)=|A|+1.$ For any $f\in A$, $E_f$ separates $v$ from $z$. On the other hand, since $z\in W(v,v_i)$, neither of the equivalence classes $E_{g_i}, i=1,\ldots,m$, separates $v$ from $z$. Hence there exists an equivalence class $E_e$ with $e\notin A$ separating $v$ from $z$. Since $e\notin A$, $E_e$ does not separate $v$ from any vertex $x_i$ of $P_v$. Hence $E_e$ crosses $X$, contrary to Claim 1. This shows that $P_v=Q_v$.

Now we will prove that $d_G(x_i,x_j)=2$ for any two vertices $x_i,x_j\in P_v$, which will later yield the existence of $H'$ as claimed. Indeed, since $x_i$ is the gate of $v_i$ in $X$ and $x_j$ is the gate of $v_j$ in $X$, $$d_G(v_i,x_j)=d_G(v_i,x_i)+d_G(x_i,x_j)\le d_G(v_i,v_j)+d_G(v_j,x_j)=2+d_G(v_j,x_j).$$
Analogously, $d_G(v_j,x_i)=d_G(v_j,x_j)+d_G(x_i,x_j)\le 2+d_G(v_i,x_i)$. Summing up the two inequalities we deduce that $d_G(x_i,x_j)\le 2$. Since $G$ is bipartite, $x_i$ and $x_j$ cannot be adjacent, thus $d_G(x_i,x_j)=2$.

Finally, we will show that $d_G(v,X)=2$, yielding $H$ as claimed. Suppose by way of contradiction that $d_G(v,X)\ge 3$. Pick any $f\in A$.  Consider the graph $G':=\pi_f(G)$ and denote the convex hull of the set $S':=\pi_f(S)$ in $G'$ by $X'$. By Lemma~\ref{convex_hull}, $\pi_f(X)\subseteq X'$. Since any class $E_{f'}$ with $f'\in A$, separates $v$ from $X$ (and therefore from $S$) in $G$, any equivalence class $E_{f'}$ with $f'\in A\setminus \{ f\}$ separates $v':=\pi_f(v)$ from $S'$ in $G'$. Therefore, $X'$ is contained in the intersection of the halfspaces defined by
$f'\in A\setminus \{ f\}$ that contain $S'$. This implies that $d_{G'}(v',X')\ge |A|-1$. On the other hand, since $d_{G'}(v',\pi_f(x_i))=d_G(v,x_i)-1$ for any $i=1,\ldots,m$ and $\pi_f(x_i)\in X'$,
we conclude that $d_{G'}(v',X')\le d_G(v,X)-1=|A|+1-1=|A|$. From the choice of the graph $G$, in the graph $G'$ the vertex $v'$ must have a gate $x'$ in the set $X'$. Since $d_{G'}(v',\pi_f(x_i))=d_G(v,x_i)-1$, $i=1,\ldots,m$,  
the vertex $x'$ cannot be one of the vertices of $\pi_f(P_v)$. Thus $d_{G'}(v',x')=d_{G'}(v',X')=|A|-1=d_G(v,X)-2$.

Let $x'_i:=\pi_f(x_i),$ $i=1,\ldots, m$. Since $x'$ is the gate of $v'$ in $X'$ and $x'_i\in X'$, we have $d_{G'}(v',x'_i)=d_{G'}(v',x')+d_{G'}(x',x'_i)$. On the other hand, since $d_{G'}(v',x'_i)=d_G(v,X)-1$ and $d_{G'}(v',x')=d_{G}(v,X)-2$, this implies that $d_{G'}(x',x'_i)=1$ for any $i=1,\ldots,m$. Since $G'$ is obtained
by $f$-contraction of a partial cube $G$, we conclude that $G$ contains a vertex $x$ such that $\pi_f(x)=x'$ and for any $i=1,\ldots, m$ either $d_G(x,x_i)=1$ or $d_G(x,x_i)=2$.
Since $d_G(x_i,x_j)=2$ and $G$ is bipartite, either $x$ is adjacent to all $x_i$, $i=1,\ldots, m$ or $d_G(x,x_i)=2$ for all $i=1,\ldots, m$. First assume $m\geq 2$. In the second case, the vertices
$x_1,\ldots, x_m$ and $x$ together with their common neighbors define the required full subdivision $H$ of $K_{m+1}$. In the first case, we conclude that
$x\in I(x_i,x_j)\subset X$ and since $d_G(v,x)=d_G(v,x_i)-1$, we obtain a contradiction with the choice of $x_i$ from the metric projection $P_v$ of $v$ on $X$. So, assume that $m=1$. Since $P_v=Q_v$, this implies that $X$ is gated, contrary to our assumption.

Finally suppose that $S$ is an isometric cycle $C$ of $G$ whose convex hull $X$ is not gated.
Then the length of $C$ is at least 6 (if $C$ is a 4-cycle, then $X=C$ is gated). Hence there exist at least three different equivalence classes $E_{e_1},E_{e_2},E_{e_3}$ crossing $C$ and $X$. By Claim 1, each of these classes coincides with a class $E_{g_i}, i=1,\ldots,m$. Hence $m\ge 3$.
\end{proof}

\begin{proposition} \label{convex_hull_cycle} Let $C$ be an isometric cycle of $G\in \mathcal{S}_4$. Then the convex hull $\conv(C)$ of $C$ in $G$ is gated.
\end{proposition}
\begin{proof} The class $\mathcal{S}_4$ is closed by taking pc-minors~\cite[Theorem 10]{Ch_separation}.
 Therefore we can suppose that $G$ is maximally contracted graph from $\mathcal{S}_4$ containing an isometric cycle $C$ with $X:=\conv(C)$ not gated. By Proposition~\ref{thm:closure}, $X$ contains 3 vertices $x_1,x_2,x_3$ at pairwise distance 2 and a vertex $v$ at distance 2 from each of the vertices $x_1,x_2,x_3$. Let $v_1,v_2,v_3$ be the common neighbors of $v$ and $x_1,x_2,x_3$, respectively. Let also $z_i$ be a common neighbor of $x_j$ and $x_k$ for all $\{i,j,k\}=\{1,2,3\}$. By Proposition~\ref{thm:closure}, the set  $T=\{ v,x_1,x_2,x_3,v_1,v_2,v_3,z_1,z_2,z_3\}$ defines four 6-cycles which are isometric cycles of $G$. The convex hull in $G$ of each of these 6-cycles is a subgraph of a 3-cube. On the other hand, $T$ is contained in each of the three intervals $I(v_i,z_i)$. Since $d_G(v_i,z_i)=4$, the convex hull of $T$ is a subgraph of a 4-cube. The convex hull of the 6-cycle $C_1=(x_1,z_3,x_2,z_1,x_3,z_2)$ cannot be a 3-cube.
 We conclude that one of the 2-paths $(x_1,z_3,x_2),(x_2,z_1,x_3),(x_3,z_2,x_1)$, say $(x_1,z_3,x_2)$, is a convex path of $G$. Consider the convex sets $I(x_1,x_2)$ and $I(v,x_3)$. They are disjoint, otherwise $z_3$ must be adjacent to $v$ and $x_3$, which is impossible. Let $H,H'$ be two complementary halfspaces separating $I(x_1,x_2)$ and $I(v,x_3)$, say $I(x_1,x_2)\subset H$ and $I(v,x_3)\subset H'$. Then necessarily $z_1,z_2\in H'$, otherwise, if say $z_1\in H$, then $x_3\in C_1\subset I(x_1,z_1)\subset H$, a contradiction. But then $x_1\in I(z_2,v)\subset H'$ and $x_2\in I(z_1,v)\subset H'$, which is impossible. This final contradiction shows that $I(x_1,x_2)$ and $I(v,x_3)$ cannot be separated, i.e., $G\notin\mathcal{S}_4$. Thus, the convex hull of any isometric cycle $C$ of a partial cube $G$ from $\mathcal{S}_4$ is gated.
\end{proof}


Analogously to~\cite{BaCh_cellular}, we will compare the Djokovi\'{c}-Winkler relation $\Theta$ to the following relation $\Psi^*$. First say that two edges $xy$ and $x'y'$ of a bipartite graph $G$ are in relation $\Psi$ if they are either equal or are opposite edges of some convex cycle $C$ of $G$. Then let $\Psi^*$ be the transitive closure of $\Psi$. Let ${\mathcal C}(G)$ denote the set of all convex cycles of $G$ and let ${\mathbf C}(G)$ be the 2-dimensional cell complex whose 2-cells are obtained by replacing each convex cycle $C$ of length $2j$ of $G$ by a regular Euclidean polygon $[C]$ with $2j$ sides.

Recall that a cell complex $\bf X$ is {\it simply connected} if it is connected and if every continuous
map of the 1-dimensional sphere $S^1$ into $\bf X$ can
be extended to a continuous mapping of the disk $D^2$ with boundary
$S^1$ into $\bf X$. Note that a connected complex  $\bf X$ is simply connected  if and only if  every continuous map from $S^1$ to the 1-skeleton of $\bf X$ is
null-homotopic.

\begin{lemma} \label{partial-cube-convex-cycle} If $G$ is a partial cube, then the relations $\Theta$ and $\Psi^*$ coincide. In particular, ${\mathbf C}(G)$ is simply connected.
\end{lemma}

\begin{proof} The proof of the first assertion is the content of \cite[Proposition 5.1]{KlSh} (it also follows by adapting the proof of ~\cite[Lemma 1]{BaCh_cellular}). To prove that ${\mathbf C}(G)$ is simply connected it suffices to show that any cycle $C$ of $G$ is contractible in ${\mathbf C}(G)$. Let $k(C)$ denote the number of equivalence classes of $\Theta$ crossing $C$. By induction on $k(C)$, we will prove that any cycle $C$ of $G$ is contractible to any of its vertices $w\in C$. Let $E_f$ be an equivalence class of $\Theta$ crossing $C$ and let $uv$ and $u'v'$ be two edges of $C$ from $\Theta$. By the first assertion, there exists a collection $C_1,C_2,\ldots,C_m$ of convex cycles and a collection of edges $e_0=uv,e_1,\ldots,e_{m-1},e_{m}=u'v'\in E_f$ such that $e_i\in C_i\cap C_{i+1}$ for any $i=1,\ldots,m-1$.  Suppose that $u,u'\in H^+_f$ and $v,v'\in H^-_f$. Let $P^+_i:=C_i\cap H^+_f, P^-_i:=C_i\cap H^-_f$ for $i=1,\ldots, m$. Let $P'$ be the path between $u$ and $u'$ which is the union of the paths $P^+_1,\ldots,P^+_m$. Analogously, let $P''$ be the path between $v$ and $v'$ which is the union of the paths $P^-_1,\ldots,P^-_m$. Finally, let $P^+:=C\cap H^+_f$ and $P^-:=C\cap H^-_f$, and suppose without loss of generality that the vertex $w$ belongs to the path $P^+$. Let $C'$ be the cycle which is the union of the paths $P'$ and $P^+$ and let $C''$ be the cycle which is the union of the paths $P^-$ and $P''$. Since $G$ is a partial cube, any equivalence class of $\Theta$ crossing $P'$
or $P''$ also crosses the paths $P^+$ and $P^-$. On the other hand, $E_f$ does not cross the cycles $C'$ and $C''$. This implies that  $k(C')<k(C)$ and $k(C'')<k(C)$. By induction assumption, $C''$ can be contracted in ${\mathbf C}(G)$ to any of its vertices, in particular to the vertex $v$. On the other hand, the union $\bigcup_{i=1} [C_i]$ can be contracted to the path $P'$ in a such a way that each edge $e_i$ is contracted to its end from $P''$. In particular, $v$ is mapped to $u$. Finally, by induction assumption, $C'$ can be contracted to the vertex $w$. Composing the three contractions ($C''$ to $v$,  $\bigcup_{i=1} [C_i]$ to $P'$, and $C'$ to $w$), we obtain a contraction of $C$ to $w$.
\end{proof}

Let $C$ be an even cycle of length  $2n$. Let $G_1$ be a subgraph of $C$ isomorphic to a path of length $\ell$ at least 2 and at most $n$. Let $\mathrm{Ex}_{\ell}(C)$ be an expansion of $C$ with respect to $G_1$ and $G_2=C$. We will call the graphs $\mathrm{Ex}_{\ell}(C)$ \emph{half-expanded cycles}. 

\begin{proposition} \label{product_of_cycles_expansion} Let $G'$ be a Cartesian product of edges and even cycles and let $G$ be an isometric expansion with respect to the subgraphs $G'_1$ and $G'_2$ of $G'$, such that $G$ contains no convex subgraph isomorphic to a half-expanded cycle. Then either $G$ is a Cartesian product of edges and even cycles or one of $G'_1, G'_2$ coincides with $G'$  while the other  is isomorphic to a subproduct of edges and cycles.  
\end{proposition}

\begin{proof} Let $G'=F_1\square F_2\square \cdots \square F_m$, 
where each $F_i, i=1,\ldots, m$ is either a $K_2$ or an even cycle $C$. Then $G'$ is a partial cube from $\mathcal{F}(Q_3^-)$.  The graph $G$ is obtained from $G'$ by an isometric expansion with respect to $G'_1$ and $G'_2$, i.e., $G'_1$ and $G'_2$ are two isometric subgraphs of $G'$ such that $G'=G'_1 \cup G'_2$, $G'_0:=G'_1\cap G'_2 \neq \emptyset$, there is no edge between $G'_1\setminus G'_2$ and $G'_2\setminus G'_1$, and $G$ is obtained from $G'$ by expansion along $G'_0$. If $G'_1=G'_2=G'$, then the expansion of $G'$ with respect to $G'_1$ and $G'_2$ is the product $G'\square K_2$ and we are done. Thus we can assume that $G'_0$ is a proper subgraph of $G'$.



\begin{claim} \label{claim} Let $L$ be a layer of $G'$, i.e., $L=\{v_1\}\square  \cdots \square \{v_{i-1}\}\square F_i\square \{v_{i+1}\}\square\cdots \square \{v_{m}\}$ with $v_j\in F_j$ for all $j\neq i$. If $F_i$ is a  cycle and $G'_1\cap L$ or $G'_2\cap L$ is different from $L$ and contains a path of length at least 2, then $G$ is a Cartesian product of edges and cycles. More precisely, $G=F_1\square  \cdots \square F_{i-1}\square F'_i\square F_{i+1}\square\cdots \square F_m$, where $F'_i$ is an isometric expansion of $F_i$ along two opposite vertices of $F_i$.
\end{claim}

\begin{proof} Since we can reorder the factors, suppose without loss of generality that $i=1$ and denote $G''=F_2\square \cdots \square F_m$. We have $L=F_1 \square \{ v\}=C \square \{ v\}$ with $v\in V(G'')$, such that $G'_1\cap L$ includes a path $P_1$ of length at least 2 but differs from $L$. Since $L$ is a convex $2j$-cycle of $G'$, $P_1$ is a shortest path of $G'$. If $L$ is included in $G'_2$, then the expansion of $L$ along $P_1$ is isomorphic to a half-extended cycle and is a convex subgraph of $G$ by Lemma~\ref{convex_expansion}, which is impossible. Thus $L$ is not included in $G'_2$, yielding that $L\cap G'_2$ is a shortest path $P_2$ of $G'$. Since $P_1$ and $P_2$ cover the cycle $L$, the only possibility is that $P_1$ and $P_2$ intersect in two antipodal vertices of the cycle $L$. Thus the image of $L$ in $G$ is a cycle of length $2j+2$. Consider any layer $L'=C \square \{ v'\}$ of $G'$ adjacent to $L$, i.e., $v'v\in E(G'')$. Then $L\cup L'$ is a convex subgraph of $G'$, thus by Lemma~\ref{convex_expansion} the expansion of $L\cup L'$ is a convex subgraph $H$ of $G$. If $L'$ is contained in $G'_1$,  then the intersection of $H$ with the half-space of $G$ corresponding to $G'_1$ is a convex subgraph isomorphic to a half-extended cycle. Thus $L'$ cannot be entirely in $G'_1$, and for the same reason it cannot be entirely in $G'_2$. Again the only possibility is that $G'_1\cap L'$ and $G'_2 \cap L'$ are shortest paths of $L'$ that intersect in two opposite vertices of $L'$.

Let $v_1,v_2 \in L \cap G'_0$ and $u_1,u_2 \in L' \cap G'_0$. We assert that after a possible relabeling, $v_1$ is adjacent to $u_1$ and $v_2$ is adjacent
to $u_2$. Suppose that this is not true. Then the neighbors $v'_1$ and $v'_2$ of  $v_1$ and $v_2$, respectively, in $L'$ are both different from $u_1$ and $u_2$. Analogously,
the neighbors $u'_1$ and $u'_2$ of  $u_1$ and $u_2$, respectively, in $L$ are both different from $v_1$ and $v_2$.  We can assume without loss of generality that $v_1'$ and
 $u'_1$ are not in $G'_1$, otherwise we can exchange $v_1$ and $v_2$ or $u_1$ and $u_2$.  We assert that $G'_1$ is not an isometric subgraph of $G'$.   Indeed, the distance in
 $G'$ between $u_1$ and $v_1$ is at most $j$ and the interval $I(v_1,u_1)$ is contained in the union $Q\cup Q'$, where $Q$ is the subpath between $v_1$ and $u'_1$ of the path
 between $v_1$ and $v_2$ passing via $u'_1$ and $Q'$ is the subpath between $v'_1$ and $u_1$ of the path between $v'_1$ and $v'_2$ passing via $u_1$. Since all vertices of $Q$
 except $v_1$ belong only to $G'_2$  and $v'_1$ does not belong to $G'_1$, we conclude that any shortest path in $G'$ between $v_1$ and $u_1$ contains at least one vertex from
 $G'_2\setminus G'_1$,  showing that  $G'_1$ is not an isometric subgraph of $G'$.  Hence $v_1$ is adjacent to $u_1$ and $v_2$ is adjacent to $u_2$. Notice that then the both layers
 have the same side of the cycles in $G'_1$ and $G'_2$ since there is no edge between $G'_1\setminus G'_2$ and $G'_2\setminus G'_1$. Propagating this argument through the graph, we conclude that all
 layers $L''$ parallel to $L$   have the same vertices in $G'_1$ and $G'_2$. Hence the traces on $F_i=C$ with respect to $G'_1$ and $G'_2$ of $L$ and $L''$  coincide:
 they are two paths $P_1$ and $P_2$ of $C$ covering the cycle and intersecting  in two opposite vertices $x',x''$ of $C$. Therefore, the graph $G'_0$ with respect
 to which we perform the isometric expansion is the subgraph of $G'$ induced by $(\{ x'\}\times V(G''))\cup (\{ x''\}\times V(G''))$, $G'_1$ is the subgraph induced by
 $V(P_1)\times V(G'')$,  and $G'_2$ is the subgraph induced by $V(P_2)\times V(G'')$. Consequently, the expansion of $G'$ with respect
 to $G'_1$ and $G'_2$ produces a graph $G$ isomorphic to $C'\square G''$, where the length of the cycle $C'$ is two more than the length of $C$. This establishes the claim.  
\end{proof}

By Claim~\ref{claim}, we can further assume  that  every layer $L$ of $G'$ coming from a cyclic factor satisfies one of the following two conditions: either both $G'_1$ and $G'_2$ include $L$, or one of $G'_1,G'_2$ includes an edge, a vertex, or nothing while the other  includes the whole layer $L$. Consequently, for each cyclic factor $F_i\cong C$ of $G'$ and each layer $L:=\{ v_1\}\square \cdots \square \{ v_{i-1}\} \square C \square \{ v_{i+1}\}\square \cdots \square \{ v_m\}$, the intersection   $L\cap G'_0$ is the whole layer $L$, an edge, a vertex, or empty.

We will now analyze  the structure of the subgraph $G'_0$ of $G'$ along which we perform the isometric expansion. Suppose that $G'$ is obtained from $G$ by contracting the equivalence class $E_f$.

\begin{claim} \label{4-cycle-layer} If $R'=(u_1,v_1,v_2,u_2)$ is a $4$-cycle in $G'$ such that the edges $u_1v_1$ and $v_1v_2$ do not lie in the same layer and $u_1,v_1,v_2 \in V(G'_0),$  then $u_2$ also belongs to $V(G'_0)$.
\end{claim}

\begin{proof} If this is not the case, then assume without loss of generality that $u_2\in G'_1\setminus G'_2$. Since the 4-cycle $R'$ is a convex subgraph of $G'$, by Lemma~\ref{convex_expansion} the expansion of $R'$ along $G'_0$ is a convex subgraph of $G$ isomorphic to $Q_3^-$, thus is a half-extended cycle, a contradiction. This contradiction shows that $u_2\in V(G'_0)$.
\end{proof}

We continue with an auxiliary assertion:

\begin{claim} \label{4-cycle} Any convex cycle $Z$ of $G$ crossed by $E_f$ is a 4-cycle.
\end{claim}

\begin{proof}
Assume by way of contradiction that $Z$ has length $\ell(Z)\ge 6$. Therefore $Z$ is contracted to a convex cycle $Z'$ of length $\ell(Z)-2$ of $G'$.
The convex sets in a Cartesian product  are products of convex sets of the factors. Thus either $Z'$ is a layer of $G'$ or $Z'$ is a 4-cycle which is a product of two edges from two different factors. In the first case $Z'$ is a layer which has two antipodal vertices in $G'_0$, one path between these vertices in $G'_1\setminus G'_2$, and the other path in $G'_2\setminus G'_1$, and this case was covered by Claim~\ref{claim}.   Thus assume that $Z'$ is a 4-cycle $(v_1,v_2,u_2,u_1)$ that has edges $v_1v_2, u_1u_2$ projected to factor $F_1$ and
edges $v_1u_1,v_2u_2$ projected to factor $F_2$. Moreover, let $v_1,u_2\in V(G'_1)\cap V(G'_2)$, $v_2 \in V(G'_1)\setminus V(G'_2)$, and $u_1 \in V(G'_2)\setminus V(G'_1)$. If both factors $F_1,F_2$ are isomorphic to $K_2$, then they can be treated as a single cyclic factor because $K_2\square K_2$ is a 4-cycle and $Z'$ is a layer. Then the result follows from Claim~\ref{claim}. Thus assume that $F_1$ is a cycle of length at least $6$ -- otherwise we are in the above case.
Let
$L_1=(v_1,v_2,\ldots ,v_{2i},v_1)$ and $L_2=(u_1,u_2,\ldots ,u_{2i},u_1)$ be the two layers of  $F_1$ that include $Z'$. They include vertices $u_1, v_2$ which are not in $G'_0$. Since $i\geq 3$ by isometry of $G'_1$ and $G'_2$, we have $v_3\in G'_1\setminus G'_2$ and $u_3\in G'_2\setminus G'_1$. But $v_3$ and $u_3$ are adjacent, which is impossible.  This establishes that any convex cycle $Z$ of $G$ crossed by $E_f$ has length 4.
\end{proof}
%


\begin{claim} \label{claim_G'_0} $G'_0$ is a subgraph of $G'$ of the form $H_1\square H_2\square\cdots \square H_m$, where each factor satisfies $H_i\subseteq F_i$ and is either a vertex, an edge, or the entire $F_i$. In particular, $G'_0$ is convex in $G'$.
\end{claim}

\begin{proof}  First we prove that $G'_0$ is connected. Let $a_1a_2$ and $b_1b_2$ be any two edges in the equivalence class $E_f$. Edges  $a_1a_2$ and $b_1b_2$ get contracted to  vertices $a',b'$ of $G'$. By Lemma~\ref{partial-cube-convex-cycle}, $a_1a_2$ and $b_1b_2$ can be connected by a sequence ${\mathcal C}=C_1,C_2,\ldots,C_k$  of convex
cycles of $G$ such that $a_1a_2\in C_1, b_1b_2\in C_k,$ and any two consecutive cycles $C_i$ and $C_{i+1}$ intersect in an edge of $E_f$.   Hence the cycles of $\mathcal C$ are contracted in $G'$ to a path $Q'$ between $a'$ and $b'$. Since all cycles $C_i$ of $\mathcal C$ are crossed by $E_f$, by Claim~\ref{4-cycle}  each $C_i,i=1,\ldots,k,$ is a 4-cycle. Thus, additionally to $a',b'$ also all  other vertices of the path $Q'$ belong to $G'_0$. Consequently, $a'$ and $b'$ belong to a common connected
component of $G'_0$.  Since $a_1a_2$ and $b_1b_2$ are arbitrary edges from $E_f$, the graph $G'_0$ is connected.

To prove the second assertion, let $I$ be a maximal subgraph of $G_0'$ of the form $I_1\square I_2\square\cdots \square I_m$, where each $I_i$ is a connected nonempty subgraph of $F_i$.
We claim that $I$  coincides with $G'_0$. If not, since $I$ and $G'_0$ are connected, there exists an edge  $vw$ of
$G'_0$ such that $v \in V(I)$ and $w\in V(G'_0)\setminus V(I)$. Let $L:=\{ v_1\}\square \cdots \square \{ v_{i-1}\} \square F_i \square \{ v_{i+1}\}\square \cdots \square \{ v_m\}$
be the layer of $G'$ that includes the edge $vw$.  Suppose that the $i$th coordinates of $v$ and $w$ are the adjacent vertices $v'_i$ and $v''_i$ of $F_i$, respectively.
Set $I':= I_1\square \cdots \square I_{i-1} \square \{ v'_i\}\square I_{i+1}\square \cdots \square I_m$ and $I'':= I_1\square \cdots \square I_{i-1} \square \{ v''_i\}\square I_{i+1}\square \cdots \square I_m$. Then $v\in V(I')\subseteq V(I)$ and $w\in V(I'')$. Since $w\notin V(I)$, by the definition of $I$, the subgraph $I''$ contains a vertex  not belonging to $G'_0$. 
Let $x$ be a closest to $w$ vertex of $V(I'')\setminus V(G'_0)$. Let $y$ be a neighbor of $x$ in $I(x,w)$. Since $I''$ is convex, $y\in I(x,w)\subset V(I'')$. By the choice of $x$, we deduce
that $y$ is a vertex of $G'_0$. Let $x'$ and $y'$ be the neighbors of respectively $x$ and $y$  in $I'$ (such vertices exist by the definitions of $I'$ and $I''$ and the
fact that $v'_i$ and $v''_i$ are adjacent in $F_i$). Since $V(I')\subset V(I)$, the vertices $x'$ and $y'$ belong to $G'_0$. Since the 4-cycle $(x,y,y',x')$ does not belong
to a single layer and $y,y',x'$ are vertices of $G'_0$, by Claim~\ref{4-cycle-layer} also $x$ is a vertex of $G'_0$, a contradiction with its choice. This establishes that $I$ coincides
with $G'_0$.

Finally, we assert that each $I_i, i=1,\ldots, m,$ is a vertex, an edge, or the whole factor $F_i$.
The assertion obviously holds if $F_i$ is an edge. Now, let $F_i=C$ be an even cycle. Since $I_i\neq\emptyset$ and $G'_0= I_1\square I_2\square\cdots \square I_m$, the assertion follows from the conclusion after Claim~\ref{claim}, that the intersection of $G'_0$ with any layer is the whole layer, an edge, a vertex, or empty.
\end{proof}

Let $H'$ be the subgraph of $G'$ induced by all vertices of $G'$ not belonging to $G'_0$.

\begin{claim} \label{connected_H'}  $H'$ is either empty or is a connected subgraph of $G'$.
\end{claim}

\begin{proof} By Claim~\ref{claim_G'_0}, $G'_0$ is a connected subgraph of $G'$ of the form $H_1\square H_2\square\cdots \square H_m$, where each $H_i$ is a vertex, an edge of $F_i$, or the whole factor $F_i$. Suppose that  $G'_0$ is a proper subgraph of $G'$.  By renumbering the factors in the product $F_1\square F_2\square \cdots \square F_m$ we can suppose that there exists an index $m'\le m$, such that for each  $i\le m'$,  $H_i$ is a proper subgraph of $F_i$ and that  for each $m'<i\le m$, we have $H_i=F_i$. For each $i\le m'$, let $F'_i$ be the (nonempty) connected subgraph of $F_i$ induced by $V(F_i)\setminus V(H_i)$. For any $i\le m$,  let $H'_i$ be the subgraph $F_1\square \cdots \square F_{i-1}\square F'_i\square F_{i+1}\square \cdots \square F_m$ of $G'=F_1\square\cdots \square F_m$. Obviously, each such $H'_i$ is a connected subgraph of $H'$ (and of $G'$). Moreover, $V(H')=\bigcup_{i=1}^m V(H'_i)$ and any two $H'_i$ and $H'_j$ with $i,j\le m$ share a vertex. This shows that $H'$ is a connected subgraph of $G'$.
\end{proof}

Now, we are ready to conclude the proof of the proposition. If both $G'_1$ and $G'_2$ are proper subgraphs of $G'$, then $G'_0$ is also a proper subgraph of  $G'$. By Claim~\ref{connected_H'},
the subgraph $H'$ of $G'$ induced by all vertices not in $G'_0$ is connected. This implies that $G'$ contains edges running between the vertices of $G'_1\setminus G'_2$ and $G'_2\setminus G'_1$, which is impossible. Consequently, we can suppose that $G'_1$ coincides with $G'$ and $G'_2$ coincides with $G'_0$.  By Claim~\ref{claim_G'_0}, $G'_2= G'_0$ has the form
$H_1\square H_2\square\cdots \square H_m$, where each $H_i$ is a vertex or an edge of $F_i$, or the whole factor $F_i$.
\end{proof}

Since each half-extended cycle can be contracted to a $Q_3^-$, we immediately have the following lemma.

\begin{lemma}\label{lemma:half-extended}
If $G\in \mathcal{F}(Q_3^-)$, then $G$ has no convex subgraph isomorphic to a half-extended cycle.
\end{lemma}

Now we are ready to prove the first part of Theorem~\ref{mthm:cells}.

\begin{theorem}\label{product_of_cycles}
The convex closure of any isometric cycle of a graph $G$ in $\mathcal{F}(Q_3^-)$ is a gated subgraph isomorphic to a Cartesian product of edges and even cycles.
\end{theorem}

\begin{proof}  Let $G$ be a minimal graph  in $\mathcal{F}(Q_3^-)$ for which we have to prove that the convex closure of an isometric cycle $C$ of $G$ is a product of cycles and edges. Since $G$ is minimal and convex subgraphs of graphs in $\mathcal{F}(Q_3^-)$ are also in $\mathcal{F}(Q_3^-)$,  we conclude that $G$ coincides with the convex closure of $C$. If $C$ is a 4-cycle, then $C$ is a convex subgraph of $G$ and we are done. Analogously, if $C$ is a 6-cycle, then since $G\in \mathcal{F}( Q_3^-),$ either $C$ is convex or the convex hull of $C$ is the 3-cube $Q_3$. So, assume that the length of $C$ is at least 8. By minimality of $G$,  any equivalence class of $G$ crosses $C$. Any contraction of $G$  is a graph $G'$ in $\mathcal{F}(Q_3^-)$ and it maps $C$ to an isometric cycle $C'$ of $G'$. By Lemma~\ref{convex_hull}, $G'$ is the convex hull of $C'$, thus by minimality choice of $G$, $G'$ is
a Cartesian product of cycles and edges, say $G'$ is isomorphic to $F_1\square F_2\square \cdots \square F_m$, where each $F_i, i=1,\ldots, m,$ is either a $K_2$ or an even cycle $C$.
The graph $G$ is obtained from $G'$ by an isometric expansion, i.e., there exist isometric subgraphs $G'_1$ and $G'_2$ of $G'$ such that $G'=G'_1 \cup G'_2$, $G'_0:=G'_1\cap G'_2 \neq \emptyset$, there is no edge between $G'_1\setminus G'_2$ and $G'_2\setminus G'_1$, and $G$ is obtained from $G'$ by expansion along $G'_0$.

By Proposition~\ref{product_of_cycles_expansion} and Lemma~\ref{lemma:half-extended},  either $G$ is a Cartesian product of edges and even cycles  or $G'_1$ coincides with $G'$  and $G'_2= G'_0$ is a proper convex subgraph of $G'$ of the form $H_1\square H_2\square\cdots \square H_m$, where each $H_i$ is a vertex, an edge of $F_i$, or the whole factor $F_i$. In the first case we are done, so suppose that the second case holds. Let $G_j$ be the image of $G'_j$ after the expansion, for $j=0,1,2$. Since $G'_2= G'_0$ is convex, $G_0$ is a convex subgraph of $G$ isomorphic to $G'_0\square K_2$. If $G'$ is the $f$-contraction of $G$, then $G_1$ and $G_2$ are the subgraphs induced by the halfspaces $H^+_f$ and $H^-_f$ of $G$.  Let $a_1a_2$ and $b_1b_2$ be two opposite edges of $C$ belonging to $E_f$.
Since $G$ is the convex hull of $C$, the cycle $C$ intersects every equivalence class of the relation $\Theta$ in $G$. In particular, this implies that contracting $E_f$, the edges $a_1a_2$ and $b_1b_2$ are contracted to vertices $a'$ and $b'$ of $G'_0$. Since $G'_0$ is a convex subgraph of $G'$, the image of $C$ under this contraction is an isometric cycle $C'$ of $G'$. Since $a',b'\in C'$,  $C'$ is  contained in $G'_0$. Since  by Lemma~\ref{convex_hull} $G'$ is the convex hull of $C'$, we conclude that $G'_1=G'_0=G'_2$, contrary to the assumption that $G'_0$ is a proper subgraph of $G'$.
\end{proof}

Together with Proposition~\ref{convex_hull_cycle}, the following gives the second part of Theorem~\ref{mthm:cells}.

\begin{proposition}\label{prop:antipodal}
 The antipodal subgraphs of graphs from $\mathcal{S}_4$ are gated and are products of edges and cycles.
\end{proposition}

\begin{proof}
Let $G$ be a  antipodal graph in $\mathcal{S}_4$ which is not in $\mathcal{F}(Q_3^-)$. Then $G$ can be contracted to a graph $G'$ that contains a convex subgraph $X$ isomorphic to $Q_3^-$. By Lemma~\ref{antipodal_minor} any contraction of an antipodal graph is an antipodal graph, thus we can assume that $G'$ is maximally contracted, i.e. every contraction of $G'$ is in $\mathcal{F}(Q_3^-)$.
%
Denote the central vertex of $X$ with $x$, the isometric cycle around it with $(v_0,v_1,\ldots ,v_5)$, and assume that $x$ is adjacent to exactly $v_0,v_2$ and $v_4$. Let $x',v_0',\ldots,v_5'\in V(X')$ be the antipodes of vertices in $X$.

First assume that $G'$ has exactly three $\Theta$-classes, namely $E_{xv_0},E_{xv_2},E_{xv_4}$. Then either $G'=X$ or $G'\cong Q_3$. In the first case $G'$ is not antipodal, while in the second case $X$ is not a convex subgraph of $G'$. Thus assume that there exists another $\Theta$-class, say $E_{wz}$. Contracting this class we obtain a graph $G''$ that has no $Q_3^-$ convex subgraphs, thus the convex closure $X'$ of $X$ in $G''$ must be isomorphic to $Q_3$. Let $ y$ be a vertex in $G'$ that gets mapped to the vertex in $X'\setminus X$ in $G''$. Vertex $y$ and $v_1$ are adjacent in $G''$ with edge in $E_{xv_4}$, but since $X$ is convex in $G'$ any path  from $v_1$ to $y$ in $G'$ must be of length 2 and first cross an edge in $E_{wz}$ and then an edge in $E_{xv_4}$. Thus there is only one such path, say $P$, and it is a convex subgraph. On the other hand, the path $(v_3,v_4,v_5)$ is convex in $X$, thus it is convex in $G'$. But then the paths $(v_3,v_4,v_5)$ and $P$ are convex subgraphs that cannot be separated by two complementary halfspaces. The latter holds since there is a path between them consisting of edges in $E_{xv_4},E_{xv_0},E_{xv_2}$, but each of these $\Theta$-classes intersects either one convex set or another. Thus $G'$ is not in $\mathcal{S}_4$. By~\cite[Theorem 10]{Ch_separation}, contracting a graph in $\mathcal{S}_4$ gives a graph in $\mathcal{S}_4$. Thus also $G$ is not in $\mathcal{S}_4$.
\end{proof}

%
%

The example in Figure~\ref{fig:apiculatelopnotpasch} shows that the second condition of Theorem~\ref{mthm:cells} does not characterize bipartite graphs with $S_4$ convexity.

\begin{figure}[ht]
\begin{center}
\includegraphics[width = .4\textwidth]{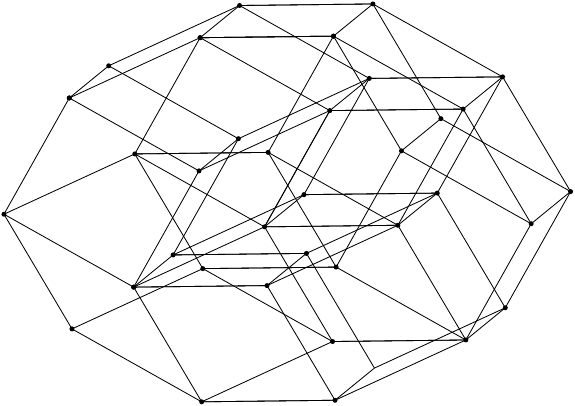}
\caption{An expansion of $Q_4^-$ that is apiculate and lopsided. In particular, the graph is not in $S_4$, the convex hull of any isometric cycle is gated, and its antipodal subgraphs are cubes and thus products of edges and cycles.}
\label{fig:apiculatelopnotpasch}
\end{center}
\end{figure}

\section{Gated amalgamation in hypercellular graphs}\label{sec:amalgam}


This section is devoted to the proof of Theorem~\ref{mthm:amalgam}. 
First, we present the 3CC-condition for partial cubes $G$  in a seemingly stronger but equivalent form:

\medskip\noindent
{\it 3-convex cycles condition} ({\it 3CC-condition}):  for  three convex cycles $C_1,C_2,C_3$ of $G$ such that any two cycles $C_i,C_j$, $1\le i<j\le 3$,
intersect in an edge $e_{ij}$ with $e_{12}\ne e_{23}\ne e_{31}$ and the three cycles intersect in a vertex, the convex hull of $C_1\cup C_2\cup C_3$ is a cell of $G$
isomorphic to $C_i \square K_2$ and $C_j,C_k$ are 4-cycles.

In fact, the above condition is equivalent to the 3CC-condition defined in Section \ref{sec:intro} since in the Cartesian product of cycles and edges $H=F_1\square\cdots\square F_m$
the only convex cycles are layers $\{v_1\}\square  \cdots \square \{v_{i-1}\}\square F_i\square \{v_{i+1}\}\square\cdots \square \{v_{m}\}$ for $F_i$ a cycle, or 4-cycles of the form $\{v_1\}\square  \cdots \square \{v_{i-1}\}\square F'_i\square \{v_{i+1}\}\square\cdots \square \{v_{j-1}\}\square F'_j\square \{v_{j+1}\} \square \cdots \square \{v_{m}\}$ for $F'_i, F_j'$ edges of $F_i, F_j$, respectively. The intersection condition of $C_1,C_2,C_3$ implies that at least two of the cycles are 4-cycles and their convex closure is of the form $\{v_1\}\square  \cdots \square \{v_{i-1}\}\square F_i\square \{v_{i+1}\}\square\cdots \square \{v_{j-1}\}\square F'_j\square \{v_{j+1}\} \square \cdots \square \{v_{m}\}$ for $F_i$ a cycle and $F_j'$ an edge, or $\{v_1\}\square  \cdots \square \{v_{i-1}\}\square F'_i\square \{v_{i+1}\}\square\cdots \square \{v_{j-1}\}\square F'_j\square \{v_{j+1}\} \square \cdots \square \{v_{k-1}\}\square F'_k\square \{v_{k+1}\} \square \cdots \square \{v_{m}\}$ for $F'_i, F_j',F_k'$ edges. In both cases cases  $C_1,C_2,C_3$ are contained in a subcell isomorphic to $C_i \square K_2$.

Any cell $X'$ which is contained in a cell $X$ of a partial cube $G$ is called a {\it face} of $X$. By Lemma~\ref{product_cycles_gated} equivalently, the faces of $X$ are the gated subgraphs of $G$ included in $X$. We denote by $X(G)$ the set of all cells of $G$ and call $X(G)$ the {\it combinatorial complex} of $G$. The {\it dimension} $\mbox{dim}(X)$ of a cell $X$ of $G$ is the number of edge-factors plus two times the number of cyclic factors.
Let us now recall the stronger 3C-condition for partial cubes $G$:

\medskip\noindent
{\it 3-cell condition} ({\it 3C-condition}):  for all $k\geq 0$ and three $(k+2)$-dimensional cells $X_1,X_2,X_3$ of $G$ such that each of the pairwise intersections $X_{12},X_{23},X_{13}$ is a cell of dimension $k+1$ and the intersection $X_{123}$ of all three cells is a cell of dimension $k$, the convex hull of $X_{1}\cup X_{2}\cup X_{3}$ is a $(k+3)$-dimensional cell.

The proof of Theorem~\ref{mthm:amalgam} is organized in the following way. We start by showing that any hypercellular graph satisfies the 3CC-condition. Together with Theorem~\ref{mthm:cells}, this shows that (i)$\Rightarrow$(ii). We then obtain (ii)$\Rightarrow$(iii), while (iii)$\Rightarrow$(ii) holds trivially. To prove (ii)$\Rightarrow$(i), we show that the class of partial cubes satisfying (ii) is closed by taking minors. Since $Q^-_3$ does not satisfies the 3CC-condition, we conclude that all such graphs are hypercellular.
The last and longest part of the section is devoted to the proof of the equivalence (i)$\Leftrightarrow$(iv).


Since by Theorem~\ref{mthm:cells}, hypercellular graphs have gated cells, the following lemma completes the proof of (i)$\Rightarrow$(ii).
\begin{lemma} \label{three_cycles}  Any hypercellular graph $G$ satisfies the 3CC-condition.
\end{lemma}

\begin{proof} Let $C_1,C_2,C_3$ be three convex cycles of a partial cube $G\in\mathcal{F}(Q_3^-)$ such that any two cycles $C_i,C_j$, $1\le i<j\le 3$, intersect in an edge $e_{ij}$ and the three cycles intersect in a vertex $v$. We proceed by induction on the number of vertices of $G$. By induction assumption we can suppose that $G$ is the convex hull of the union $C_1\cup C_2\cup C_3$. If each of the cycles $C_1,C_2,C_3$ is a 4-cycle, then their union is an isometric subgraph $H$ of $G$ isomorphic to $Q^-_3$. Since $G\in \mathcal{F}(Q_3^-)$,  $H$ is not convex. Therefore the convex hull of $H=C_1\cup C_2\cup C_3$ is the 3-cube $Q_3$, and we are done. Thus suppose that one of the cycles, say $C_1$, has length $\ge 6$.  Let the edge $e_{12}$ be of the form $v_1v$. Let $u$ be the neighbor of $v_1$ in $C_1$ different from $v$. Let $E_f$ be the equivalence class of $\Theta$ defined by the edge $uv_1$. We claim that $E_f$ does not cross
$C_2$ and $C_3$, or, equivalently, that $C_2\cup C_3\subset W(v_1,u)$. Since $G\in {\mathcal F}(Q^-_3)$, by Proposition~\ref{convex_hull_cycle}, each of the cycles $C_1,C_2,C_3$ is a gated subgraph of $G$. Since $u$ is adjacent to $v_1\in C_2$, the vertex $v_1$ is the gate of $u$ in $C_2$, whence $C_2\subset W(v_1,u)$. Analogously, since $C_1$ is gated and $v\in C_1\cap C_3$, the gate of $u$ in $C_3$ must belong to $I(u,v)\subset C_1$. Since the length of $C_1$ is at least 6,
this gate cannot be adjacent to $u$ and $v$, thus $v$ is the gate of $u$ in $C_3$. Since $v_1\in I(u,v)$, again we conclude that $C_3\subset W(v_1,u)$. Hence
$E_f$ does not cross the cycles $C_2$ and $C_3$.  Let $G':=\pi_f(G)$ and $C'_i:=\pi_f(C_i)$, for $i=1,2,3$. Since each $C_i, i=1,2,3$ is gated, by Lemma~\ref{contraction_gated} each $C'_i,i=1,2,3$ is a gated subgraph of $G'$ and by Lemma~\ref{convex_hull} $G'$ is the convex hull of $C'_1\cup C'_2\cup C'_3$.
Notice that the three  cycles $C'_1,C'_2,C_3'$ pairwise intersect in the same edges as the cycles $C_1,C_2,C_3$ and all three in the vertex $v$.

Since $G'\in \mathcal{F}(Q_3^-),$ by induction assumption $G'$ is isomorphic to the Cartesian product $C \square K_2$, where $C$ is isomorphic to one of $C_1',C_2',C_3'$.
The graph $G$ is obtained from the graph $G'$ by an isometric expansion with respect to the subgraphs $G'_1$ and $G'_2$ of $G'$. By Proposition~\ref{product_of_cycles_expansion} and Lemma~\ref{lemma:half-extended},
\begin{itemize}
\item[(i)] $G$ is a Cartesian product of edges and even cycles or
\item[(ii)] $G'_1$ coincides with $G'$  and $G'_2$ is isomorphic to a subproduct of edges and cycles.
\end{itemize}
 The only convex cycles of length at least 6 in a product of edges and cycles are layers. Therefore in the case (i) $C_1$ must be a layer $L$ in the product. Each of cycles $C_2$ and $C_3$ shares an edge with $C_1$. The only such cycles are 4-cycles between layer $L$ and any other layer adjacent to $L$. Since also $C_2$ and $C_3$ share an edge, they must both be between $L$ and some layer $L'$. Then $L$ and $L'$ form a cell isomorphic to $L \square K_2\cong C_1 \square K_2$ that is the convex hull of $C_1\cup C_2 \cup C_3$.

Finally, assume (ii) holds. Since $G'\cong C\square K_2$, $G'_2$ is either a vertex, an edge, a 4-cycle, a layer isomorphic to $C$, or the whole $G'$. Thus, $G'_2$ intersects no cyclic layer in just two antipodal vertices, i.e., no convex cycle of $G'$ gets extended. A contradiction, since $C_1'$ should be extended.
\end{proof}

We will now 
establish the implication (ii)$\Rightarrow$(iii), while (iii)$\Rightarrow$(ii) trivially holds.
\begin{proposition} \label{three_cells} If $G$ is a partial cube in which cells are gated and which satisfies the 3CC-condition, then $G$ satisfies the 3C-condition.

\end{proposition}
\begin{proof}
 Since the properties of $G$ are closed under restriction, without loss of generality we consider $G=\conv(X_1\cup X_2\cup X_3)$. Since cells are gated in $G$, by Lemma~\ref{product_cycles_gated} $X_{123}$ is a subproduct of $X_{12},X_{23},X_{13}$ and $X_{ij}$ is a subproduct of $X_i$ and $X_j$ for all $i,j\in\{1,2,3\}$, where in all cases the factors of the subproducts are vertices, edges,  or factors of the superproducts.  Indeed by the conditions on the dimensions, the subproducts all have the same factors than their superproducts except that either one edge-factor from the superproduct is a vertex in the subproduct or one cyclic factor from the superproduct is an edge in the subproduct.
 This gives that any $v\in X_{123}$ has exactly one neighbor $v_{ij}\in X_{ij}\setminus X_{123}$ for all $i,j\in\{1,2,3\}$. Now, since the cells $X_{1},X_{2},X_{3}$ are products, a path of the form $(v_{ij},v,v_{ik})\subset X_i$ is contained in the unique convex cycle $C_i$ of $X_i$ accounting for the two supplementary dimensions of $X_i$ compared to $X_{123}$, for all $\{i,j,k\}=\{1,2,3\}$. Since $G$ satisfies the 3CC-condition, $\conv(C_1\cup C_2\cup C_3)$ is a cell $X_v$ of $G$ isomorphic to $C_i\square K_2$ and $C_j\cong C_k$ are 4-cycles, for some $\{i,j,k\}=\{1,2,3\}$.
 Moreover, by the product structure of $X_i$ the $\Theta$-classes of $C_i$ and their order on $C_i$ do not depend on the choice of $v\in X_{123}$, for all $i\in\{1,2,3\}$. Therefore, for all $v,w\in X_{123}$ and some $i\in\{1,2,3\}$ we have $X_v\cong X_w\cong C_i\square K_2$, where corresponding edges are in the same $\Theta$-class of $G$. Since they are separated by $\Theta$-classes crossing $X_{123}$, for different $v,w\in X_{123}$, the cells $X_v$ and $X_w$ are disjoint and by construction the union of all of them covers $X_1\cup X_2\cup X_3$. We obtain that $X_1\cup X_2\cup X_3\subseteq X_{123}\square C_i\square K_2$, which is a $(k+3)$-dimensional cell of $G$, thus gated and thus convex. Since $G=\conv(X_1\cup X_2\cup X_3)$, we get $\conv(X_1\cup X_2\cup X_3)\cong X_{123}\square C_i\square K_2$, which establishes the claim.
\end{proof}

To show (ii)$\Rightarrow$(i), in Proposition~\ref{condition(ii)} we prove that the class of partial cubes satisfying (ii)
is minor-closed.  Since $Q_3^-$ does not satisfy the 3CC-condition, the graphs satisfying (ii)  cannot be contracted to $Q_3^-$, thus they are hypercellular.

%

\begin{proposition} \label{condition(ii)}
The family of partial cubes having gated cells and satisfying the 3CC-condition is a pc-minor-closed family.
\end{proposition}

\begin{proof}
If a condition of the proposition is violated for a convex subgraph of a partial cube $G$, then it is also violated for $G$. Therefore the family in question is closed under restrictions.

Let now $G$ be a partial cube satisfying the conditions of the proposition and let $G'$ be a contraction of $G$ along some equivalence class $E_f$. Pick a cell $X'$ in $G'$. By Lemma~\ref{convex_expansion}, the expansion $X$ of $X'$ is a convex subgraph of $G$, thus $X$ also satisfies  the conditions of the proposition. By the 3CC-condition, $X$ has no convex subgraph isomorphic to a half-expanded cycle. Thus Proposition~\ref{product_of_cycles_expansion} provides us with the structure of $X$; in particular, $X$ includes a cell $Y$ such that $\pi_f(Y)=X'$. Since the cells of $G$ are gated, by Lemma~\ref{contraction_gated}, $X'$ is gated. Therefore the cells of $G'$ are gated.

Now, let $C_1',C_2',C_3'$ be three convex cycles of $G'$ such that any two cycles $C_i',C_j'$, $1\le i<j\le 3$, intersect in an edge $e'_{ij}$ and the three cycles intersect in a vertex $v'$. Let $G'_1$ and $G'_2$ be the subgraphs of $G'$ with respect to which we perform  the expansion of $G'$ into $G$. By Proposition~\ref{product_of_cycles_expansion} (or directly using the fact that $C_1',C_2',C_3'$ are convex cycles), for each $C'_i$, $i\in \{1,2,3\}$, we have one of the following three options: (a) either both $G'_1\cap C'_i$ and $G'_2\cap C'_i$ coincide with $C'_i$, or (b) one of $G'_1\cap C'_i$ and $G'_2\cap C'_i$ is the whole cycle $C'_i$ and other is an edge, a vertex, or empty, or (c) both $G'_1\cap C'_i$ and $G'_2\cap C'_i$ are paths corresponding to halves of $C_i'$ with intersection in two antipodal vertices of $C'_i$. Using this trichotomy, we divide the analysis in the following cases.

\medskip\noindent
{\bf Case 1.} For $G'_1$ or $G'_2$, say for $G'_1$, and for at least two of the three cycles $C'_1,C'_2, C'_3,$ say for $C'_1,C'_2$,  we  have $G'_1\cap C'_i=C'_i$ and $G'_1\cap C_j'=C'_j$.

\medskip
Then the edges $e'_{13}$ and $e'_{23}$ are in $G'_1$. Thus either $G'_1\cap C'_3=C'_3$ or $G'_1\cap C'_3$ is a half of $C'_3$ that includes $e'_{13}$ and $e'_{23}$. Then in the expansion we have 3 convex cycles $C_1,C_2,C_3$ pairwise sharing an edge and a vertex in the intersection of all three, such that $\pi_f(C_i)=C_i'$ for $i\in \{1,2,3\}$. Since the 3CC-condition holds in $G$, two of the cycles $C_1,C_2,C_3$ are 4-cycles and all three are included in a cell $X$ isomorphic to $C_k \square K_2$ where $C_k$ is the third cycle. Since a contraction can only shorten the cycles, at least two of the cycles $C_1',C_2',C_3'$ are 4-cycles and the convex hull of all three must be included in a cell $\pi_f(X)$. The only contraction of $C_k \square K_2$ that has at least three convex cycles is isomorphic to $C'_k \square K_2$.

\medskip\noindent
{\bf Case 2.} For $G'_1$ or $G'_2$, say for $G'_1$, among $C'_1,C'_2, C'_3$ there exists a unique cycle, say $C'_1$, such that  $G'_1\cap C_1'=C_1'$.

\medskip
By symmetry and in view of Case 1, for two other cycles $C'_2$ and $C'_3$ we have only one of the following options: (1) either $G'_2\cap C'_j=C'_j$ for exactly one $j\in \{2,3\}$
or (2) for all $j\in \{2,3\}$ and $k\in \{1,2\}$ we have that $G'_k\cap C'_j$ is a half of the cycle $C'_j$.

First consider the option (2). By properties of $C_1'$, in the half-space $G_1$ of $G$ corresponding to $G'_1$ in the expansion there exists a convex cycle $C_1$ of the same length as $C_1'$ such that $\pi_f(C_1)=C_1'$. Moreover $C_2'$ and $C_3'$ get expanded to convex cycles $C_2,C_3$ each sharing exactly one edge with $C_1$ and having one vertex in the intersection of all three. Since the cells of $G$ are gated, the cycles $C_2,C_3$ are gated. The cycles $C_2,C_3$ share at least one edge. If a vertex of $e'_{23}$ is in $G'_1\cap G'_2$, then $C_2$ and $C_3$ share 2 edges, which impossible because $C_2$ and $C_3$ are gated. By the 3CC-condition, two of $C_1,C_2,C_3$ are 4-cycles, which is impossible because in $G'$  one of those 4-cycles $C_i$ will get contracted to an edge and not to the cycle $C'_i$.

Now consider the option (1) that $G'_2\cap C'_2=C'_2$ and both $G'_1\cap C'_3, G'_2\cap C'_3$ are halves of $C_3'$. Since $v\in G'_1\cap G'_2$, the antipode $u$ of $v$ in $C'_3$ also belongs to $G'_1\cap G'_3$. Hence $C'_3$ gets expanded to $v$, its antipode  $u$, and $e_{12}'\in G'_1\cap G'_2$. Therefore to ensure that we do not have $G'_k\cap C_j' =C_j'$ for some $k\in \{1,2\}$ and 
$j\in \{1,2,3\}$, we must have $G'_1\cap C'_2=e_{12}'$ and $G'_2\cap C'_1=e_{12}'$. Now the expansion of $C_1'\cup C_2'\cup C_3'$ has 4 convex cycles: the expansion $C_3$ of $C_3'$, the convex cycles $C_1$ and $C_2$ that get mapped to $C_1'$ and $C_2'$ by the contraction and a 4-cycle $C_4$ between $C_1$ and $C_2$. Cycles $C_1,C_3,C_4$ pairwise intersect in three different edges  and all
have a common vertex, thus their convex  closure $X$ is isomorphic to $C_3\square K_2$ (since the cycle $C_3$ must have length at least 6). This proves that $C_1$ is a 4-cycle. Let $C_5$ be the fourth cycle, sharing edges with $C_3$ and $C_4$ different from $C_1$. Then $C_5$ shares two edges with $C_2$ which is possible only if $C_2=C_5$. Thus we see that again $\pi_f(X)$ is a gated cell including $C_1',C_2',C_3'$.

\medskip\noindent
{\bf Case 3.} For every $i\in \{1,2,3\}$, $G'_1\cap C'_i$ and $G'_2\cap C'_i$ are halves of $C'_i$ intersecting in two antipodal vertices of $C'_i$.

\medskip
If $G'_1\cap G'_2$ contains any vertex of $e_{ij}'$ for $i,j\in \{1,2,3\}$, then there exist convex cycles $C_i$ and $C_j$ in $G$ that share two edges, which is impossible. Thus $C_1',C_2',C_3'$ get extended to cycles $C_1,C_2,C_3$ pairwise sharing an edge and a vertex in common. Then two of them must be 4-cycles, which is not the case because  then two of them get contracted to edges. We have proved that the 3CC-condition also holds for $G'$, thus the class we consider is closed under contractions. This finishes the proof.
\end{proof}

%
%


The remaining part of this section is devoted to the proof of the equivalence (i)$\Leftrightarrow$(iv). For an equivalence class $E_f$ of $\Theta$, we denote by $N(E_f)$ the \emph{carrier} of $f$, i.e., the subgraph of $G$ which is the union of all cells of $G$ which are crossed by $E_f$. The carrier $N(E_f)$ splits into its positive and negative parts $N^+(E_f):=N(E_f)\cap H^+_f$ and $N^-(E_f):=N(E_f)\cap H^-_f$.

\begin{lemma} \label{carrier_projection} Let $G$ be a hypercellular graph and $e,f\in \Lambda, e\ne f$. Then $\pi_e(N(E_f))$ is the carrier of $E_f$ in
$\pi_e(G)$.
\end{lemma}
\begin{proof}
Let $Y\in N(E_f)$ be a cell of $G$. Since contractions of products are products, $\pi_e(Y)$ is a product of edges and even cycles in $\pi_e(G)$ and clearly crosses $E_f$. Furthermore, since $Y=\conv(C)$ for a cycle $C$ in $G$, we have by Lemma~\ref{convex_hull}, that $\pi_e(Y)\subseteq \conv(\pi_e(C))$. Since $\pi_e(G)$ is hypercellular, $\conv(\pi_e(C))$ is a cell by Theorem~\ref{mthm:cells}. Thus, $\pi_e(Y)$ is convex in $\pi_e(G)$ by Lemma~\ref{product_cycles_gated}. Therefore $\pi_e(Y)$ is a cell of $\pi_e(N(E_f))$.

Conversely, let $Y$ be a cell in the carrier of $f$ in $\pi_e(G)$ and $Y'$ be its expansion with respect to $e$. By Lemma~\ref{convex_expansion}, $Y'$ is convex and by Proposition~\ref{product_of_cycles_expansion} $Y'$ is either a product of cycles and thus a cell of $N(E_f)$, or $Y'$ consists of two cells $Y'', Y'''$ separated by $E_e$, where say $Y''$ is isomorphic to $Y$. Since $f$ crosses $Y''$, $Y''$ is in $N(E_f)$ and $Y$ arises as its contraction, so we are done.
\end{proof}

\begin{lemma} \label{two_cycles}  Let $G$ be a hypercellular graph. 
Then any two cells $Y',Y''$ of $X(G)$ either are disjoint or intersect in a cell of $X(G)$.
\end{lemma}

\begin{proof}
Let $Y',Y''$ be two arbitrary intersecting cells of $G$. Let $Y_0:=Y'\cap Y''$. Since $Y'$ and $Y''$ are gated subgraphs of $G$, $Y_0$ is also gated. In particular, $Y_0$
 is a gated subgraph of $Y'$. Since $Y'$ is a product of edges and cycles $F_1\square\cdots \square F_m$, by Lemma~\ref{product_cycles_gated}, $Y_0$ is a product $F'_1\square \cdots\square F'_m$, where each $F'_i$ is a  vertex, an edge, or the whole factor $F_i$. Hence $Y_0$ is a cell and we are done. Now suppose that some $F'_i\cong P=(x, \ldots, y, \ldots, z)$ is a path of length $\ge 2$ within the cyclic factor $F_i$. Since $P$ is convex, the length of $P$ must be less than half of the length of $F_i$. Thus the antipodal vertex of $y$ in the $F_i$, say $y'$, is not in $P$. Now, $y'$ cannot have a gate $w$ in $P$, since if $w$ is between $x$ and $y$ there is no shortest path from $y'$ through $w$ to $z$.  Symmetrically, if $w$ is between $y$ and $z$ there is no shortest path from $y'$ through $w$ to $x$.
\end{proof}
%

\begin{lemma} \label{two_cells_carrier} Let $G$ be a hypercellular graph  and $f\in \Lambda$. If two cells $Y',Y''$ of $N(E_f)$ intersect, then they share an edge of $E_f$. 
\end{lemma}

\begin{proof}  Let $y\in Y'\cap Y''$ and suppose without loss of generality that $y\in H^-_f$. Since $Y'\in N(E_f)$, there exists an edge $u'v'\in E_f$ with $u',v'\in Y'$. Suppose $u'\in H^-_f$ and $v'\in H^+_f$. If $v'\in Y''$, then $u'\in I(v',y)\subset Y''$ by convexity of $Y''$, thus the edge $u'v'$ belongs to $Y'\cap Y''$ and we are done. So, suppose $v'\notin Y''$. Let $v$ be the gate
of $v'$ in $Y''$ and let $x$ be a vertex of $Y''\cap H^+_f$ (such a vertex exists because $Y''\in N(E_f)$). Since $v\in I(v',x)$ and $H^+_f$ is convex, we conclude that $v\in H^+_f$. Since $y\in H^-_f$, on any shortest path $P$ from $v$ to $y$ we will meet an edge $v''u''$ of $E_f$. Since $v'',u''\in I(v,y),$ $v,y\in Y''$, and $Y''$ is convex, the edge $v''u''$ belongs to $Y''$. On the other hand, since $v\in I(v',y)$ and $Y'$ is  convex, we conclude that the edge $v''u''$ also belongs to $Y'$.
\end{proof}

%


\begin{proposition} \label{carrier} For any equivalence class $E_f$ of a hypercellular graph $G$, the carrier $N(E_f)$ is a gated subgraph of $G$. Therefore, $N^+(E_f)$ is gated in the halfspace $H^+_f$, $N^-(E_f)$ is gated in $H^-_f$, and the extended halfspaces $H^+_f\cup N(E_f)$ and $H^-_f\cup N(E_f)$ are gated in $G$.
\end{proposition}

\begin{proof} First, since by Lemma~\ref{partial-cube-convex-cycle} the relations $\Theta$ and $\Psi^*$ coincide, $N^+(E_f)$, $N^-(E_f)$ and consequently $N(E_f)$ are connected subgraphs of $G$.

Through Claims~\ref{claim:isom},~\ref{claim:conv}, and ~\ref{claim:cycles} we will prove that $N^+(E_f)$ is convex. 
Suppose that $N^+(E_f)$ is not convex. Choose two vertices $y,z\in N^+(E_f)$ with minimal distance $k(y,z):=d_{N^+(E_f)}(y,z)$ that can be connected by a shortest path $R$ of $G$ outside $N^+(E_f)$. Let $P$ be a shortest $y,z$-path in $N^+(E_f)$. Let us prove that $P$ is a shortest path of $G$. If this was not the case, we could replace $y$ by its neighbor $y'$ in $P$. But from the minimality in the choice of $y,z$, we conclude that $I(y',z)\subseteq N^+(E_f)$. Thus, the subpath of $P$ between $y'$ and $z$ is a shortest path of $G$. Now, since  $y\notin I(y',z)$, we have $z\in W(y',y)$, yielding that $P$ is a shortest path of $G$. Again by the choice of $y,z$, we conclude that $P$ and $R$ intersect only in their common endvertices $y,z$.

\begin{claim}\label{claim:isom}
 Any shortest path between a vertex of $P$ and a vertex of $R$ passes via $y$ or $z$. In particular, $C:=P\cup R$ is an isometric cycle of $G$.
\end{claim}
\begin{proof}
We claim that if $Q$ is a shortest path connecting two interior vertices $p$ of $P$ and $r$ of $R$, then $Q$ passes via $y$ or $z$. Suppose that this is not the case. Then we can find a shortest path $Q=(p:=q_0,q_1,\ldots,q_{k-1},q_k:=r)$ between  two interior vertices $p$ of $P$ and $r$ of $R$ such that $Q\cap C=\{ p,r\}$.
Since $p,r\in I(y,z)$ and $Q\subset I(p,r)$, we conclude that $Q\subset I(y,z)$, because intervals of partial cubes are convex. This yields $q_1\in I(p,y)\cup I(p,z)$, since otherwise $y,z\in W(p,q_1)$ and $q_1\in I(y,z)$ contradict the convexity of $W(p,q_1)$. Since $k(p,y)<k(z,y)$ and $k(p,z)<k(z,y)$, we conclude that $I(p,y)\cup I(p,z)\subset N^+(E_f)$, giving $q_1\in N^+(E_f)$. We can iterate this argument by first replacing $p$ by $q_1$ and $q_1$ by $q_2$, etc., and obtain that all vertices $q_1,q_2,\ldots, q_{k-1},q_k=r$ belong to $N^+(E_f)$ and $k(q_i,y)<k(y,z), k(q_i,z)<k(y,z)$. In particular, $r\in N^+(E_f)$ and $k(r,y)<k(y,z), k(r,z)<k(y,z)$, thus by our assumption $R\subset I(r,y)\cup I(r,z)\subset N^+(E_f)$. This contradiction shows that the path $Q$ does not exist, i.e., any shortest path between a vertex of $P$ and a vertex of $R$ passes via $y$ or $z$. In particular, this implies that $C=P\cup R$ is an isometric cycle of $G$. 
\end{proof}

\begin{claim}\label{claim:conv}
$C=P\cup R$ is a convex cycle of $G$.
\end{claim}
\begin{proof} We proceed as in the proof of Lemma~\ref{partial-cube-convex-cycle}. If $C$ is not convex, then by Claim~\ref{claim:isom} there exist
two vertices $p,p'\in P$ connected by a shortest path $P'$ which intersects $P$ only in $p,p'$ or there exist two vertices $r,r'\in R$ connected by a shortest path $R'$ which intersects
$R$ only in $r,r'$.  Let $P''$ be the subpath of $P$ between $p$ and $p'$ in the first case and let $R''$ be the subpath of $R$ between $r$ and $r'$ in the second case.
Let $C'$ be the cycle obtained from $C$ by replacing the path $P''$ by $P'$ in the first case and let $C''$ be the cycle obtained from $C$ by replacing the path $R''$ by $R'$ in the second case.
If the first case occurs and $\{ p,p'\}\ne \{y,z\},$ then $k(p,p')<k(y,z)$, whence $P'\subset N^+(E_f)$. Therefore applying Claim~\ref{claim:isom} to the cycle $C'$ instead of $C$, we conclude that $C'$ is an isometric cycle of $G$. Analogously, if $\{ r,r'\}\ne \{ y,z\}$, then no vertex of $R'$ belongs to $N^+(E_f)$. Indeed, if say $w\in R'\cap N(E_f)$, then $k(w,y)<k(y,z), k(w,z)<k(y,z)$ and by minimality the vertices $r$ and $r'$ belong to $N(E_f)$, contrary to the assumption that $R\cap N(E_f)=\{ y,z\}$. Again applying Claim~\ref{claim:isom} to the cycle $C''$ instead of $C$, we conclude that $C''$ is an isometric cycle of $G$. Finally, if  $\{ p,p'\}=\{ y,z\}$, then $P''=P$ and one can see that either
$P'\subseteq N^+(E_f)$ and $C'$ is an isometric cycle of $G$ or $P'\cap N^+(E_f)=\{ y,z\}$ and we redefine $C'$ as
the cycle formed by $P$ and $P'$, which is isometric by Claim~\ref{claim:isom}. Similarly, if  $\{ r,r'\}=\{ y,z\}$, then $R''=R$ and we can suppose that $C''$ is an isometric cycle of $G$ sharing
with $C$ either the path $P$ or the path $R$.  Consequently, in all cases we derive a new isometric cycle of $G$ ($C'$ or $C''$) obtained by replacing either a subpath $P''$ of $P$ by $P'$ or replacing a subpath $R''$ of $R$ by $R'$. Suppose without loss of generality that we are in the first case, i.e., the new isometric cycle is $C'$.

Let $x''$ be a vertex of $P''$ different from $p,p'$. Let $x$ be the opposite of $x''$ in the cycle $C$ and let $x'$ be the opposite of $x$ in the cycle $C'$. Since $C$ and $C'$ are isometric cycles of the same length of $G$,  $x'$
is a vertex of $P'$ different from $p,p'$. Then $p,p'\in I(x,x')$, thus by convexity of $I(x,x')$ we obtain that $x''\in I(x,x')$. But this is impossible because $x'$ and $x''$ have the
same distance to $x$ because they are opposite to $x$ in $C$ and $C'$, respectively, and $C$ and $C'$ have the same length. This proves that  $C$ is convex.
\end{proof}

Since the cells of $N(E_f)$ are convex, the path $P\subset N^+(E_f)$ of $C$ cannot be contained in a single cell. Thus there exist two consecutive edges $e'=uv$ and $e''=vw$ of $P$ and two cells $Y',Y''$ of $N(E_f)$ such that $u\in Y'\setminus Y'', w\in Y''\setminus Y',$ and $v\in Y'\cap Y''$. By the following claim this is impossible.


\begin{claim}\label{claim:cycles}
If there exist two cells $Y',Y''$ of $N(E_f)$ and a convex cycle $C$ with edges $e'=uv,e''=vw$ on it, such that $u\in Y'\setminus Y'', w\in Y''\setminus Y',$ and $v\in Y'\cap Y''$, then  there exists a cell in $N(E_f)$ that includes $C$.
\end{claim}

\begin{proof} By Lemma~\ref{two_cycles}, the intersection $Y' \cap Y''$ is a  product of edges and cycles. Since $Y',Y'' \in N(E_f)$,  Lemma~\ref{two_cells_carrier} yields that $Y' \cap Y''$ contains at least one edge from $E_f$. Let $F$ be a factor of $Y' \cap Y''$ and $L$ be a corresponding layer that is crossed by $E_f$. We will establish the claim by proving that $e',e''$ lie in a cell of $N(E_f)$, that is isomorphic to $L \, \square \, C$. Pick any edge $e=vz$ in $Y' \cap Y''$ that lies in the layer $L$. 

First assume that $e$ and $e'$ lie in the same layer of $Y'$. The factor $F$ is an edge or an even cycle, but it must be a strict subset of a factor of $Y'$ since $e'\notin Y' \cap Y''$. Thus $L$ is isomorphic to an edge and this edge must be in $E_f$. In particular, $e\in E_f$. Since $Y'$ and $Y''$ are products of edges and cycles, there exists a convex cycle $C'$ of $Y'$ passing via the edges $e'$ and $e$ and there exists a convex cycle $C''$ of $Y''$ passing via the edges $e$ and $e''$. By Lemma~\ref{three_cycles}, $C$, $C'$, and $C''$ are contained in a cell $Y$ of $G$. Since $e\in E_f$, this cell is in $N(E_f)$.

By symmetry, we are left with the case that $e$ and $e'$ as well as $e$ and $e''$ lie in different layers of $Y'$ and $Y''$, respectively. Consequently, there exists a 4-cycle $C'=(u,v,z,u')$ of $Y'$ passing via the edges $e'$ and $e$ and a 4-cycle $C''=(w,v,z,w')$ of $Y''$ passing via the edges $e$ and $e''$. By Lemma~\ref{three_cycles}, $C$, $C'$, and $C''$ are contained in a cell $Y\cong C  \, \square \, K_2$ of $G$, where $e$ lies in a layer corresponding to the factor $K_2$. If $e\in E_f$, then $Y$ is a cell in $N(E_f)$ and we are done. If $e \notin E_f$, then $L$ is isomorphic to an even cycle. Consider $z$, and let $z'$ be its neighbor, different from $v$, that lies in $Y' \cap Y''$ in $L$. Since $z$ is in $Y$, $z$ is incident to a cycle $C'$ isomorphic to $C$ and lies on the path $(u',z,w')$ of $C'$. Considering $C'$ and edges $u'z$, $zw'$, and $zz'$, we can as before with $C$, $e'$, $e''$, and $e$, obtain a cell $Z$ isomorphic to $C  \, \square \, K_2$. The union of $Y$ and $Z$ is isomorphic to $C  \, \square \, P_3$, where $P_3$ is the path on 3 vertices. Inductively picking neighbors in the layer $L$ we obtain a graph isomorphic to $C \, \square \, F$, that contains $C$.
\end{proof}

We have sown that $N^+(E_f)$ and symmetrically $N^-(E_f)$ are convex. To see that $N(E_f)$ is convex, pick two vertices $x\in N^+(E_f), y\in N^-(E_f)$, a shortest $(x,y)$-path $R$ and a vertex $z$ of $R$. Since $R$ connects a vertex of $H^+_f$ with a vertex of $H^-_f$, necessarily $R$ contains an edge $x'y'$ in $E_f$, say $x'\in H^+_f$ and $y'\in H^-_f$. The vertex $z$ belongs to one of the two subpaths of $R$ between $x$ and $x'$ or $y'$ and $y$, say the first. Then $z\in I(x,x')$. Since $x,x'\in N^+(E_f)$ and $N^+(E_f)$ is convex, we conclude that $z\in N^+(E_f)\subset N(E_f)$, showing that the carrier $N(E_f)$ is convex.

Now, suppose that $G$ is a minimal graph in $\mathcal{F}(Q_3^-)$ containing a non-gated carrier $N(E_f)$. Since $N(E_f)$ is convex and by Lemma~\ref{carrier_projection} any contraction
$\pi_e(N(E_f))$ for $e\in \Lambda, e\ne f,$ is the carrier of $E_f$ in $G'=\pi_e(G)$ and thus is gated in $G'$, by Proposition~\ref{thm:closure} there exist two vertices $x_1,x_2\in N(E_f)$ with $d_G(x_1,x_2)=2$ and a vertex $v\notin N(E_f)$ at distance 2 from $x_1,x_2$, such that the vertices $v,x_1,x_2$ do not contain a common neighbor. Since $G\in \mathcal{F}(Q_3^-)$, the last condition
implies that the convex hull of $v,x_1,x_2$ is a 6-cycle $C_1$. Let $u$ be the unique common neighbor of $x_1$ and $x_2$ in $C$. Since $N(E_f)$ is convex, $u$ also belongs to $N(E_f)$. Namely, if say $v\in H^+_f$, then $x_1,u,x_2\in N^+(E_f)$. Since by Proposition~\ref{convex_hull_cycle} each cell of $G$ is gated, the vertices $x_1$ and $x_2$ cannot belong to a common cell. Thus there exist
two cells $Y',Y''$ of $N(E_f)$ such that the edge $x_1u$ belongs to $Y'$ and $ux_2$ belongs to $Y''$. By Claim~\ref{claim:cycles}, there exists a cell $Y$  of the carrier $N(E_f)$ that includes $C$, contrary to the assumption that the vertex $v$ of $C$ does not belong to $N(E_f)$. This establishes that $N(E_f)$ is gated. By Lemma~\ref{contraction_gated} also $N^+(E_f)$ is gated in $H^+_f$ and $N^-(E_f)$ is gated in $H^-_f$. Consequently, the extended halfspaces $H^+_f \cup N(E_f)$ and $H^-_f \cup N(E_f)$ are gated in $G$.
\end{proof}

Now, we are ready to prove the following result (the equivalence (i)$\Leftrightarrow$(iv) of Theorem~\ref{mthm:amalgam}):

\begin{theorem}\label{thm:amalgam} A partial cube $G$ is hypercellular if and only if each finite convex subgraph of $G$
can be obtained by gated amalgams from Cartesian products of edges and even cycles.
\end{theorem}

\begin{proof} First suppose that a finite graph $G$ is obtained by gated amalgam from two graphs $G_1,G_2\in \mathcal{F}(Q_3^-)$. Suppose by way of contradiction that $G\notin \mathcal{F}(Q_3^-)$
and suppose that $G$ is a minimal such graph. Then any proper convex subgraph $H$ of $G$ is either contained in one of the graphs $G_1,G_2$ or is the gated amalgam of $H\cap G_1$ and $H\cap G_2$, thus $H\in \mathcal{F}(Q_3^-)$ by minimality of $G$. Thus there exists a sequence of contractions of $G$ to the graph $Q^-_3$. Let $E_f$ be the first such contraction, i.e., the graph $G':=\pi_f(G)$ does not belong to $\mathcal{F}(Q_3^-)$. On the other hand, by Lemma~\ref{contraction_gated}, $G'_1:=\pi_f(G_1)$ and $G'_2:=\pi_f(G_2)$ are gated subgraphs of $G'$. Moreover, $G'_1$ and $G'_2$ belong to $\mathcal{F}(Q_3^-)$ because $G_1$ and $G_2$ belong to $\mathcal{F}(Q_3^-)$ and $\mathcal{F}(Q_3^-)$ is closed by contractions. As a result we obtain that the graph $G'\notin \mathcal{F}(Q_3^-)$ is the gated amalgam 
of the graphs $G'_1,G'_2\in \mathcal{F}(Q_3^-)$, contrary to the minimality of $G$. This establishes that the subclass of $\mathcal{F}(Q_3^-)$ consisting of finite graphs from $\mathcal{F}(Q_3^-)$ is closed by gated amalgams.

Conversely, suppose that $G$ is an arbitrary finite convex subgraph of a graph from $\mathcal{F}(Q_3^-)$. We follow the schema of proof of implication (3)$\Rightarrow$(4) of~\cite[Theorem 1]{BaCh_cellular}. If $G$ is a single cell, then we are done. Otherwise, we claim that $G$ is a gated amalgamation of two proper gated subgraphs $G_1$ and $G_2$.

First suppose that there exist two disjoint maximal cells $Y'$ and $Y''$. Let $y'\in Y'$ and $y''\in Y''$ be two vertices realizing the distance $d(Y',Y'')=\min \{ d(x,z): x\in Y', z\in Y''\}$. Since $Y'\cap Y''=\emptyset$, necessarily $y'\ne y''$. Since $Y'$ and $Y''$ are gated, from the  choice of $y',y''$  it follows that $y'$ is the gate of $y''$ in $Y'$ and $y''$ is the gate of $y'$ in $Y''$. Let $y$ be a neighbor of $y'$ on a shortest path between $y'$ and $y''$. Suppose that the edge $y'y$ belongs to the equivalence class $E_f$. Notice that $y''$ is also the gate of $y$ in $Y''$ and $y'$ is the gate of $y$ in $Y'$. Therefore $Y'\subseteq W(y',y)=H^+_f$ and $Y''\subseteq W(y,y')=H^-_f$. Consequently, $Y'$ and $Y''$ are not contained in the carrier $N(E_f)$, thus $H^+_f\setminus N(E_f)$ and $H^-_f\setminus N(E_f)$ are nonempty. By Proposition~\ref{carrier}, $N(E_f),$ $H^+_f\cup N(E_f),$ and $H^-_f\cup N(E_f)$ are gated subgraphs of $G$, thus $G$ is the gated amalgam of $H^+_f\cup N(E_f)$ and $H^-_f\cup N(E_f)$ along the common gated subgraph $N(E_f)$.

Thus further we may suppose that all maximal cells of $G$ pairwise intersect. Since they are gated and $G$ is finite, by the Helly theorem for gated sets~\cite[Proposition 5.12 (2)]{VdV}, the maximal cells of $G$ intersect in a non-empty cell $X_0$.

\begin{claim} \label{proper_carrier} There exists an equivalence class $E_f$ of $G$ such that the carrier $N(E_f)$ of $E_f$ does not contain all maximal cells of $G$ and $E_f$ contains an edge $uv$ with $v\in X_0$ and $u\notin X_0$. Moreover, all maximal cells of the carrier $N(E_f)$ contain the edge $uv$.
\end{claim}

\begin{proof} By definition, $X_0$ is a proper face of each maximal cell $X$ of $G$. Therefore, there exists an edge $uv$ with $v\in X_0$ and $u\in X\setminus X_0$. Suppose that $uv$ belongs to the equivalence class $E_f$ of $G$. Then $X_0\subseteq W(v,u)$. Notice that $X$ belongs to the carrier $N(E_f)$ of $E_f$. Since $u\notin X_0$, there exists a maximal cell $X'$ such that $u\notin X'$. Since $v\in X'$, we assert that $X'$ does not belong to $N(E_f)$. Indeed, suppose $E_f$ contains an edge $u''v''$ with both ends in $X'$. Assume without loss of generality that $W(u,v)=W(u'',v'')$ and $W(v,u)=W(v'',u'')$. Since $u\in I(v,u'')$ and $v,u''\in X'$, by the convexity of $X'$ we conclude that $u\in X'$, a contradiction.  This shows that $N(E_f)$ consists of all maximal cells containing the edge $uv$.
\end{proof}

Let $E_f$ be an equivalence class of $G$ as in Claim~\ref{proper_carrier}, in particular, $uv$ is an edge of $E_f$ with $v\in X_0$ and $u\notin X_0$. Let $X_1,\ldots, X_k$ be the maximal cells of $G$ containing the edge $uv$. By the second assertion of Claim~\ref{proper_carrier}, $N(E_f)$ coincides with the union $\bigcup_{j=1}^k X_j$. Let $X_{k+1},\ldots, X_m$ be the remaining  maximal cells of $G$, i.e., the maximal cells not containing the vertex $u$ (such cells exist by the choice of $E_f$). Set $Y:=\bigcup_{i=k+1}^m X_i$ and notice that by the choice of $X_0$ we have $Y\subseteq W(v,u)$.

Let $Z$ be the subgraph of $G$ induced by the intersection of $N(E_f)$ with $Y$, i.e., $Z=\bigcup_{i=k+1}^m Z_i,$ where $Z_i:=N(E_f)\cap X_i$, $i=k+1,\ldots,m$.   By Proposition~\ref{carrier},  $N(E_f)$ is gated. Since by Proposition~\ref{convex_hull_cycle} each cell $X_i$ of $Y$  is also gated, each $Z_i$ is gated, and thus is a face of $X_i$, $i=k+1,\ldots,m$, by Lemma~\ref{product_cycles_gated}.


Now we define a gated subgraph $Z^*$ of $G$, which extends $Z$ and \emph{separates} $N(E_f)$ from $Y$, i.e., it contains their intersection and there is no edge from $N(E_f)\setminus Z^*\neq\emptyset$ to $Y\setminus Z^*\neq\emptyset$.  Each maximal cell $X_j$ in $N(E_f)$ is a Cartesian product of edges and even cycles, say $X_j=F_1\square \cdots \square F_{p}$. Let $L_j$ be the layer of $X_j$ containing the edge $uv$.  Suppose that $L_j=\{ v_1\}\square \cdots \{ v_{l-1}\}\square F_{l}\square \{ v_{l+1}\}\square \cdots\square \{ v_{p}\}$, where $F_{l}$ is the $l$th factor of $X_j$ and $v_s$ is a vertex of the factor $F_s, s\ne l$.  If $L_j=uv$, i.e. $L_j$ comes from an edge-factor $F_{l}=u'_jv'_j$, then set
$Z^*_j:=F_1\square\cdots \square F_{l-1}\square \{ v'_j\}\square F_{l+1}\square \cdots \square F_p.$
Since $u\notin Z^*_j$, $Z^*_j$ is a proper gated subgraph of $X_j$.  
Now, suppose that $L_j$ comes from a cyclic factor $F_l$ of $X_j$.
Let $vw_j$ be the  edge of $L_j$ incident to $v$ and different from $uv$. Suppose that the edges $uv$ and $vw_j$ of $L$
come from the edges $u'_jv'_j$ and $v'_jw'_j$ of $F_l$, respectively. Set
$Z^*_j:=F_1\square\cdots \square F_{l-1}\square \{ v'_j,w'_j\}\square F_{l+1}\square \cdots \square F_p.$
Again, since $u\notin Z^*_j$, $Z^*_j$ is a proper gated subgraph of $X_j$. Equivalently, $Z^*_j$ is the subgraph of $X_j$ induced by all vertices of $X_j$
whose gates in the gated cycle $L_j$ is either $v$ or $w_j$.  Notice also that in both cases $Z^*_j$ is a proper face of $X_j$ included in $W(v,u)$.
Finally, set $Z^*:=\bigcup_{j=1}^k Z^*_j$.

\begin{claim} \label{Z*-face2} For each $j=1,\ldots,k$, we have $Z^*\cap X_j=Z^*_j$.
\end{claim}

\begin{proof} By definition, $Z^*_j\subseteq Z^*\cap X_j$. To prove the converse inclusion, it suffices to show that for any $j'\in \{ 1,\ldots, k\}$, $j'\ne j$, we have $Z^*_{j'}\cap X_j\subseteq Z^*_j$.  Consider the layers $L_j$ of $X_j$ and $L_{j'}$ of $X_{j'}$ containing the edge $uv$. Each of them consists either of the edge $uv$ or is a gated cycle of $G$. If $L_j$ is $uv$, then $Z^*_j$ coincides with
$X_j\cap W(v,u)$. Since $Z^*_{j'}\subseteq W(v,u),$ necessarily $Z^*_{j'}\subseteq X_j\subseteq Z^*_j$. Now suppose that $L_j$ is an even cycle.  Suppose by way of contradiction that $Z^*_{j'}\cap X_j$ contains a vertex $x$ not included in $Z^*_j$. Since $x\in W(v,u)$, the gate of $x$ in $L_j$ is a vertex $x'$ of $L_j\cap W(v,u)$ different from $v$. Since $X_{j'}$ is convex, $x',v\in I(x,u)\subset X_{j'}$. Since $X_{j'}$ is gated and contains three different vertices $u,v,x'$ of the gated cycle $L_j$, necessarily $X_{j'}$ contains the entire cycle $L_j$. This implies $L_{j'}=L_j$ and $w_j=w_{j'}$. By definition of $Z^*_{j'}$, we also conclude that $x'=w_{j'}$. Since $x\in X_j$, by definition of $Z^*_j$ we must have $x\in Z^*_j$, contrary to the choice of $x$.
\end{proof}
%

\begin{claim} \label{Z*-face1} For each $i=k+1,\ldots,m$, we have  $Z^*\cap X_i=Z_i$. In particular, $Z^*\cap Y=Z$.
\end{claim}

\begin{proof}

For each maximal cell $X_j,i=1,\ldots,k,$ of $N(E_f)$, consider the intersection $Z_{ji}$ of $X_j$ with each  cell $X_i, i=k+1,\ldots,m,$ of $Y$. From the definition of $Z$ it follows that each $Z_i, i=k+1,\ldots,m,$ can be viewed as the union of all $Z_{ji}$, $j=1,\ldots, k$, thus $Z$ can be viewed as the union of all $Z_{ji}$, $j=1,\ldots, k, i=k+1,\ldots, m$.

Now, let $k+1\le i\le m$. First we prove that  for any $1\le j\le k$ the set $Z_{ji}$
is included in $X_j\cap Z^*$ (which coincides with $Z^*_j$ by Claim~\ref{Z*-face2}).  This is obviously so if the layer $L_j$ is the edge $uv$: in this case, since $X_i\subset W(v,u)$, $Z_{ji}=X_j\cap X_i$ is a subset of $W(v,u)\cap X_j=Z^*_j$. Now, suppose that $L_j$ is an even cycle. Suppose by way of contradiction that $Z_{ji}=X_j\cap X_i$ contains a vertex $x$ whose gate $x'$ in $L_j$ is different from $v$ and $w_j$. Since $x\in W(v,u)$, necessarily $w_j$ and $x'$ belong to the interval $I(x,v)$. Since $x,v\in Z_{ji}$ and $Z_{ji}$ is convex, $w_j,x'\in Z_{ji}$. Since $Z_{ji}$ and $L_j$ are gated and $Z_{ji}\cap L_j$ contains the vertices $u,v,x'$, necessarily $L_j$ must be included in $Z_{ji}$. Since $u\in L_j\setminus Z_{ji}$, we obtained a contradiction. This establishes the inclusion $Z_{ji}\subseteq Z^*\cap X_i\subseteq Z^*$.

We have  $Z^*\cap X_i= (\bigcup_{j=1}^k Z^*_j) \cap X_i=\bigcup_{j=1}^k (Z^*_j \cap X_i)$. By Claim~\ref{Z*-face2}, the latter equals to $\bigcup_{j=1}^k (Z^* \cap X_j \cap X_i)=\bigcup_{j=1}^k (Z^* \cap Z_{ji})=\bigcup_{j=1}^k Z_{ji}$, where the last equation holds by the inclusion established above. Finally, by the definition, $\bigcup_{j=1}^k Z_{ji}=Z_i$.
\end{proof}

\begin{claim} \label{S-gated} Let $S$ be a subgraph of $G$ such that 
the intersection of $S$ with any maximal cell of $G$ is non-empty and gated (i.e., a face by Lemma~\ref{product_cycles_gated}). Then $S$ is a gated subgraph of $G$.
\end{claim}

\begin{proof} Let $X$ be a maximal cell of $G$, $x\in X$ a vertex, and $S^*:=S\cap X$. By our assumptions, $S^*$ is a nonempty face of $X$, thus a gated subgraph of $G$. Let $x^*$ be the gate of $x$ in $S^*$. We assert that $x^*$ is also the gate of $x$ in the set $S$, i.e., for any
vertex $y\in S$, we have $x^*\in I(x,y)$. Suppose that $y$ belongs to a maximal by inclusion cell $R$ in $S$. Let $R_0:=X\cap R$ and let $x_0$ be the gate of $x$ in $R_0$. Since $R\subseteq S$,
necessarily $R_0\subseteq S^*$, whence $x^*\in I(x,x_0)$. Therefore, to prove that $x^*\in I(x,y)$ it suffices to show that $x_0\in I(x,y)$. For this it is enough to prove that
$x_0$ is the gate of $x$ in $R$. Suppose by way of contradiction
that the gate of $x$ in $R$ is a vertex $x'$ different from $x_0$. Then $x'\in I(x,x_0)\subset X$ because $X$ is convex.
Since $x'\in R$, we conclude that $x'\in X\cap R=R_0$. This contradicts the assumption that $x_0$ is the gate of $x$ in $R_0$. Hence $x^*$ is the gate of $x$ in $S$, establishing that $S$ is gated.
\end{proof}

By Claims~\ref{Z*-face2} and~\ref{Z*-face1}, the intersection of $Z^*$ with each cell $X_i,i=1,\ldots,m,$ of $G$ is a proper face of $X_i$ (and thus a gated subgraph of $G$). Hence $Z^*$ satisfies the conditions of Claim~\ref{S-gated}, thus $Z^*$ is a gated  subgraph of $G$. Since $Z^*\subseteq N(E_f)\cap W(v,u)$ and $u\in N(E_f)\setminus Z^*$, $Z^*$ is a proper subgraph of $N(E_f)$. Since by Claim~\ref{Z*-face1} $Z^*\cap Y=Z$ and $Z$ is a proper subgraph of $Y$, the gated subgraph $Z^*$ separates any vertex of $N(E_f)\setminus Z^*\ne \emptyset$ from any vertex of $Y\setminus Z^*=Y\setminus Z\ne\emptyset$.  Consequently, $G$ is the gated amalgam  of $N(E_f)$ and $Y\cup Z^*$ along $Z^*$, 
concluding the proof of the theorem.
\end{proof}

\section{The median cell property}\label{sec:median}
Three (not necessarily distinct) vertices $x,y,z$ of a graph  $G$
are said to form a {\it metric triangle} $xyz$ if the intervals
$I(x,y),I(y,z),$ and $I(z,x)$ pairwise intersect only in the
common end vertices.  A (degenerate) equilateral metric triangle of
size 0 is simply a
single vertex. We say that a  metric triangle $xyz$ is a {\it
quasi-median} of the triplet $u,v,w$ if
$$d(u,v)=d(u,x)+d(x,y)+d(y,v),$$
$$d(v,w)=d(v,y)+d(y,z)+d(z,w),$$
$$d(w,u)=d(w,z)+d(z,x)+d(x,u).$$
Observe that, for every triplet $u,v,w,$ a
quasi-median $xyz$ can be constructed in the following way: first
select any vertex $x$ from $I(u,v)\cap I(u,w)$ at maximal distance
to $u,$ then select a vertex $y$ from $I(v,x)\cap I(v,w)$ at
maximal distance to $v,$ and finally select any vertex $z$ from
$I(w,x)\cap I(w,y)$ at maximal distance to $w.$ In the case that
the quasi-median is degenerate $(x=y=z),$ it is a median of the
triplet $u,v,w.$

We continue with the following characterization of metric triangles in hypercellular graphs:

\begin{proposition}\label{metric_triangle} If $G$ is a hypercellular graph and $xyz$ is a metric triangle of $G$, then $x,y,z$
belong to a common cell of $G$. In particular, the gated hull $\langle\langle x,y,z\rangle\rangle $ coincides with the convex hull \conv$(x,y,z)$
and is a cell of $G$.
\end{proposition}

\begin{proof}  First we prove the result for an arbitrary finite hypercellular graph $G$. By  Theorem~\ref{mthm:amalgam} 
either $G$ is a single cell and we are done, or $G$ is a gated amalgam of two proper gated subgraphs $G_1$ and $G_2$. Suppose without loss of generality that $y,z\in V(G_1)$. If $x\in V(G_1)$, then we can apply induction hypothesis to $G_1$ and conclude that $x,y,z$ belong to a common cell of $G_1$, and thus to a common cell of $G$. Now suppose that $x\in V(G_2)\setminus V(G_1)$. Let $x'$ be the gate of $x$ in $G_1$. Since $x'$ belongs to $G_1$ and $x$ not, $x'\ne x$. Since $x'\in I(x,y)\cap I(x,z)$, we obtain a contradiction with the assumption that $xyz$ is a metric triangle of $G$. Thus $x,y,z$ belong to a common cell of $G$. Since each cell of $G$ is gated and $xyz$ is a metric triangle, the gated hull of $x,y,z$ coincides with  the convex hull$\conv(x,y,z)$ and is a cell.

Now, suppose that $G$ is an arbitrary hypercellular graph. Let $G'$ be the subgraph induced by the convex hull of $x,y,$ and $z$. Then $G'$ is a finite hypercellular graph. By the above result for finite graphs, we have that $G'$ is a convex Cartesian product of edges and even cycles. Therefore, $G'$ is the convex hull of an isometric cycle of $G$. By Theorem~\ref{mthm:cells}
, $G'$ is a gated cell of $G$.
\end{proof}


For a triple of vertices $u,v,w$ of a graph $G$, a $u$-{\it apex}
relative to $v$ and $w$ is a vertex $x:=(uvw)\in
I(u,v)\cap I(u,w)$ such that $I(u,x)$ is maximal with respect to inclusion. A graph $G$ is
{\it apiculate}~\cite{BaCh_wma1} if and only if  for any vertex $u$ the vertex set of $G$
is a meet-semilattice with respect to the base-point order
$\preceq_u$ defined by $v\preceq_u v'$ $\Leftrightarrow$ $v\in I(u,v'),$ that is,
$I(u,v)\cap I(u,w)=I(u,(uvw))$ for any vertices $v,w$. Note that many partial cubes are not apiculate, see~\cite{bjedzi-90} for this discussion with respect to tope graphs of oriented matroids. For any triplet
$u,v,w$ of vertices of an apiculate graph $G$, the vertices $u,v,w$ admit unique apices $x:=(uvw),y:=(vuw),$ and $z:=(wuv)$ and
admit a unique quasi-median defined by the metric triangle $xyz$.

\begin{lemma} \label{Pasch-apiculate}~\cite[Proposition 2]{BaCh_wma1} Every Pasch graph $G$ is apiculate.
Consequently, every hypercellular graph is apiculate.
\end{lemma}


We say that a triplet $u,v,w$ of vertices in an apiculate graph $G$ admits a {\it median cell} (respectively, a {\it median cycle}) if the gated
hull $\langle\langle x,y,z\rangle\rangle $ of the unique quasi-median $xyz$ of $u,v,w$ is a Cartesian product of vertices, edges, and cycles
(respectively, a cycle or a single vertex). Notice that any median-cell is either a vertex or is a Cartesian product of even cycles of length $\ge 6$.
A graph $G$ is called {\it cell-median} (respectively, {\it cycle-median}) if $G$ is apiculate and any triplet $u,v,w$ of $G$
admits a unique median cell (respectively, unique median cycle or vertex). By \cite[Proposition~3]{BaCh_cellular}, bipartite cellular
graphs are cycle-median. This result has been extended in~\cite{Po3} by showing that all graphs which are gated amalgams of even cycles
and hypercubes are cycle-median, and those are exactly the netlike cycle-median partial cubes. Now, we are ready to prove Theorem~\ref{mthm:median}.

%

\begin{theorem} \label{median_cell} A partial cube  $G=(V,E)$ is cell-median if and only if $G$ is hypercellular.
\end{theorem}

\begin{proof}  First we prove that hypercellular graphs are cell-median. By Corollary~\ref{Pasch-Peano} and Lemma~\ref{Pasch-apiculate} it follows that any graph $G$ from $\mathcal{F}(Q_3^-)$ is apiculate. Therefore, to show that $G$ is cell-median it suffices to show that if $xyz$ is a metric triangle of $G$, then the gated hull $\langle\langle x,y,z\rangle\rangle $ of $x,y,z$ is a cell; this is  Proposition~\ref{metric_triangle}.

Conversely, to prove that cell-median partial cubes are hypercellular graphs we will use Theorem~\ref{mthm:amalgam}(ii). Namely, we have to prove that a cell-median partial cube $G$ satisfies the 3CC-condition  and that any cell $X$ of $G$ is gated. Suppose by way of contradiction, that $G$ contains a cell $X$ and a vertex not having a gate in $X$. Let $v$ be such a vertex closest to $X$. Since $v$ does not have a gate, we can find two vertices $x,y\in X$ such that $I(x,v)\cap X=\{ x\}, I(y,v)\cap X=\{ y\}$,  and $x$ is closest to $v$ in $X$. From the choice of $v$, we conclude that $I(v,x)\cap I(v,y)=\{ v\}$. Hence, the vertices $v,x,$ and $y$ define a metric triangle of $G$. By the median-cell property, the convex hull of $v,x,y$ is a gated cell $Y$ of $G$. Let $Z:=X\cap Y$. Notice that $x,y\in Z$ and $v\notin Z$. Notice also that $Z$ is convex but not gated, otherwise we will get a contradiction with the choice of $v$. Since $Z$ is convex, $Z$ is a subproduct of $X$ and $Y$ and is a Cartesian product of convex paths and cycles.   Let $Z=Z_1\square Z_2\square\cdots \square Z_m$. Suppose also that $X=X_1\square X_2\square\cdots \square X_m$ and $Y=Y_1\square Y_2\square\cdots \square Y_m$, where each $X_i, i=1,\ldots, m,$ and each $Y_j,j=1,\ldots, m,$ is an even cycle, an edge, or a vertex, and each $Z_i$ is a convex subgraph of each $X_i$ and $Y_i$, $i=1,\ldots, m$.  Since $Z$ is not gated, at least one factor, say $Z_1$, is a convex
path of length at least 2, and $X_1$ and $Y_1$ are even cycles.

Let $z$ be a vertex of $Z=Z_1\square Z_2\square\cdots \square Z_m$ of the form  $z=(z_1, z_2, \ldots, z_m)$. Then the layers $X_1\square \{z_2\} \square \cdots \square \{z_m\}$ of $X$ and $Y_1\square \{z_2\} \square \cdots \square \{z_m\}$ of $Y$ are respectively a convex and a gated cycle of $G$. These two cycles  intersect in a path of length at least two, namely in $Z_1\square \{z_2\} \square \cdots \square \{z_m\}$. By the following Claim~\ref{claim:convex_gated}, this is impossible. This contradiction establishes that the cell $X$ is gated.

\begin{claim}\label{claim:convex_gated} Let $C_1,C_2$ be two distinct convex cycles of a partial cube $G$. If  $C_2$ is gated, then $C_1\cap C_2$ is empty, a vertex, or an edge of $G$.
\end{claim}

\begin{proof} Suppose by way of contradiction that $C_1\cap C_2$ contains a path $(v_1,v,v_2)$ of length 2. Let $u$ be the antipodal to $v$ vertex of $C_1$. If $u\in C_2$, then $u,v\in C_2$ and by convexity of $C_2$ we deduce that $C_1=I(u,v)\subseteq C_2$, thus $C_1=C_2$, a contradiction. Consequently, $u\notin C_2$. Let $x$ be the gate of $u$ in $C_2$. Since $v_1,v_2\in C_2$, $x \in I(u,v_1)\cap I(u,v_2)$. From these inclusions we conclude that either $x=v$ or $x$ is the antipodal to $v$ vertex of $C_2$. Since $v_1,v_2\in I(u,v)$, necessarily $x\ne v$. But if
$x$ is the antipode of $v$ in $C_2$, then $C_2\subset I(v_1,u)\cup I(v_2,u)$, which is only possible if $C_1=C_2$.
\end{proof}



To establish the 3CC-condition, let  $C_1,C_2,C_3$ be three convex cycles of $G$ such that any two cycles $C_i,C_j$, $1\le i<j\le 3$,
intersect in an edge $e_{ij}$ and the three cycles intersect in a vertex $x$.  Since the cells of $G$ are gated, $C_1,C_2,C_3$ are gated cycles of $G$.
Let $e_{12}=xx_2,e_{23}=xx_0,$ and $e_{13}=xx_3$. Let $v_1,v_2,$ and $v_3$ be the vertices of respectively $C_1,C_2$, and $C_3$ antipodal to $x$. If $v_1,v_2$, and $v_3$ define
a metric triangle, then the gated hull of $v_1,v_2,v_3$ is a Cartesian product of vertices, edges, and even cycles containing $C_1,C_2$, and $C_3$, and we are done. So suppose without loss of generality
that there exists a vertex $u_1\in I(v_1,v_2)\cap I(v_1,v_3)$ adjacent to $v_1$.
Notice that $x_2$ and $x_3$ are the gates of $v_1$ in the cycles $C_2$ and $C_3$, respectively. In fact this is true since the gate of $v_1$ in $C_2$ must be in $I(v_1,x_2)\cap C_2=\{x_2\}$ and the gate of $v_1$ in $C_3$ must be in $I(v_1,x_3)\cap C_3=\{x_3\}$.
%

Since $u_1$ is adjacent to $v_1$,  one can easily show that the gates $y_2$ and $y_3$ of $u_1$ in $C_2$ and $C_3$ are two vertices adjacent to $x_2$ and $x_3$, respectively. If $y_2$ or $y_3$ coincides with $x$, then $u_1\in I(v_1,x)$, contrary to the assumption that the cycle $C_1$ is convex. Thus $y_2$
is the second  neighbor of $x_2$ in $C_2$ and $y_3$ is the second neighbor of $x_3$ in $C_3$. Since $y_2,y_3\in W(u_1,v_1), x_2,x,x_3\in W(v_1,u_1)$, and $W(u_1,v_1)$ is convex, we deduce that $d(y_2,y_3)=2$. Consequently, $y_2$ and
$y_3$ have a common neighbor $z_0$. First suppose that $z_0\ne x_0$, i.e., $x_0$ is not adjacent to one of the vertices $y_2,y_3$, say $x_0$ and $y_2$ are not adjacent. Since $C_2$ and $C_3$ are convex, $z_0$ cannot be adjacent to $x$. Thus $d(z_0,x)=3$,
whence the 6-cycle $C_0:=(z_0,y_2,x_2,x,x_3,y_3)$ is isometric. Since $C_0$ intersects $C_2$ and $C_3$ along paths of length 2, by Claim~\ref{claim:convex_gated}, this cycle cannot be gated and thus cannot be convex. Since $G$ is cell-median, the convex hull of $C_0$ cannot
be a $Q^-_3$, thus its convex hull is a 3-cube $Q_3$. Therefore the intervals $I(y_2,x)$ and $I(x,y_3)$ are squares of $G$ which necessarily must coincide with $C_2$ and $C_3$. Consequently, $x_0$ is adjacent to $y_2$ and $y_3$, contrary to the assumption that $x_0$ and $y_2$ are not adjacent. Now, suppose that $z_0=x_0$, i.e., $C_2=(x,x_2,y_2,x_0)$ and $C_3=(x,x_0,y_3,x_3)$. In this case, $y_2=v_2$ and $y_3=v_3$. If $C_1$ is also a 4-cycle, then we get an isometric $Q^-_3$, which must be completed to a 3-cube, otherwise $v_1,y_2,$ and $y_3$ define a metric triangle whose gated hull is not a cell.

So, $C_1$ is a cycle of length at least 6. We assert that the gated hull of $v_1,y_2,$ and $y_3$ is a cell isomorphic to $C_{1}\square K_2$. For the sake of contradiction, assume that this is not the case and assume that $C_1$ has minimal length among all convex cycles with two 4-cycles attached to them such that they pairwise intersect in three different edges, all three in a vertex, and their convex hull is not a cell.  If the vertices $y_2,y_3$ have a second common neighbor $p$, then we get an isometric $Q^-_3$ which must be completed to a $Q_3$. Consequently, $x_2$ and $x_3$ have a common neighbor different from $x$, which is impossible because $C_1$ is convex. Thus $x_0$ is the unique common neighbor of $y_2$ and $y_3$. Let $u^*_1$ be the apex of $u_1$ with respect to the pair $y_2,y_3$. We assert that $u_1=u^*_1$. Suppose not and let $u'_1$ be a neighbor of $u_1$ in $I(u_1,u^*_1)$. Consider the gate of $u'_1$ in $C_1$. If this gate is not the vertex $v_1$, then it must be one of the neighbors of $v_1$ in $C_1$ and $u'_1$ must be adjacent to this vertex. But if this is say the neighbor $v'_1$ of $v_1$ in the path $I(v_1,x_2)$, then $v'_1$ cannot belong to a shortest path between $u'_1$ and $x_3$, whence $v'_1$ cannot serve as a gate of $u'_1$. Thus $v_1$ must be the gate of $u'_1$ in $C_1$. In this case, $d(u'_1,x_2)=2+d(v_1,x_2)=d(u'_2,y_2)+1$. Since $d(u'_1,y_2)=d(u_1,y_2)-1$ and $d(u_1,y_2)+1=d(u_1,x_2)=1+d(v_1,x_2)$, we will obtain a contradiction. This shows that $u^*_1=u_1$, i.e., $I(u_1,y_2)\cap I(u_1,y_3)=\{ u_1\}$. Since $y_2$ and $y_3$ are closer to $u_1$ than $x_0$ and $x_0$ is the unique common neighbor of $y_2$ and $y_3$, we conclude that the triplet $u_1,y_2,y_3$ defines a metric triangle. Hence $\langle\langle u_1,y_2,y_3 \rangle\rangle$ is a gated cell $U$ of $G$.

Since $(y_2,x_0,y_3)$ is a convex path of length 2 of the cell $U=U_1\square\cdots \square U_m$, necessarily $(y_2,x_0,y_3)$ is contained in a layer of $U$ which is a gated cycle $C'_1$ of $G$, say  $C'_1=U_1\square \{ u_2\} \square \cdots \square \{u_m\}$ for a cyclic factor $U_1$ of length $\ge 6$. First suppose that $u_1\notin C'_1$. Then the length of $C'_1$ is smaller than  the length of $C_1$. From the choice of $C_1$ and since $C'_1$ pairwise intersects  the cycles $C_2$ and $C_3$, we conclude that the gated hull of $C'_1\cup C_2\cup C_3$ is a cell $U'$ isomorphic to $C'_1\square K_2$. But then in $U'$ we can find a gated cycle $C''_1$ isomorphic to $C'_1$ and  containing the convex path $(x_2,x,x_3)$. Since $C''_1$ is shorter than $C_1$ and $x_1,x,x_3\in C''_1\cap C_1$, we obtain a contradiction with Claim~\ref{claim:convex_gated}. Now, let $u_1\in C'_1$. Then obviously the cell $U$ coincides with $C'_1$. Since $C'_1$ and $C_1$ are gated cycles of the same length and we have the edges $v_1u_1,x_2y_2,xx_0$, and $x_3y_3$, one can easily show that any vertex $z'$ of $C'_1$ is adjacent to a unique vertex $z$ of $C_1$ such that the subgraph $H$ of $G$ induced by $C_1\cup C'_1$ is isomorphic to $C_1\square K_2$. To conclude the proof of the 3CC-condition, it remains to show that $H$ is a convex subgraph of $G$. For this it suffices to show that for any vertex $q\notin V(H)$ adjacent to a vertex $p$ of $H$, $q$ does not belong to a shortest path between $p$ and some vertex $q'$ of $H$. Suppose without loss of generality that $p\in C_1$ and let $p'$ be the unique neighbor of $p$ in $C'_1$. Then obviously $p$ is the gate of $q$ in $C_1$, thus $p\in I(q,r_1)$ for every $r_1 \in C_1$. Analogously, $p'$ must be the gate of $q$ in $C'_1$, otherwise since $d(q,p')=2$, the gate of $q$ must be one of the neighbors of $p'$ in $C'_1$ and we obtain a $K_{2,3}$, which is forbidden in partial cubes. Therefore $p'\in I(q,r_2)$ for any vertex $r_2\in C'_1$. Since $p\in I(q,p')$, we conclude that $p\in I(q,r_2)$. This implies that $p\in I(q,q')$, thus $q$ cannot lie in $I(p,q')$. This establishes the 3CC-condition and concludes the proof of the theorem.
\end{proof}

\section{Properties of hypercellular graphs}\label{sec:properties}
We continue with several properties of hypercellular graphs, in particular we prove Theorems~\ref{mthm:COM},~\ref{mthm:stars}, and~\ref{mthm:fixed}.
First, we show how hypercellular graphs are related with other known classes of partial cubes.
We also establish some basic properties of geodesic convexity  in hypercellular graphs and establish a
fixed-cell property. Some of these results directly follow from Theorem~\ref{mthm:amalgam}.

\subsection{Relations with other classes of partial cubes}\label{subsec:otherclasses}
Recall that  bipartite cellular graphs are the bipartite graphs in which all
isometric cycles are gated. It is shown in~\cite{BaCh_cellular} that bipartite cellular graphs are partial cubes and that any finite bipartite graph is a bipartite cellular
graph if it can be can be obtained by successive gated amalgamations from its isometric cycles. In~\cite{Po3}, Polat investigated a class of netlike partial cubes in which each finite convex subgraph
is a gated amalgam of even cycles -  let us call them \emph{Polat graphs} for now. They are exactly the netlike partial cubes satisfying the median cycle property and generalize bipartite cellular graphs as well as median graphs. 
 Theorem~\ref{mthm:amalgam} and Theorem~\ref{mthm:median} have the following corollary:

%

\begin{corollary}
Bipartite cellular graphs are precisely the graphs in  ${\mathcal F}(Q_3^-,Q_3)$, while median graphs are precisely the graphs in ${\mathcal F}(Q_3^-,C_6)$ and Polat graphs are ${\mathcal F}(Q_3^-,C_6\square K_2)$. In particular, the latter class contains the first two and all three classes are contained in the class of hypercellular graphs.
\end{corollary}
\begin{proof}
Since the hypercellular graphs are exactly the graphs from ${\mathcal F}(Q_3^-)$, the last assertion follows from the first ones. Median graphs, bipartite cellular graphs, and Polat graphs are pc-minor closed families. Since $Q_3^-$ and $Q_3$ are not cellular, $Q_3^-$ and $C_6$ are not median, and $Q_3^-$ and $C_6\square K_2$ are not Polat graphs, this settles the inclusion of all three families in ${\mathcal F}(Q_3^-,Q_3)$, ${\mathcal F}(Q_3^-,C_6)$, and ${\mathcal F}(Q_3^-,C_6\square K_2)$, respectively.

Conversely, let $G$ be a graph from ${\mathcal F}(Q_3^-,Q_3)$. Since $G$ is hypercellular, by Theorem~\ref{mthm:amalgam} any finite convex subgraph of $G$  can be obtained by successive gated amalgamations from cells.
Since $Q_3$ is a forbidden pc-minor, all cells of $G$ are edges or even cycles. Thus $G$ is a bipartite cellular graph.

Analogously, let $G$ be a graph from ${\mathcal F}(Q_3^-,C_6)$. Then $G$ does not contain convex
cycles of length $\ge 6$. Hence any cell of $G$ is a cube. Consequently, any finite convex subgraph of $G$  can be obtained by successive gated amalgamations from cubes, i.e., $G$ is median. Alternatively, by Theorem~\ref{mthm:median}, $G$ satisfies the median cell property. Since, any cell of $G$ is a cube,  all median cells of $G$ are vertices and therefore $G$ is a median graph.

Finally, let $G$ be a graph from ${\mathcal F}(Q_3^-,C_6\square K_2)$. Since $G$ is hypercellular, by Theorem~\ref{mthm:amalgam} any finite convex subgraph of $G$  can be obtained by successive gated amalgamations from cells.
Since $C_6\square K_2$ is a forbidden pc-minor, all cells of $G$ are even cycles or cubes. Thus, $G$ is a Polat graph.
\end{proof}

With a cell $X=F_1\square \cdots \square F_m$ of $G$ we associate a convex polyhedron $[X]$ obtained as a Cartesian product of segments and regular polygons, where each face $F_i$ which is a $K_2$ is replaced by a unit segment and any face $F_i$ which is an even cycle $C$ of length $2n$ is replaced by a regular polygon with $2n$ sides. Hence $\mbox{dim}(X)$ can be viewed as the (topological) dimension of $[X]$. Since by Lemma~\ref{two_cycles}, in a hypercellular graph $G$ the intersection of any two cells is also a cell, the union of all convex polyhedra $[X], X\in {\mathbf X}(G),$ can be viewed as a polyhedral cell complex, which we denote by ${\mathbf X}(G)$. The {\it dimension} $\mbox{dim}(G)$ of a graph $G$ from $\mathcal{F}(Q_3^-)$ is the dimension of this cell complex, i.e., the maximum dimension of a cell of $G$. Notice that the 1-skeleton of ${\mathbf X}(G)$ coincides with $G$ and the 2-skeleton of ${\mathbf X}(G)$ coincides with ${\mathbf C}(G)$.

The following was announced as Theorem~\ref{mthm:COM} in the introduction:
\begin{corollary}  \label{thoremC} Any finite hypercellular  graph $G$ is the tope graph of a COM, more precisely,
$G$ is a tope graph of a zonotopal COM. Consequently, the zonotopal cell complex ${\mathbf X}(G)$ of any
locally-finite hypercellular graph $G$ is contractible.
\end{corollary}

\begin{proof} By~\cite[Proposition 3]{BaChKn}, each COM can be obtained from its maximal faces (which are all oriented
matroids) using  COM amalgamations.
Since a gated amalgamation is a stronger version of a COM amalgamation and each Cartesian product of edges and
even cycles is the tope graph of a realizable oriented matroid, Theorem~\ref{mthm:amalgam} 
implies that each finite
graph $G$ from $\mathcal{F}(Q_3^-)$ is the tope graph of a zonotopal COM. From the contractibility
of the cell complexes of all COMs established in~\cite[Proposition 14]{BaChKn}, it follows that for any finite
hypercellular graph $G$ its zonotopal  complex ${\mathbf X}(G)$ is contractible.

Now, we will prove the contractibility of  ${\mathbf X}(G)$ for any locally-finite hypercellular graph $G$. For this, we will represent
$G$ as a directed union  of finite convex subgraphs $G_i$ of $G$. Let $v_0$ be an arbitrary fixed vertex and let $B_i(v_0)$ be the ball
of radius $i$ centered at $v_0$. Since $G$ is locally-finite, each such ball $B_i(v_0)$ is finite. Moreover, since $G$ is a partial cube,
the convex hull$\conv(A)$ of any finite set $A$ of $G$ is finite (because$\conv(A)$ coincides with the intersection of $V(G)$ with
the smallest hypercube $H$ of $H(\Lambda)$ hosting $A$ and $H$ is finite-dimensional). Hence the subgraph $G_i$ of $G$ induced by
conv$(B_i(v_0))$ is a finite convex subgraph of $G$, and thus hypercellular. Therefore, by the first part,  each of the zonotopal
complexes ${\mathbf X}(G_i)$, $i\ge 0$, is contractible. Consequently, ${\mathbf X}(G)$ is the direct union $\bigcup_{i\ge 0} {\mathbf X}(G_i)$
of  contractible complexes, thus ${\mathbf X}(G)$ is contractible by Whitehead's theorem.
\end{proof}



\subsection{Convexity properties}\label{subsec:convex}
The geodesic convexity of a graph $G=(V,E)$ satisfies the {\it join-hull commutativity property} (JHC) if for any convex set $A$ and any vertex $x\notin A$,$\conv(x\cup A)=\bigcup\{ I(x,v): v\in A\}$~\cite{VdV} holds. It is well-know and easy to prove that JHC property is equivalent to the Peano axiom: if $u,v,w$ is an arbitrary triplet of vertices, $x\in I(u,w)$ and $y\in I(v,x)$, then there exists a vertex $z\in I(v,w)$ such that $y\in I(u,z)$. A graph $G$ is called a {\it Pasch-Peano graph}~\cite{BaChVdV,VdV} if the geodesic convexity of $G$ satisfies the Pasch and Peano axioms. In particular, such a graph is in $\mathcal{S}_4$.

\begin{corollary} \label{Pasch-Peano} Any hypercellular graph $G$  is a Pasch-Peano graph.
\end{corollary}

\begin{proof} Both the Pasch and the Peano axioms concern triplets of vertices $u,v,w$ and vertices included in the convex hull of $u,v,w$. Since the convex hull of any finite set of vertices in a partial cube is finite, to prove that a hypercellular graph is Pasch-Peano, it suffices to prove that each finite hypercellular graph is Pasch-Peano. Since each of the Pasch and Peano axioms are preserved by gated amalgams and Cartesian products~\cite{BaChVdV,VdV}, now the result directly follows Theorem~\ref{mthm:amalgam} 
and the fact that cycles and edges are Pasch-Peano graphs.
\end{proof}

The {\it Helly number} $h(G)$ of a graph $G$ is the smallest number $h\ge 2$ such that every finite family of (geodesically) convex sets meeting $h$ by $h$ has a nonempty intersection.
The {\it Caratheodory number} $c(G)$ is the smallest number $c\ge 2$ such that for any set $A\subset V$ the convex hull of $A$ is equal to the union of the convex hulls of all subsets of $A$ of size $c$. The {\it Radon number} $r(G)$ of a graph $G$ is the smallest number $r\ge 2$ such that any set of vertices $A$ of $G$ containing at least $r+1$ vertices can be partitioned into two sets
$A_1$ and $A_2$ such that $\conv(A_1)\cap \conv(A_2)\ne \emptyset$. More generally, the {\it $m$th partition number (Tverberg number)} is the smallest integer $p_m\ge 2$ such that any set of vertices $A$ of $G$ containing at least $p_m+1$ vertices can be partitioned  into $m$ sets $A_1,\ldots, A_m$ such that $\cap_{i=1}^m\conv(A_i)\ne \emptyset$. For a detailed treatment of
all these fundamental parameters of abstract and graph convexities, see~\cite{VdV}.

The following result is straightforward:

\begin{lemma} \label{helly_cycle} For $G\cong K_2$, $h(G)=r(G)=2$ and $c(G)=1$.   If $G\cong C$, then $h(G)=r(G)\le 3$ and $c(G)=2$ ($h(G)=r(G)=3$ if $C$ is of length at least 6).
\end{lemma}

\begin{corollary} Let $G$ be a hypercellular graph. Then $h(G)\leq 3$, $c(G)\le 2\mbox{dim}(G),$ and $r(G)\le 10\mbox{dim}(G)+1$. More generally, $p_m\le (6m-2)\mbox{dim}(G)+1$.
\end{corollary}

\begin{proof} We will use the results of~\cite[Chapter II, \S2]{VdV} for Cartesian products and of~\cite{BaChVdV} or~\cite[Chapter II, \S3]{VdV} for gated amalgams of convexity structures. Notice also that since in partial cubes convex hulls of finite sets are finite, it suffices to establish our results for finite hypercellular graphs $G$.
By these results, $h(G_1\square G_2)=\max\{ h(G_1),h(G_2)\}$ and if $G$ is the gated amalgam of $G_1$ and $G_2$, then $h(G)=\max\{ h(G_1),h(G_2)\}$. By these formulas,
Lemma~\ref{helly_cycle}, and Theorem~\ref{mthm:amalgam} 
, we conclude that $h(G)\leq 3$ for any hypercellular graph $G$. In case of the Caratheodory number, we have
$c(G_1\square G_2)\le c(G_1)+c(G_2)$ and $c(G)=\max\{ c(G_1),c(G_2)\}$ if $G$ is a gated amalgam of $G_1$ and $G_2$. By  the first formula and Lemma~\ref{helly_cycle},
we conclude that if $X$ is a Cartesian product of $k'$ cyclic factors and $k''$ edges, then $c(X)\le 3k'+2k''$ and $\mbox{dim}(X)=2k'+k''$, yielding $c(X)\le 2\mbox{dim}(X)$. To deduce the upper bounds for Radon and partitions numbers, we will use the following inequality of~\cite{DoReSi} (see also~\cite[5.15.1]{VdV}) for all convexities: $p_m(G)\le c(G)(m\cdot h(G)-1)+1$. Replacing in this formula $h(G)=3$ and $c(G)\le  2\mbox{dim}(G),$ we obtain the required inequalities.
\end{proof}

For cellular graphs, an exact bound $p_m\le 3m$ for partition number was obtained in \cite{ChTo}.

We conclude this subsection with a local-to-global characterization of convex and gated sets in hypercellular graphs. A similar characterization of gated sets was obtained for bipartite cellular graphs~\cite[Proposition 1]{BaCh_cellular} and netlike graphs~\cite[Theorem 6.2]{Po1}. Notice also that other local characterizations of convex and gated sets are known for weakly modular graphs~\cite{Ch_triangle}. The following result is similar to the content of Claim~\ref{S-gated}.

\begin{proposition}\label{cell_intersect} A connected subgraph $H$ of a hypercellular graph  $G$ is convex (respectively, gated) if and only if the intersection of $H$ with each cell of $G$ is convex (respectively, gated).
\end{proposition}

\begin{proof} We closely follow the proof of \cite[Proposition 1]{BaCh_cellular}. Necessity is evident: any cell $X$ of $G$ is convex and gated, therefore $X$ intersect each convex (respectively, gated) subgraph in a convex (respectively, gated)
subgraph.

As to the converse, in both cases we will first show that $H$ is convex. For two vertices $y,z$ we denote by  $k(y,z):=d_H(y,z)$ the distance between $y$ and $z$ in $H$.
Suppose the contrary and let $v,x$ be two vertices of $H$ minimizing $k(y,z)$ such that $I(v,x)$ is not included in $H$. Then there exists a shortest $(v,x)$-path $Q$ whose
inner vertices do not belong to $H$. Let $P$ be any path of minimal length joining $v$ and $x$ inside $H$. We assert that $P$ is a shortest path of $G$. By the
choice of $v,x$, the paths $P$ and $Q$ intersect only in $v$ and $x$. Let $u$ be a neighbor of $v$ in $P$. If $P$ were longer than $Q$, then $u\notin I(v,x)$, whence
$v\in I(u,x)$.  Since $k(u,x)<k(v,x)$, by the minimality in the choice of the pair $v,x$ we conclude that $Q\subset I(u,x)\subset H$, a contradiction. Thus $P$ is a shortest
path of $G$. Let $w$ be the neighbor of $v$ in $Q$. If the vertices $w$ and $u$ have a common neighbor $y$ different from $v$, since $u,w\in I(v,x)$ and $y\in I(w,u)$,
from the convexity of the interval $I(v,x)$ we conclude that $y\in I(v,x)$. This implies that $y\in I(w,x)\cap I(u,x)$. Since $I(u,x)\subset H$, we conclude that
$y\in H$. Since $w\notin H$, the intersection of $H$ with the square $(v,u,y,w)$ (which is a cell of $G$) is not convex. This contradiction shows that $I(u,w)=\{ u,v,w\}$.
Hence $I(u,w)\cap I(u,x)=\{ u\}$ and $I(w,u)\cap I(w,x)=\{ w\}$. On the other hand, the minimality in the choice of the pair $v,x$ implies that $I(x,u)\cap I(x,w)=\{ x\}$.
Consequently, the triplet $x,u,w$ defines a metric triangle of $G$. By Proposition~\ref{metric_triangle}, this metric triangle $xuw$ is included in a cell $X$ of
$G$. Since $X$ is convex and $v\in I(u,w)$, we obtain $v\in X$. But then $X\cap H$ is not convex because $v,x\in X\cap H$ and $w\in I(v,x)\setminus H$. This contradiction shows
that $H$ is convex.

Now suppose that the intersection of $H$ with each cell of $G$ is gated. Suppose that $H$ is not gated. Choose a vertex $z$ at minimum distance to $H$ having no gate in
$H$. Let $x$ be a vertex of $H$ closest to $z$, and let $y$ be a vertex of $H$ such that the interval $I(z,y)$ does not contain $x$, where $d(x,y)$ is as small as possible.
Then the intervals $I(x,y),I(y,z),$ and $I(z,x)$ intersect each other only in the common end vertices. Hence $x,y,$ and $z$ define a metric triangle $xyz$. By Proposition
\ref{metric_triangle}, $xyz$ is contained in a cell $X$. Since $z\notin H$ and $x,y\in H$, the choice of the vertices $z$ and $x,y$ implies that $X\cap H$ is not gated.
This contradiction establishes that $H$ is gated and concludes the proof.
\end{proof}

\subsection{Stars and thickening}\label{subsec:stars}
A \emph{star} $\St(v)$ of a vertex $v$ (or a star $\St(X)$ of a cell $X$) is the union of all cells of $G$ containing $v$ (respectively, $X$).

\begin{proposition}\label{stars}
For any cell $X$ of a hypercellular graph $G$ in which all cells are of finite dimension, the star $\St(X)$ is gated.
\end{proposition}

\begin{proof} Since $\St(X)$ is a connected subgraph of $G$, by Proposition~\ref{cell_intersect} it is enough to prove that the intersection of $\St(X)$ with any cell $Y$ of $G$ is gated. We apply Lemma~\ref{lem:subcells} to the cell $Y$ to show that $\St(X)\cap Y$ is gated. First, notice that $\St(X)\cap Y$ is connected. Indeed, let $X_1, X_2 \ldots$ be the maximal cells  of $\St(X)$ intersecting $Y$. Since $X_1, X_2 \ldots$ intersect in $X$ and each of these cells intersects $Y$, by the Helly property for gated sets, $Y\cap X_1\cap X_2 \ldots$ is non-empty and gated. Thus,
any two vertices of $\St(X)\cap Y$ can be connected with a path in $\St(X)\cap Y$ passing via this intersection, whence $\St(X)\cap Y$ is connected.

%
%
%
%

Let  $P=(y',x,y'')$ be any 2-path in $\St(X)\cap Y$ and let $C=\langle\langle P\rangle\rangle$ be its gated hull in $Y$. We will prove that $C$ is included in $\St(X)$. This is obviously so if the path $P$ is included in a single cell $X_i$ of $\St(X)$. Indeed, in this case $C$ is included in the gated subgraph $X_i\cap Y$. So, assume that the edges $y'x$ and $xy''$ do not belong to a common cell of  $\St(X)$. Notice that each of these edges belong to a cell of $\St(X)$: for example, the edge $xy'$ belongs to all cells of $\St(X)$ that contain a furthest from $X$ vertex of the pair $\{x,y'\}$.  Let $X_i$ be a cell of $\St(X)$ including the edge $xy'$. Analogously, let $X_j$ be a cell of $\St(X)$  including the edge $xy''$. By what was assumed above, $X_i$ and $X_j$ are not included in each other, in particular $X_i\ne X_j$. Let $X_{ij}:=X_i\cap X_j$. Then $X_{ij}$ is a cell of $\St(X)$ containing $x$ but not containing $y'$ and $y''$.

Let $Z$ be a maximal cell of the form $Z=\langle\langle C \cup Z'\rangle\rangle$ for some subcell $Z'$ of $X_{ij}$ containing $x$. Note that such a cell exists since $Z'$ can be chosen to be $x$ and $C$ is a cell of $Y$ containing $x$.  Since the intersection of cells is a cell, we can further assume that $Z':=X_{ij}\cap Z$. We assert that $Z=\langle\langle C\cup X_{ij}\rangle\rangle$. Suppose that this is not the case. Then there exists an edge $zw\in X_{ij}$ with $z\in Z'=X_{ij}\cap Z$ and $w\notin Z$. Let $k$ be the dimension of $Z'$. We will use the following property of cells of hypercellular graphs, which is a direct consequence of Lemma~\ref{product_cycles_gated}:

\begin{claim} If $D'$ is a subcell of dimension $\ell$ of a  cell $D$ of $G$ and $v'v$ is an edge with $v'\in D'$ and $v\in D\setminus D'$, then $\dim(\langle\langle D'\cup v'v\rangle\rangle)=\ell+1$.
\end{claim}

Since $Z'\cup xy'\subset X_i$ with $x\in Z', y'\notin  Z'\subset X_{ij}$, by the previous claim the dimension of  $\langle\langle Z'  \cup xy' \rangle\rangle$ is $k+1$. Similarly, the dimension of
$\langle\langle Z'\cup xy'' \rangle\rangle$ is $k+1$. Moreover, $Z'\cup zw\subset X_{ij}$ with $z\in Z',w\notin Z'$, thus  $\langle\langle Z'  \cup zw \rangle\rangle$ has also dimension
$k+1$. Now we have $xy',\langle\langle Z'  \cup zw \rangle\rangle \subset X_i$ with $x\in \langle\langle Z'  \cup zw \rangle\rangle$ since $x\in Z'$ and $y'\notin \langle\langle Z'  \cup zw \rangle\rangle$
since $y'$ is not in $X_{ij}$, thus $\langle\langle Z'  \cup zw \cup xy' \rangle\rangle$ has dimension $k+2$. Analogously, $\langle\langle Z'  \cup zw \cup xy'' \rangle\rangle$ has dimension $k+2$. Finally,
since $xy'', \langle\langle Z'  \cup xy' \rangle\rangle \subset Z$ with $x \in \langle\langle Z'  \cup xy' \rangle\rangle$ and $y''\notin \langle\langle Z'  \cup xy' \rangle\rangle$, the dimension
of $\langle\langle Z'  \cup xy' \cup xy'' \rangle\rangle$ is also $k+2$.

Consequently, we have proved that $\langle\langle Z'  \cup zw \cup xy' \rangle\rangle$, $\langle\langle Z'  \cup zw \cup xy'' \rangle\rangle$, and $\langle\langle Z'  \cup xy' \cup xy'' \rangle\rangle$ are three cells of dimension $k+2$ that pairwise intersect in the cells $\langle\langle Z'  \cup zw  \rangle\rangle$, $\langle\langle Z'   \cup xy' \rangle\rangle$, and $\langle\langle Z' \cup xy'' \rangle\rangle$  of dimension $k+1$ and the intersection of all three cells is the cell $Z'$ of dimension $k$.
%
By Theorem~\ref{mthm:amalgam}(iii), there is a $(k+3)$-dimensional cell $W$ that includes all of them. In particular, $W=\langle\langle C\cup Z' \cup zw\rangle\rangle$. Since  the gated hull of $Z' \cup wz$ is a subcell $Z''$ of $X_{ij}$ properly containing $Z'$ and  since $W=\langle\langle C \cup Z''\rangle\rangle$,  we obtain a contradiction  to the maximality of $Z$.

Thus, $Z=\langle\langle C\cup X_{ij}\rangle\rangle$ is a cell of $\St(X)\cap Y$ including $C$ and $X\subset X_{ij}$, whence $C\subset \St(X)$. Lemma~\ref{lem:subcells} implies that $\St(X)\cap Y$ is gated.
\end{proof}

The {\it thickening} $G^{\Delta}$ of a hypercellular graph $G$ is a graph having the same set of vertices as $G$ and two vertices
$u,v$ are adjacent in $G^{\Delta}$ if and only if $u$ and $v$ belong to a common cell of $G$. A graph $H$ is called a {\it Helly graph} if any  collection
of pairwise intersecting balls of $G$ has a nonempty intersection~\cite{BaCh_survey}. Analogously, $H$ is called a {\it 1-Helly graph} (respectively, {\it clique-Helly graph}) if any collection
of pairwise intersecting 1-balls (balls of radius 1) of $G$ (respectively, of maximal cliques) has a nonempty intersection.

\begin{proposition} \label{thickening}  The thickening $G^{\Delta}$ of a locally-finite hypercellular graph $G$ is a  Helly graph.
\end{proposition}

\begin{proof} Pick any vertex $v$ of $G$ and let $B_1(v)$ denote the ball of radius 1 of $G^{\Delta}$ centered at $v$. From the definition of $G^{\Delta}$ it immediately follows that $B_1(v)$ is isomorphic to the star $\St(v)$ of $v$ in $G$. Since  $G$ is locally-finite, $\St(v)$ is finite. By Proposition~\ref{stars},  $\St(v)$ is a gated subgraph of $G$. By the Helly property of finite gated sets, we conclude that $G^{\Delta}$ is a  1-Helly graph.  Any maximal clique $K$ of $G^{\Delta}$ is the intersection of all 1-balls centered at the vertices of $K$, therefore the family of maximal cliques of $G^{\Delta}$ can be obtained as the intersections of 1-balls of $G^{\Delta}$. By~\cite[Remark 3.6]{ChChHiOs}, $G^{\Delta}$ is a clique-Helly graph. By Theorem~\ref{mthm:COM} the zonotopal cell complex ${\mathbf X}(G)$ of $G$ is contractible and therefore simply connected. This easily implies that the clique complex of $G^{\Delta}$ is simply connected. Consequently, $G^{\Delta}$ is a clique-Helly graph with a simply connected clique complex. By~\cite[Theorem 3.7]{ChChHiOs}, $G^{\Delta}$ is a Helly graph.
\end{proof}

Propositions~\ref{stars} and~\ref{thickening} together with Proposition~\ref{carrier} conclude the proof of Theorem~\ref{mthm:stars}.

\subsection{Fixed cells}\label{subsec:fixedcells}
In this subsection we prove Theorem~\ref{mthm:fixed}. First, we follow  ideas of Tardif~\cite{Ta} to generalize fixed box theorems for median graphs to hypercellular graphs. We will prove that the fixed box in the case of hypercellular graphs is a cell. We obtain this cell verbatim  as in the case of median graphs.
Set
$$F(G):=\{W: W \textrm{ is an inclusion maximal proper halfspace of } G \textrm{ and } V(G)\setminus W \textrm{ is not}\}.$$
Let $Z(G)=\bigcap_{W\in F(G)} W$ (if $F(G)=\emptyset$, then set $Z(G):=G$). Now we recursively define $Z^{\infty}(G)$. Set $Z^{0}(G):=G$ and for every ordinal $\alpha$, let $Z^{\alpha+1}(G):=Z(Z^{\alpha}(G))$ if  $Z^{\alpha}(G)$ has been defined and let $Z^{\alpha}(G):=\cap_{\beta<\alpha}Z^{\beta}(G)$ if $\alpha$ is a limit ordinal. For every graph $G$ there exists a minimal ordinal $\gamma$ such that $Z^{\gamma}(G)=Z^{\gamma+1}(G)$. Finally, define $Z^{\infty}(G):=\cap_{\gamma} Z^{\gamma}(G)$.

\begin{lemma}\label{lem:fixed_cell}
Let $G$ be a hypercellular graph not containing infinite isometric rays. Then $Z^{\infty}(G)$ is a finite cell of $G$.
\end{lemma}

\begin{proof} First notice that since $G$ does not contain infinite isometric rays and all cells of $G$ are Cartesian products of cycles and edges, all cells of $G$ are finite.
Since $Z^{\infty}(G)$ is an intersection of convex subgraphs, $Z^{\infty}(G)$  is convex.
Since every cell $X$ of $G$ is a Cartesian product of edges and even cycles, any proper halfspace of $X$ is maximal by inclusion, hence $F(X)=\emptyset$.  Therefore, $Z(X)=X$. Suppose by way of contradiction that there exists a convex subgraph $H$ of $G$ which is not a cell and such that $Z(H)=H$. Since $H$ is convex, $H$ is hypercellular. Since there are no infinite isometric  rays, $Z(H)=H$ if and only if for every edge $ab$ of $H$ the halfspaces $W(a,b)$ and $W(b,a)$ are maximal, which is equivalent to  the condition that for every edge $ab$ the carrier $N(E_{ab})$ is the whole graph $H$. Let $X$ be the intersection of all maximal cells of $H$; consequently, $X$ is a cell of $H$. As in the proof of Theorem~\ref{thm:amalgam}, if there exist two disjoint maximal cells of $H$, then for every edge $f$ on a shortest path between them, the carrier $N(E_{f})$ is not the whole graph $H$. Hence the maximal cells of $H$ pairwise intersect. Since they are finite and gated, by the Helly property for gated sets, $X$ is nonempty. If $X\neq H$, by Claim~\ref{proper_carrier}, there exists an edge of $H$ whose carrier does not include all maximal cells,  a contradiction. Hence $H=X$, i.e., $H$ is a finite cell.  Consequently, $Z^{\infty}(G)$ is a finite cell of $G$.
\end{proof}

We continue with the proof of assertion (i) of Theorem~\ref{mthm:fixed}.

\begin{proposition}\label{automorphism_box}
If $G$ is  a hypercellular graph not containing infinite isometric rays, then  there exists a finite cell $X$ in $G$ fixed by every automorphism of $G$.
\end{proposition}

\begin{proof}
Every automorphism $\varphi$ of $G$ maps maximal halfspaces to maximal halfspaces, thus $\varphi(Z^{\infty}(G))=Z^{\infty}(G)$. By Lemma~\ref{lem:fixed_cell}, $Z^{\infty}(G)$ is a finite cell, thus every automorphism of $G$ fixes the cell $Z^{\infty}(G)$.
\end{proof}

A \emph{non-expansive map} from a graph $G$ to a graph $H$ is a map $f: V(G)\rightarrow V(H)$ such that for any $x,y\in V(G)$ it holds $d_H(f(x),f(y))\leq d_G(x,y)$.

\begin{lemma}\label{lem:fix_for_nonexp}
Let $G$ be a hypercellular graph and $f$ be a non-expansive map from $G$ to itself. Let $u,v,w$ be any three vertices of $G$ and let $X=\langle\langle u',v',w' \rangle\rangle $ be their median-cell. If $f(u)=u,f(v)=v,f(w)=w$, then $f$ fixes each of the apices $u'=(uvw),v'=(vwu),w'=(wuv)$ and $f(X)=X$.
\end{lemma}

\begin{proof}
Denote by $u'v'w'$ the unique quasi-median of the triplet $u,v,w$. The map $f$ fixes each of the vertices $u,v,w$ and maps shortest paths between them to shortest paths. This implies that $f$ fixes the vertices of the metric triangle $u'v'w'$. Let $X=\langle\langle u',v',w'\rangle\rangle$ be the gated cell induced by this triplet. Since $u'v'w'$ is a metric triangle, we conclude that $X\cong C_1\square \cdots \square C_k$, where each $C_i$ is an even cycle of length $n_i$ at least 6. Moreover, we can suppose without loss of generality that  $u',v',w'$ are embedded in this product as $u'=(0,0,\ldots,0)$, $v'=(i_1,i_2,\ldots,i_k)$, and $w'=(j_1,j_2,\ldots,j_k)$, where $i_m,j_m-i_m,n_m-j_m<n_m/2$ for all $1\leq m\leq k$. Since $f(u')=u',f(v')=v',f(w')=w'$ and any vertex of $X$ lies on a shortest path between one of the pairs of $u',v',w'$, we conclude that $f(X)\subset X$. It remains to prove that $f(X)=X$.

Without loss of generality assume that among all $j_m-i_m$ with $1\leq m\leq k$, the difference $j_1-i_1$ is minimal. The vertex $y=(i_1,0,0,\ldots,0)$ belongs to $I(u',v')$ and is located at  distance $j_1-i_1$ from $z=(j_1,0,0,\ldots, 0)\in I(w',u')$. The vertices of $I(u',v')$ at distance $i_1$ from $u'$ have the form $(i_1-y_1,y_2,\ldots,y_k)$, for $0\leq y_i \leq i_m, 1\leq m\leq k$, with $y_2+\ldots+ y_k=y_1$. On the other hand, the  vertices of  $I(w',u')$ at distance $n_1-j_1$ from $u'$  have the form $(j_1+z_1,n_2-z_2,\ldots,n_k-z_k)$ for $0\leq z_m \leq n_m-j_m, 1\leq m\leq k$ with $z_2+\ldots+ z_k=z_1$, where the $m$-th coordinate is computed in $\mathbb{Z}_{n_m}$. We will now find  all pairs $(y',z')$ where $y'\in I(u',v'), z'\in I(w',v')$ and $y'$ and $z'$ are at distance  $j_1-i_1$.

We distinguish two cases. On one hand assume that for a chosen  pair $y',z'$ there exists a coordinate $m, 1\leq m \leq k$ such that the projections of $(y',z')$-shortest paths to the $m$-th  coordinate belong to the interval between $i_m$ and $j_m$. Then the distance between $y'$ and $z'$ is at least $j_m-i_m$, and since $j_m-i_m\geq j_1-i_1$, we have $y'=(0,\ldots,0,i_m,0,\ldots, 0)$ and $z'=(0,\ldots,0,j_m,0,\ldots, 0)$ with $j_m-i_m=j_1-i_1$ and $i_m=i_1, n_m-j_m=n_1-j_1$. This implies $j_m=j_1$  and $n_m=n_1$. Assume now that $f$ maps $y,z$ to $y',z'$. There exists an automorphism $\varphi$ of $X$ that swaps coordinates $1$ and $m$ of $X$ and fixes $u',v',w'$. Since proving that $f(X)=X$ is the same as proving that $\varphi(f(X))=X$, we can, in this case, assume that $f$ fixes $y,z$.

On the other hand, if for a pair $y',z'$ and every coordinate $m, 1\leq m \leq k$, the projection of $(y',z')$-shortest paths to the $m$-th coordinate does not belong to the interval between $i_m$ and $j_m$, then the distance between $y'=(i_1-y_1,y_2,\ldots,y_k)$ and $z'=(j_1+z_1,n_2-z_2,\ldots,n_k-z_k)$ is $((i_1-y_1)+(n_1-j_1-z_1))+(y_2+z_2)+\ldots+(y_k+z_k)=n_1-(j_1-i_1)>n/2>j_1-i_1.$ Since this is impossible, we can by the previous paragraph assume that $f$ fixes $y$ and $z$. Then $f$ fixes $(0,0,\ldots,0),(i_1,0,\ldots,0),(j_1,0,\ldots,0)$, thus it must fix every $(x,0,\ldots,0)$, $0\leq x < n_1$.

Now we will prove that every cyclic layer of the form $(C_1,x_2,\ldots,x_k)$ is mapped by $f$ to  a cyclic layer of the form $(C_1,x_2',x_3',\ldots,x_k')$. We proceed by induction on $x_2+x_3+\ldots +x_k$. It holds for $(C_1,0,0,\ldots,0)$. Without loss of generality consider only $f(C_1,x_2+1,x_3,\ldots,x_k)$, assuming that  $f(C_1,x_2,\ldots,x_k)=(f_1(C_1),y_2,y_3,\ldots,y_k)$ for some automorphism $f_1$ of $C_1$.

For every $x_1 \in C_1$, the vertex $y=f(x_1,x_2+1,x_3,\ldots,x_k)$ must be equal or adjacent to $(f_1(x_1),y_2,y_3,\ldots, y_k)$, thus it must be of the form $(f_1(x)+s,y_2,y_3,\ldots, y_k)$ or $(f_1(x),y_2,\ldots,y_m+s,\ldots, y_k)$ for some $2\leq m \leq k$ and $s\in \{-1,0,1\}$. Now we analyze the options for $a=f(x+n_1/2-1,x_2+1,x_3,\ldots,x_k)$ and $b=f(x+n_1/2+1,x_2+1,x_3,\ldots,x_k)$. Both must be at distance at most two from each other, at distance at most $n_1/2-1$ from $y$, and $a$ must be adjacent or equal to $(f_1(x+n_1-1),y_2,y_3,\ldots, y_k)$ while $b$ must be adjacent or equal to $(f_1(x+n_1+1),y_2,y_3,\ldots, y_k)$. If $y=(f_1(x)+s,y_2,y_3,\ldots, y_k)$, this implies that $a=(f_1(x+n_1/2-1)+s,y_2,y_3,\ldots, y_k)$ and $b=(f_1(x+n_1/2+1)+s,y_2,y_3,\ldots, y_k)$. If $y=(f_1(x),y_2,\ldots,y_m+s,\ldots, y_k)$, then $a=(f_1(x+n_1/2-1),y_2,\ldots, y_m+s,\ldots, y_k)$ and $b=(f_1(x+n_1/2+1),y_2,\ldots,y_m+s,\ldots, y_k)$. In each case $y,a,b$ spans a cycle, since the length of $C_1$ is at least six and $f$ is a non-expansive map. Thus $f(C_1,x_2+1,x_3,\ldots,x_k)$ is a cycle of the form $(f_1(C)+s,y_2,y_3,\ldots, y_k)$ or $(f_1(C),y_2,\ldots,y_m+s,\ldots, y_k)$ for some $2\leq m \leq k$ and $s\in \{-1,0,1\}$, proving the assertion.

We have proved that $f$ acting on $X$ has blocks of imprimitivity of the form $(C_1,x_2,\ldots,x_k)$ and it holds $f(C_1,0,\ldots,0)=(C_1,0,\ldots,0)$, $f(C_1,i_2,\ldots,i_k)=(C_1,i_2,\ldots,i_k)$ and $f(C_1,j_2,\ldots,j_k)=(C_1,j_2,\ldots,j_k)$.
By the induction on the number of factors of $X$, $f$ acts as an automorphism on the quotient graph, thus $f$ acts as an automorphism on $X$.
\end{proof}

An endomorphism $r$ of $G$ with $r(G)=H$ and $r(v)=v$ for all vertices $v$ in  $H$ is called a \emph{retraction} of $G$ and  $H$ is called a \emph{retract} of $G$. Moreover, if $r$ is just a non-expansive map from $G$ to itself, the map is called a \emph{weak retraction} and $H$ a \emph{weak retract} of $G$.

\begin{corollary}
A weak retract $H$ of a hypercellular graph $G$ is a hypercellular graph.
\end{corollary}

\begin{proof}
Let $r$ be a weak retraction of $G$ to $H$. For arbitrary vertices $u,v,w$ of $H$ it holds $r(u)=u,r(v)=v,r(w)=w$, thus by Lemma~\ref{lem:fix_for_nonexp} $X=r(X)\subseteq r(G)=H$, where $X$ is the median cell of $u,v,w$. This proves that $H$ satisfies the median-cell property and by Theorem~\ref{mthm:median}, $H$ is hypercellular.
\end{proof}

We continue with the proof of assertion (ii) of Theorem~\ref{mthm:fixed}.

\begin{proposition}\label{nonexpansive}
Let $G$ be a hypercellular graph and let $f$ be a non-expansive map from $G$ to itself such that $f(S)=S$ for some finite set $S$ of vertices of $G$. Then there exists a finite cell $X$ of $G$ that is fixed by $f$. In particular, if $G$ is a finite hypercellular graph, then it has a fixed cell.
\end{proposition}

\begin{proof}
 Let $H$ be the subgraph of $G$ induced by the set of all vertices $v$ in $\conv(S)$ for which there exists an integer $n_v>0$ such that $f^{n_v}(v)=v$. Since $S$ is finite, also $\conv(S)$ is finite, therefore $H$ is finite and nonempty. Notice that $f(\conv(S))\subseteq \conv(S)$, thus $H\supseteq f(H) \supseteq f(f(H)) \supseteq \ldots$, but since for every $v\in V(H)$ there exists $n_v$ such that $f^{n_v}(v)=v$, the inclusions cannot be strict. Thus $f(H)=H$ and $f$ acts as an automorphism on $H$.

Let $u,v,w\in V(H)$ be arbitrary vertices of $H$. Let $n$ be the least common multiple of $n_u,n_v,n_w$. Then $f^n$ fixes each of the vertices $u,v,w$. By Lemma~\ref{lem:fix_for_nonexp}, $f^n(X)=X$ where $X$ is the median cell of $u,v,w$. Since $X$ is finite, this proves that also $X\subset H$. Therefore $H$ satisfies the  median-cell property and,  by Theorem~\ref{mthm:median}, $H$ is hypercellular. Applying Proposition~\ref{automorphism_box} to $H$, we deduce that there exists a finite fixed cell.
\end{proof}

The above proposition follows the ideas from~\cite{Ta}, but the main difficulty   is to prove Lemma~\ref{lem:fix_for_nonexp}. In the case of median graphs this lemma is not needed since $X$ is a single vertex. The next proposition uses ideas from~\cite{ImKl2} to generalize yet another classical result on median graphs and to prove assertion (iii) of Theorem~\ref{mthm:fixed}.

\begin{proposition}\label{regular}
If $G$ is a finite regular hypercellular graph, then $G$ is a single cell, i.e., $G$ is isomorphic to a Cartesian product of edges and even cycles.
\end{proposition}

\begin{proof}
Pick an arbitrary edge $ab$ in $G$ such that $W(a,b)$ is an inclusion minimal  halfspace. We will prove that also $W(b,a)$ is minimal. The carrier $N(E_{ab})$ is the union  of maximal cells of $G$ crossed by $E_{ab}$. For each such maximal cell $X$ of $N(E_{ab})$ there exists a unique automorphism of $X$ that fixes edges of $E_{ab}\cap X$ and maps $X\cap W(a,b)$ to $X\cap W(b,a)$ and vice versa. This automorphisms extends to an automorphism $\varphi$ of $N(E_{ab})$ that maps $N(E_{ab})\cap W(a,b)$ to $N(E_{ab})\cap W(b,a)$ and vice versa. For the sake of contradiction assume now that $W(b,a)$ is not minimal and that there exists an edge $cd$ with $c \in N(E_{ab})\cap W(b,a)$ and $d\notin N(E_{ab})$. The vertex $c'=\varphi(c) \in W(a,b)$ has the same degree in $N(E_{ab})$ as $c$. Since $G$ is regular, there must exist an edge $c'd'$ with $d'\notin N(E_{ab})$. Since $N(E_{ab})$ is gated, $E_{c'd'}$ does not cross the carrier $N(E_{ab})$, thus $W(d',c')\subsetneq W(a,b)$. This contradicts the choice of $ab$.

Since $G$ is finite, we have proved that for every edge $ab$, $W(a,b)$ and $W(b,a)$ are  minimal halfspaces. Thus for every edge $ab$, we have $N(E_{ab}) = G$. This implies that $Z(G)=G$, thus $Z^{\infty}(G)=G$. By Lemma~\ref{lem:fixed_cell}, $G$ is  a single cell.
\end{proof}

\section{Conclusions and open questions}\label{sec:concl}

In the present paper we have established a rich cell-structure for hypercellular graphs. In particular, we have obtained that they generalize bipartite cellular and median graphs in a natural way. On the other hand, we expect that other properties and characterizations of median graphs or cellular graphs extend naturally to hypercellular graphs. Some of those questions concern the metric structure of hypercellular graphs, while other questions ask for replacing metric conditions by topological or algebraic conditions. Namely, in all our results we characterized hypercellular graphs among partial cubes. One can ask to what extent we can characterize hypercellular graphs and their cell complexes among all graphs and complexes.

For example, as we noticed already, by a result of Gromov~\cite{Gr}, CAT(0) cube complexes are exactly the simply connected cube complexes satisfying the cube condition, i.e., the
cubical version  of the 3C-condition. As proved in~\cite{Ch_CAT}, median graphs are exactly the graphs whose associated cube complexes are CAT(0). In fact, it is shown in~\cite{Ch_CAT}
(see also~\cite{BrChChGoOs} for other similar results) that median graphs are exactly the graphs of square complexes which are simply connected and satisfy the square version of the
3CC-condition: any three squares pairwise intersecting in three edges and all three intersecting in a vertex belong to a 3-cube. We believe that a similar result holds for hypercellular graphs.
Namely, let ${\bf X}$ be a hyperprism complex, i.e., a polyhedral cell complex whose cells are Cartesian products of segments and regular polygons with an even number of sides and glued in a such a way that the intersection of any two cells is a cell. The 1-skeleton of ${\bf X}$ is the graph $G({\bf X})$ having the 0-cells of ${\bf X}$ as vertices and 1-cells as edges. Finally, we call
a cell complex of a hypercellular graph a {\it hypercellular complex}. We conjecture that hypercellular graphs can be characterized in the following way:

\begin{conjecture} \label{complex} For a graph $G$, the following conditions are equivalent:
\begin{itemize}
\item[(i)] $G$ is hypercellular;
\item[(ii)] $G$ is the 1-skeleton of a simply connected hyperprism complex ${\bf X}$ satisfying the 3C-condition;
\item[(iii)] $G$ is the 1-skeleton of a simply connected polygonal complex ${\bf X}$ (whose 2-cells are regular polygons
with an even number of sides) satisfying the 3CC-condition.
\end{itemize}
Moreover, all hypercellular cell complexes are CAT(0) spaces.
\end{conjecture}

Since CAT(0) spaces obey the fixed point property~\cite{BrHa}, the fact that hypercellular cell complexes are CAT(0) spaces would immediately imply Proposition~\ref{nonexpansive}.

Median graphs are exactly the discrete median algebras (for this and other related results see the paper~\cite{BaHe} and the survey~\cite{BaCh_survey}). In~\cite{BaCh_wma1}, the apex algebras of weakly median graphs have been characterized. The {\it apex algebra} of an apiculate graph $G$ associates to each triplet of vertices $u,v,w$,  the apex $(uvw)$ of $u$ with respect to $v$ and $w$.

\begin{problem} Characterize the apex ternary algebras of hypercellular graphs.
\end{problem}

In view of the fact that median graphs are precisely retracts of hypercubes~\cite{Ba_retracts}, we believe that the following is true:

\begin{conjecture} \label{retract} A partial cube $G$ is hypercellular if and only if $G$ is a (weak) retract of a Cartesian product of bipartite cellular graphs.
\end{conjecture}

Note that it is not true that hypercellular graphs are retracts of Cartesian products of their cells or, more generally, that they are exactly the retracts of Cartesian products of edges and even cycles. Indeed, let $H$ be the hypercellular 
graph which is the 6-cycle with an attached pendant edge. Then $H$ has $C_6$ and $K_2$ as cells.  It can be directly checked
that $H$ is not a retract of $C_6\square K_2$. Moreover, suppose that $H$ is an isometric subgraph of the Cartesian product $G:=F_1\square\cdots\square F_k$, where $F_1,\ldots,F_k$ are edges or even cycles. Then it can be shown that $H$ is included in a subproduct $G':=F'_1\square\cdots\square F'_k$ of $G$ which is either isomorphic to $C_6\square K_2$ or to $Q_4=K_2\square K_2\square K_2\square K_2$.  If $H$ is a retract of $G$, then necessarily $H$ is a retract of $G'$. But we noticed above that $H$ cannot be a retract of $C_6\square K_2$. Since $H$ is not a median graph, also $H$ cannot be a retract of the 4-cube $Q_4$. This shows that the hypercellular graph $H$ is not a retract of any Cartesian product of even cycles and edges.

A group $F$ \emph{acts by automorphisms} on a cell complex $\bf X$ if there is an injective
homomorphism $F\to \mbox{Aut}({\bf X})$ called an \emph{action of $F$}. The action is
\emph{geometric} (or \emph{$F$ acts geometrically}) if it is proper (i.e., cells
stabilizers are finite) and cocompact (i.e., the quotient ${\bf X}/F$ is compact). A group $F$ is called
a {\it Helly group}~\cite{ChChHiOs2} if $F$ acts geometrically on the clique complex of a Helly graph.
Analogously, we will say that a group $F$ is {\it hypercellular} if $F$ acts geometrically on a cell complex
${\bf X}(G)$ of a hypercellular graph $G$ (in this case, $G$ is locally-finite).  Analogously
to~\cite[Proposition 6.32]{ChChHiOs}
one can show that $F$ acts geometrically  on the clique complex of the thickening $G^{\Delta}$ of $G$. By Proposition
\ref{thickening},  $G^{\Delta}$ is a  Helly graph, thus any hypercellular group is a Helly group. Since all Helly groups
are biautomatic~\cite{ChChHiOs2}, we obtain the following corollary:

\begin{corollary} Any hypercellular group is a Helly group and thus is biautomatic.
\end{corollary}

In Theorem~\ref{mthm:COM} we have shown that finite hypercellular graphs are tope graphs of zonotopal COMs. In~\cite{BaChKn} the question is raised whether zonotopal COMs are fibers of realizable COMs. In our case this specializes to solving the following:
\begin{problem}
 Is every finite hypercellular graph a convex subgraph of the tope graph of a realizable oriented matroid?
\end{problem}

In the first part of the paper we have obtained a few results for $\mathcal{S}_4$ similar to those for hypercellular graphs.
%
We believe, that it is possible to use analogous amalgamation techniques as we did for Theorem~\ref{mthm:amalgam} and Theorem~\ref{mthm:COM} in order to prove:
\begin{conjecture}
 Every finite graph in $\mathcal{S}_4$ is the tope graph of a COM.
\end{conjecture}

\bibliographystyle{amsalpha}

\end{document}